\documentclass[reqno]{amsart}

\usepackage{amssymb}
\usepackage{amsmath}
\usepackage{amsthm}
\usepackage{pgfplots} 
\usepackage{graphicx}
\usepackage[utf8x]{inputenc}
\usepackage{tikz-cd}
\usepackage{bbm}
\usepackage{bm}
\usepackage{subcaption}
\usepackage[boxruled]{algorithm2e}
\usepackage{mathtools}
\usepackage{lipsum}
\usepackage{tabularx}
\usepackage[title,titletoc]{appendix}
\usepackage{booktabs}
\usepackage{calc}
\usepackage{array}
\usepackage{tikz}
\usepackage{boites}
\usepackage{here}
\usepackage[]{hyperref}
\usepackage{natbib,enumerate}
\usepackage{enumitem}
\usepackage{cleveref}
\usepackage{adjustbox}
\setlist[itemize]{itemsep=2pt, topsep=2pt, leftmargin=*}
\setlist[enumerate]{itemsep=2pt, topsep=2pt, leftmargin =*}
\usetikzlibrary{arrows}
\usepackage{extarrows}
\usepackage{ stmaryrd }
\usepackage{ dsfont }
\usepackage{ longtable }
\usepackage{comment}
\usepackage{nameref}
\usepackage{hyperref}
\pgfplotsset{compat=1.17} 

\makeatletter
\@namedef{subjclassname@2020}{\textup{2020} Mathematics Subject Classification}
\makeatother

\usepackage{etoolbox}
\usepackage[margin=0.75in,marginparwidth=0.65in]{geometry}

\definecolor{BrickRed}{rgb}{0.6, 0.1875, 0.2475}
\definecolor{Cerulean}{rgb}{0.0, 0.36, 0.4875}
\definecolor{MyGreen}{rgb}{0.1946, 0.3887, 0.0551}

\newcounter{introcounter}

\newtheoremstyle{introtheorem}
  {1\baselineskip\@plus.2\baselineskip\@minus.2\baselineskip}
  {1\baselineskip\@plus.2\baselineskip\@minus.2\baselineskip}
  {\slshape}
  {}
  {\bfseries}
  {.}
  { }
  {}

\theoremstyle{introtheorem}
\newtheorem{introthm}[introcounter]{Theorem}

\theoremstyle{plain}
\newtheorem{thm}{Theorem}[section]
\newtheorem{Prop}[thm]{Proposition}
\newtheorem{Lemma}[thm]{Lemma}

\theoremstyle{definition}
\newtheorem{Def}[thm]{Definition}
\newtheorem{Construction}[thm]{Construction}
\newtheorem{Notation}[thm]{Notation}

\theoremstyle{remark}
\newtheorem{Ex}[thm]{Example}
\newtheorem{Rmk}[thm]{Remark}
\newtheorem{Question}[thm]{Question}

\crefname{introcounter}{Theorem}{Theorems}
\crefname{Lemma}{Lemma}{Lemmas}
\crefname{Def}{Definition}{Definitions}
\crefname{Question}{Question}{Questions}
\crefname{Prop}{Proposition}{Propositions}
\crefname{thm}{Theorem}{Theorems}
\crefname{Construction}{Construction}{Constructions}
\crefname{Question}{Question}{Questions}
\crefname{Rmk}{Remark}{Remarks}
\crefname{Notation}{Notation}{Notations}
\crefname{section}{Section}{Sections}
\crefname{subsection}{Subsection}{Subsections}
\crefname{Ex}{Example}{Examples}

\makeatletter
\newlength{\temp@boite}
\newlength{\saveparindent}
\def\encadrement#1#2{%
  \def\bkvz@before@breakbox{\ifhmode\par\fi\vskip5pt\vskip\breakboxskip\relax}%
  \fboxrule=0.4pt
  \fboxsep=5pt
  \def\bkvz@set@linewidth{\advance\linewidth -2\fboxrule
                          \advance\linewidth -2\fboxsep}%
  \def\bkvz@left{\color{#2}\vrule \@width\fboxrule\hskip\fboxsep\color{black}}%
  \def\bkvz@right{\color{#2}\hskip\fboxsep\vrule \@width\fboxrule\color{black}}%
  \def\bkvz@top{\hbox to \hsize{%
      \color{#2}%
      \setlength{\temp@boite}{\fboxrule+0.5ex}%
      \vrule\@width\fboxrule\@height \temp@boite %
      \rule[0.5ex]{2em}{\fboxrule}%
      {#1}%
      \setlength{\temp@boite}{\textwidth-2em-\widthof{#1}-2\fboxrule}%
      \rule[0.5ex]{\temp@boite}{\fboxrule}%
      \setlength{\temp@boite}{\fboxrule+0.5ex}%
      \vrule\@width\fboxrule\@height \temp@boite}}%
  \def\bkvz@bottom{\color{#2}\hrule\@height\fboxrule}%
  \breakbox\vspace{0pt}}

\makeatother

\makeatletter
\def\encadrementombre#1#2{%
  \fboxsep=5pt
  \fboxrule=0.4pt
  \def\bkvz@before@breakbox{%
    \ifhmode\par\fi\vskip5pt\vskip\breakboxskip\relax}%
  \def\bkvz@set@linewidth{%
    \advance\linewidth -2\fboxrule 
    \setlength{\fboxsep}{5pt}\advance\linewidth -2\fboxsep
  }%
  \def\bk@line{\hbox to \linewidth{%
      \ifbkcount\smash{\llap{\the\bk@lcnt\ }}\fi
      \setlength{\fboxsep}{0pt}\colorbox{#2!15}{%
        \setlength{\fboxsep}{5pt}%
        {\color{#2}\vrule width \fboxrule}\hskip\fboxsep
        \box\bk@bxa
        \hskip\fboxsep{\color{#2}\vrule width\fboxrule}%
        }%
      }}%
  \def\bkvz@top{\hbox to \hsize{%
      \setlength{\temp@boite}{\fboxrule+0.5ex}%
      \color{#2!15}\vrule\@width\textwidth\@height\temp@boite %
      \hspace{-\textwidth}%
      \color{#2}%
      \vrule\@width\fboxrule\@height \temp@boite %
      \rule[0.5ex]{2em}{\fboxrule}%
      {\fboxsep=2pt\colorbox{white}{#1}}%
      \setlength{\temp@boite}{\textwidth-2em-\widthof{\fboxsep=2pt\colorbox{white}{#1}}-2\fboxrule}%
      \rule[0.5ex]{\temp@boite}{\fboxrule}%
      \setlength{\temp@boite}{\fboxrule+0.5ex}%
      \vrule\@width\fboxrule\@height \temp@boite %
}}%
  \def\bkvz@bottom{{\color{teal}\hrule\@height\fboxrule}}%
  \color{black}\breakbox}%

\makeatother

\makeatletter
\def\encadrementombrevar#1#2{%
  \fboxsep=5pt
  \fboxrule=0.4pt
  \def\bkvz@before@breakbox{%
    \ifhmode\par\fi\vskip5pt\vskip\breakboxskip\relax}%
  \def\bkvz@set@linewidth{%
    \advance\linewidth -2\fboxrule 
    \setlength{\fboxsep}{5pt}\advance\linewidth -2\fboxsep
  }%
  \def\bk@line{\hbox to \linewidth{%
      \ifbkcount\smash{\llap{\the\bk@lcnt\ }}\fi
      \setlength{\fboxsep}{0pt}\colorbox{#2!15}{%
        \setlength{\fboxsep}{5pt}%
        {\color{#2}\vrule width \fboxrule}\hskip\fboxsep
        \box\bk@bxa
        \hskip\fboxsep{\color{#2}\vrule width\fboxrule}%
        }%
      }}%
  \def\bkvz@top{\hbox to \hsize{%
      \setlength{\temp@boite}{\fboxrule+0.5ex}%
      \color{#2!15}\vrule\@width\textwidth\@height\temp@boite %
      \hspace{-\textwidth}%
      \color{#2}%
      \vrule\@width\fboxrule\@height \temp@boite %
      \rule[0.5ex]{2em}{\fboxrule}%
      {\fboxsep=2pt\fcolorbox{#2}{white}{#1}}%
      \setlength{\temp@boite}{\textwidth-2em-\widthof{\fboxsep=2pt\fcolorbox{#2}{#2!15}{#1}}-2\fboxrule}%
      \rule[0.5ex]{\temp@boite}{\fboxrule}%
      \setlength{\temp@boite}{\fboxrule+0.5ex}%
      \vrule\@width\fboxrule\@height \temp@boite %
}}%
  \def\bkvz@bottom{{\color{#2}\hrule\@height\fboxrule}}%
  \color{black}\breakbox}%

\makeatother

\newcommand{\N}{\mathbb{N}}
\newcommand{\Z}{\mathbb{Z}}
\newcommand{\R}{\mathbb{R}}
\newcommand{\C}{\mathbb{C}}

\newcommand{\id}[1]{\ensuremath{\mathsf{id}_{#1}}}
\newcommand{\op}[0]{\ensuremath{\mathsf{op}}}
\newcommand{\colim}[0]{\ensuremath{\mathsf{colim}}}
\newcommand{\limit}[0]{\ensuremath{\mathsf{lim}}}

\newcommand{\E}[1]{\ensuremath{\mathcal{E}_{#1}}}
\newcommand{\Einfty}[0]{\ensuremath{\mathcal{E}_{\infty}}}
\newcommand{\all}[0]{\ensuremath{\mathsf{all}}}
\newcommand{\finstar}[1]{\ensuremath{ \langle #1 \rangle }}
\newcommand{\fin}[1]{\ensuremath{ \langle #1 \rangle ^\circ}}
\newcommand{\fold}[0]{\ensuremath{\mathsf{fold}}}
\newcommand{\factor}[2]{\left. \raise 2pt\hbox{\ensuremath{#1}} \right/
 	  \hskip -2pt\raise -2pt\hbox{\ensuremath{#2}}}

\newcommand{\Affl}{{\mathbb{A}^1} }
\newcommand{\Gm}{\mathbb{G}_m}
\newcommand{\Sph}{\mathbb{S}}
\newcommand{\calS}{\mathcal{S}}
\newcommand{\Spec}[0]{ \mathsf{Spec} }

\newcommand{\Lmot}[0]{\ensuremath{\mathsf{L}_\mathsf{mot}}}

\newcommand{\Smash}{\ensuremath{\wedge}}
\newcommand{\ko}[0]{\ensuremath{\mathsf{ko}}}

\newcommand{\MSL}[0]{\ensuremath{\mathsf{MSL}}}
\newcommand{\MO}[0]{\ensuremath{\mathsf{MO}}}
\newcommand{\MU}[0]{\ensuremath{\mathsf{MU}}}
\newcommand{\MSp}[0]{\ensuremath{\mathsf{MSp}}}
\newcommand{\MGL}[0]{\ensuremath{\mathsf{MGL}}}
\newcommand{\MSO}[0]{\ensuremath{\mathsf{MSO}}}

\newcommand{\Mtop}[0]{\ensuremath{M_\mathsf{top}}}

\newcommand{\rR}[0]{\ensuremath{r_{\R}}}
\newcommand{\rC}[0]{\ensuremath{r_{\mathbb{C}}}}
\newcommand{\HZ}[0]{\ensuremath{\mathsf{H}\Z}}

\newcommand{\calR}[0]{\ensuremath{\mathcal{R}}}
\newcommand{\calJ}[0]{\ensuremath{\mathcal{J}}}
\newcommand{\calC}[0]{\ensuremath{\mathcal{C}}}
\newcommand{\calD}[0]{\ensuremath{\mathcal{D}}}
\newcommand{\calF}[0]{\ensuremath{\mathcal{F}}}
\newcommand{\calG}[0]{\ensuremath{\mathcal{G}}}

\newcommand{\calE}[0]{\ensuremath{\mathcal{E}}}

\newcommand{\Sp}{\mathsf{Sp}}
\newcommand{\Fun}[0]{\ensuremath{\mathsf{Fun}}}
\newcommand{\map}[0]{\ensuremath{\mathsf{Map}}}
\newcommand{\Sm}{\ensuremath{\mathsf{Sm}}}
\newcommand{\Spc}[0]{\ensuremath{ \mathsf{Spc}}}
\newcommand{\spaces}[0]{\Spc}
\newcommand{\Pre}{\ensuremath{\mathcal{P}}}
\newcommand{\SHR}[0]{\ensuremath{\mathsf{SH}(\R)}}
\newcommand{\SH}[0]{\ensuremath{\mathsf{SH}}}

\newcommand{\Fin}[0]{\ensuremath{\mathsf{Fin}}}
\newcommand{\Finstar}[0]{\ensuremath{\mathsf{Fin}_\ast}}
\newcommand{\PrL}[0]{\ensuremath{\mathsf{Pr^L}}}

\newcommand{\CAlg}[0]{\ensuremath{\mathsf{CAlg}}}
\newcommand{\Alg}[0]{\ensuremath{\mathsf{Alg}}}

\newcommand{\Span}[0]{\ensuremath{\mathsf{Span}}}
\newcommand{\Cat}[0]{\ensuremath{\mathsf{Cat}_\infty}}
\newcommand{\wgpd}[0]{\ensuremath{{\omega,\simeq}}}
\newcommand{\kgpd}[0]{\ensuremath{{\kappa,\simeq}}}

\newcommand{\longeq}[0]{\ensuremath{\xlongequal{\phantom{\to}}}}

\newcommand{\lsimeq}[1]{\ensuremath{\xrightarrow[\ #1\ ]{\raisebox{-0.4ex}{\tiny $\simeq$}}}}

\tikzset{
    weq/.style={anchor=south, rotate=90, inner sep=.5mm}
}

\usetikzlibrary{calc}
\usetikzlibrary{decorations.pathmorphing}

\tikzset{curve/.style={settings={#1},to path={(\tikztostart)
    .. controls ($(\tikztostart)!\pv{pos}!(\tikztotarget)!\pv{height}!270:(\tikztotarget)$)
    and ($(\tikztostart)!1-\pv{pos}!(\tikztotarget)!\pv{height}!270:(\tikztotarget)$)
    .. (\tikztotarget)\tikztonodes}},
    settings/.code={\tikzset{quiver/.cd,#1}
        \def\pv##1{\pgfkeysvalueof{/tikz/quiver/##1}}},
    quiver/.cd,pos/.initial=0.35,height/.initial=0}

\tikzset{tail reversed/.code={\pgfsetarrowsstart{tikzcd to}}}
\tikzset{2tail/.code={\pgfsetarrowsstart{Implies[reversed]}}}
\tikzset{2tail reversed/.code={\pgfsetarrowsstart{Implies}}}
\tikzset{no body/.style={/tikz/dash pattern=on 0 off 1mm}}

\usepackage{subfiles}

\makeatletter
\def\@tocline#1#2#3#4#5#6#7{\relax
  \ifnum #1>\c@tocdepth 
  \else
    \par \addpenalty\@secpenalty\addvspace{#2}%
    \begingroup \hyphenpenalty\@M
    \@ifempty{#4}{%
      \@tempdima\csname r@tocindent\number#1\endcsname\relax
    }{%
      \@tempdima#4\relax
    }%
    \parindent\z@ \leftskip#3\relax \advance\leftskip\@tempdima\relax
    \rightskip\@pnumwidth plus4em \parfillskip-\@pnumwidth
    #5\leavevmode\hskip-\@tempdima
      \ifcase #1
       \or\or \hskip 2.5em \or \hskip 5em \else \hskip 7.5em \fi%
      #6\nobreak\relax
    \dotfill\hbox to\@pnumwidth{\@tocpagenum{#7}}\par
    \nobreak
    \endgroup
  \fi}
\makeatother

\begin{document}

\title[The real and complex Betti realizations of motivic multiplicative Thom~spectra]{The real and complex Betti realizations \\ of motivic multiplicative Thom~spectra}
\date{\today}

\author{Julie Bannwart}
\address{Institut für Mathematik, JGU Mainz, Germany}
\email{\href{bannwart.julie@gmail.com}{bannwart.julie@gmail.com}}
\subjclass[2020]{Primary: 14F42, Secondary: 55P43, 55N22}

\begin{abstract}
    Real Betti realization is a symmetric monoidal functor from the category of motivic spectra over $\R$ to that of topological spectra, extending the functor that associates to a scheme over $\R$ the space of its real points. In this article, we show that the motivic Thom spectrum functor and the topological one correspond to each other, as symmetric monoidal functors of $\infty$-categories, under real (and complex) realization. More precisely, we compare the constructions of Bachmann--Hoyois (in the motivic setting) and Ando--Blumberg--Gepner (in the topological setting). In particular, we obtain equivalences of $\Einfty$-rings between the real realizations of the variants $\MGL$, $\MSL$, and $\MSp$ of algebraic cobordism, and the variants $\MO$, $\MSO$, and $\MU$ of topological cobordism, respectively, and an analogous result for complex realization.
\end{abstract}
\keywords{Real Betti realization, complex Betti realization, multiplicative topological Thom spectra, multiplicative motivic Thom spectra, motivic colimit functor, $\infty$-categories and $\Einfty$-structures.}

\maketitle
\tableofcontents


\pagenumbering{arabic}

\newpage
\section*{Introduction}\label{Sect:introduction}

Classical cobordism spectra such as $\MO$, $\MSO$, $\MU$ admit geometric descriptions as colimits of (infinite suspensions of) Thom spaces of certain bundles over Grassmannians (\cite[\S III, Ex.\ 2.4]{Adams-blue-book}). Modern treatments reinterpret these constructions using the symmetric monoidal Thom spectrum $\infty$-categorical functor $\Mtop: \spaces_{/\Sp^\simeq} \to \Sp$ (see \cite[Thm.\ 1.6]{ABG}). On objects, the functor $\Mtop$ is given by taking the colimit of the $X$-shaped diagram in $\Sp$ represented by the morphism $X \to \Sp^\simeq$, that is, $\Mtop(F: X \to \Sp^\simeq) \simeq \colim_{X} F \in \Sp$.\\

Since $\Mtop$ is symmetric monoidal, an $\E{n}$-structure on $F$ induces an $\E{n}$-structure on the associated Thom spectrum. The spectrum $\MO$ and its classical $\Einfty$-structure (see \cite[\S IV.2]{May-Einfty}) can be recovered as the image under $\Mtop$ of an $\Einfty$-map $BO \to \mathsf{Pic}(\Sp) \hookrightarrow \Sp^\simeq$, called the $j$-homomorphism (see \Cref{Ex:MSO}). This functor recovers many other classical spectra with higher multiplicative structures, for example $\HZ_{(2)}$ as an $\E{2}$-ring (\cite{Mahowald}, also see \cite[Thm.\ 5.7]{ACB}). More generally, the (symmetric monoidal) Thom spectrum functor can be used to produce $\Einfty$- or $\E{n}$-structures on certain motivic spectra; perhaps more importantly, it can also be used to produce multiplicative maps between such ring spectra. For example, using this formalism, one can produce an $\E{2}$-map $\HZ_{(2)} \to \MSO_{(2)}$, showing in particular that $\MSO_{(2)}$ is an $\HZ_{(2)}$-module (see \cite[Cor.\ 3.7]{HLN}).\\

While $\MO$ and its variants represent cobordism cohomology theories for spaces, one may ask about the algebraic analogs. Indeed, one can define similar-looking cohomology theories for schemes, yielding various flavors of algebraic cobordism. Analogously to the way spectra represent stable cohomology theories for spaces, Morel and Voevodsky define a category of motivic spectra $\SH(k)$ over a base field $k$, whose objects represent certain cohomology theories for smooth schemes over $k$ (those satisfying $\Affl$-invariance and Nisnevich descent). Those are basic objects of motivic homotopy theory, where one works with the ($\infty$-)categories of motivic spaces $\Spc(k)$ and motivic spectra $\SH(k)$.\\

The original geometric construction of the cobordism spectra has a motivic counterpart. Most notably, the algebraic cobordism spectrum $\MGL$ \cite{Voevodsky-ICM}, or its variant $\MSL$, can similarly be described as a colimit of (infinite suspensions of) Thom spaces of algebraic vector bundles over certain Grassmannians (see \cite[\S 4]{BH-eta} for a modern description in these terms). A symmetric monoidal motivic Thom spectrum functor was also constructed in \cite[\S 16.3]{BH-norms}: it is an $\infty$-categorical functor $M: \Pre_\Sigma(\Sm_k)_{/\SH^\simeq} \to \SH(k)$, given on objects by a similar colimit formula as in the topological setting, namely $M(\calG\to \SH^\simeq) \simeq \colim_{(x,X) \in (\Sm_\R)_{\sslash\mathcal{G}}} (p_X)_\sharp \gamma(x)$, as we will recall in \Cref{Subsect:motThom}. Again, this formalism agrees with the geometric construction, and $\MGL$ or $\MSL$ can be recovered as the motivic Thom spectra of a motivic $\Einfty$-map $K^\circ \to \SH^\simeq$ (see \Cref{Ex:MGLMSLMSp}).\\

The purpose of this article is to compare these two multiplicative Thom spectrum constructions under real Betti realization. The latter is a symmetric monoidal functor $\rR : \SH(\R) \to \Sp$ between the $\infty$-category of motivic spectra over $\R$ and that of spectra (see \Cref{Subsect:rR}). It extends the functor of real points, sending a scheme to its space of real points, and connects the motivic and topological pictures. Computations of real realizations of motivic spectra have been carried out in a number of cases and have also been used to obtain results in the motivic category (see e.g.\ \cite[2.3, 2.9, 2.13]{ARO}, \cite[Thm.\ 5.5]{Levine-comparison} for $\HZ$, \cite[3.9, 3.35, 4.7]{BH-eta}, \cite[5.15, 5.16]{NZ}, \cite[\S 7]{NRZ}, etc.). We will also compare the Thom spectrum constructions under complex Betti realization, which is defined similarly. \\

Since real realization preserves colimits, the geometric description of $\MGL$ in terms of Grassmannians in the motivic setting recovers the topological one for $\MO$ under real realization, giving the equivalence $\rR(\MGL) \simeq \MO$ at the level of spectra \cite[Cor.\ 4.7]{BH-eta}. The compatibility of their respective $\Einfty$-structures is more subtle. Our main result shows that real (or complex) realization is indeed compatible with the motivic and topological multiplicative Thom spectrum functors. This requires a precise comparison of the constructions of their symmetric monoidal structures.

\begin{introthm}[\protect{\Cref{Prop:Thomagreeassmfunctorsfinal,Prop:MandMtopagreeassmII}}]\label{Prop:introthmThom} Let $k= \R$ or $\C$. The Betti realization $r_k$ of the motivic multiplicative Thom spectrum functor $M : \Pre_\Sigma(\Sm_k)_{/\SH^\wgpd} \to \SH$ (as constructed in \cite[\S 16.3]{BH-norms}) is equivalent as a symmetric monoidal functor to the topological multiplicative Thom spectrum functor $\Mtop : \spaces_{/\Sp^\kgpd} \to \Sp$ (as constructed in \cite[Thm.\ 1.6]{ABG}). More precisely, there is a commutative diagram of symmetric monoidal functors
\[\begin{tikzcd}[row sep = 2em, column sep = 2.5em]
   {\Pre_\Sigma(\Sm_k)_{/\SH^\wgpd}} & {\spaces_{/\Sp^\kgpd}} \\
	{\SH(k)} & \Sp
	\arrow[from=1-1, to=1-2]
	\arrow["M"', from=1-1, to=2-1]
	\arrow["\Mtop", from=1-2, to=2-2]
	\arrow["{r_k}"', from=2-1, to=2-2]
    \end{tikzcd}\]
where the top horizontal arrow is also induced by $r_k$ (see Diagram (\ref{Diag:Mtop'}) in \Cref{Prop:Mtop'}).
\end{introthm}

As an application, we compute the realizations of several variants of algebraic cobordism spectra as $\Einfty$-rings.

\begin{introthm}[\protect{\Cref{Prop:MGLMSLMSp,Prop:MGLMSLMSpcomplex}}]\label{Prop:introthmMSL} There are equivalences of $\Einfty$-rings
    $$\begin{matrix}
        \rR\MGL\simeq \MO,\qquad &\rR\MSL \simeq \MSO,\qquad  &\rR\mathsf{MSp}\simeq \MU,\phantom{^\mathsf{top}.} \\
    \rC\MGL \simeq \MU,\qquad  &\rC\MSL \simeq \mathsf{MSU},\qquad  &\rC\mathsf{MSp}\simeq \mathsf{MSp}^\mathsf{top}.
    \end{matrix}$$
\end{introthm}

To prove \Cref{Prop:introthmThom}, we take a detour through the more general motivic colimit functors \cite{BEH}, where $\SH$ is replaced by a more general presheaf $\calF$. We establish a naturality statement for this construction with respect to the presheaf $\calF$, and use it to compare the motivic Thom functor $M$ to a motivic colimit functor more closely related to real realization, taking values in $\Sp$. Finally, we show that this new motivic colimit functor factors through the usual domain $\Spc_{/\Sp^\kgpd}$ of $\Mtop$, and that the functor obtained in this way is indeed equivalent to $\Mtop$.\\

\Cref{Prop:introthmMSL} is then a rather direct consequence of \Cref{Prop:introthmThom}. This corollary is used in another article of the author to compute the real Betti realization of $\ko$, the very effective cover of the Hermitian $K$-theory motivic spectrum, as an $\E{1}$-ring. Indeed, the equivalence of $\Einfty$-rings $\rR(\MSL)\simeq\MSO$ is used to produce an $\E{2}$-map $\HZ_{(2)} \to (\rR\ko)_{(2)}$ in a similar vein to the example of the map $\HZ_{(2)} \to \MSO_{(2)}$ above. \\

\textbf{Organization of the article.} \Cref{Sect:preliminaries,Sect:Thom} are dedicated to establishing our context and some tools; they recall preexisting results from the literature. \Cref{Sect:comp} contains our main contribution.

More precisely, \Cref{Sect:preliminaries} is devoted to some $\infty$-categorical prerequisites about symmetric monoidal structures on categories of presheaves and slice categories. Then, \Cref{Sect:Thom} recalls the constructions of symmetric monoidal structures on the topological and motivic Thom spectrum functors. In \Cref{Sect:comp}, we prove our main results \Cref{Prop:introthmThom,Prop:introthmMSL}. Finally, \Cref{Subsect:rCThom} establishes the analogous results in the case of complex realization. The proof is similar in spirit to that of the real case, but a few modifications are needed. The argument also applies to the real realization functor, but it is slightly more involved than the first proof we provided, hence the decision to include it separately. The reader may refer to the \nameref{Sect:indexofnotation} at the very end for any unexplained notation.\\

\textbf{Conventions and prerequisites}\hfill
\begin{itemize}
    \item More introductory material for this article can be found in \cite{Mastersthesis}.
    \item We use the language of $\infty$-categories following \cite{Lurie-HTT} and \cite{Lurie-HA}. For us, ``category'' will always mean ``$\infty$-category'' (and we will specify ``1-category'' if needed). We denote by $\Cat$ the $\infty$-category of small $\infty$-categories. This setup is crucial to our work; indeed the main goal is to compare higher multiplicative structures built using this language.
    \item Let $\Spc$ be the category of spaces and $\Sp$ be the category of spectra.
    \item If $S$ is a qcqs base scheme, we denote by $\Sm_S$ the 1-category of smooth schemes of finite type over $S$. Let $\Spc(S)$ and $\SH(S)$ be the categories of motivic spaces and motivic spectra, respectively, over $S$ (constructed in the setting of model categories in \cite{MV}, and in that of $\infty$-categories in \cite{Robalo} for example). If $X \in \Sm_S$ and $Y\in \Spc$, we will abuse notation and denote the corresponding motivic spaces or motivic spectra also by $X$ and $Y$.
   \item Let $\rR : \SH(\R) \to \Sp$ be the real realization functor (see \Cref{Subsect:rR}). We will also denote the corresponding functors $\Pre(\Sm_\R) \to \Spc$ and $\Spc(\R) \to \Spc$, or their pointed analogs, by $\rR$. 
\end{itemize}

\vspace{0.4em}
\section{Some \texorpdfstring{$\infty$}{infinity}-categorical prerequisites}\label{Sect:preliminaries}

This introductory section contains preliminaries about certain $\infty$-categorical constructions needed to define and work with the symmetric monoidal structures we will consider on the topological and motivic Thom spectrum functors. More precisely, we will study Day convolution on categories of presheaves (\Cref{Subsect:dayconvolution}), especially its universal property, and symmetric monoidal structures on slice categories (\Cref{Subsect:slices}), since the source of the Thom spectrum functors we will consider are such categories, as discussed in the \nameref{Sect:introduction}.

\subsection{Day convolution on \texorpdfstring{$\infty$}{infinity}-categories of presheaves}\label{Subsect:dayconvolution}

The category of functors between two symmetric monoidal categories can be endowed with a symmetric monoidal structure, called Day convolution.

\begin{thm}[\protect{\cite[Prop.\ 2.11]{Glasman}} and \protect{\cite[Ex. 2.2.6.9]{Lurie-HA}}]\label{Prop:maindayconvolution} Let $\calC^\otimes$ and $\calD^\otimes$ be symmetric monoidal categories. Assume that $\calD$ has all small colimits and that the tensor product on $\calD$ preserves them in each variable separately. Then the category $\Fun(\calC,\calD)$ admits a symmetric monoidal structure $\Fun(\calC,\calD)^\otimes$, called \emph{Day convolution}. It has in particular the property that there is an equivalence of categories $\CAlg(\Fun(\calC,\calD)^\otimes) \simeq \Fun^\mathsf{lax}(\calC^\otimes,\calD^\otimes)$.
\end{thm}

For any symmetric monoidal category $\calC^\otimes$, in the particular case $\calD^\otimes = \spaces^\times$, Day convolution endows the category of presheaves $\Pre(\calC)$ with a symmetric monoidal structure (because the opposite of $\calC$ has an induced symmetric monoidal structure \cite[Rmk.\ 2.4.2.7]{Lurie-HA}). This symmetric monoidal structure has the following universal property:
\begin{Prop}[\protect{\cite[Cor.\ 4.8.1.12 and Rmk.\ 4.8.1.13]{Lurie-HA}}]\label{Prop:dayconvolutionpsh}
Let $\calC^\otimes$ be a symmetric monoidal category. Day convolution is the essentially unique symmetric monoidal structure on the category $\Pre(\calC)$ such that:
\begin{enumerate}
    \item the tensor product preserves colimits in both variables separately,
    \item and the Yoneda embedding $y: \calC \to \Pre(\calC)$ can be extended into a symmetric monoidal functor $y^\otimes$.
\end{enumerate}  
\end{Prop}

\begin{Prop}\label{Prop:laxdayconvolution} Let $\calC^\otimes$ be a symmetric monoidal category, and consider $\Pre(\calC)^\otimes$ with the Day convolution symmetric monoidal structure. Let $\calD^\otimes \in\CAlg(\PrL)$. Then, the symmetric monoidal Yoneda embedding $y^\otimes: \calC^\otimes \to \Pre(\calC)^\otimes$ induces an equivalence
    $- \circ y^\otimes: \Fun^{\mathsf{L,lax}}(\Pre(\calC)^\otimes,\calD^\otimes) \xrightarrow{\simeq} \Fun^{\mathsf{lax}}(\calC^\otimes,\calD^\otimes)$,
    where $\Fun^{\mathsf{lax}}$ denotes lax symmetric monoidal  functors, and $\Fun^{\mathsf{L,lax}}$ those preserving colimits.
\end{Prop}

\begin{Def}\label{Def:smLKE} In the setting of \Cref{Prop:laxdayconvolution}, let $\mathsf{LKE}^\otimes$ be an inverse for the equivalence $- \circ y^\otimes$. Given $F\in \Fun^{\mathsf{lax}}(\calC^\otimes,\calD^\otimes)$, we say that the lax symmetric monoidal functor $\mathsf{LKE}^\otimes(F) : \Pre(\calC)^\otimes \to \calD^\otimes$ is \emph{left Kan extended as a lax symmetric monoidal functor} from $F$.
\end{Def}

\begin{proof}[Proof of \Cref{Prop:laxdayconvolution}]
    Day convolution makes $\Fun(\Pre(\calC),\calD)$ and $\Fun(\calC,\calD)$ into symmetric monoidal categories $\Fun(\Pre(\calC),\calD)^\otimes$ and $\Fun(\calC,\calD)^\otimes$, respectively. By \cite[Cor.\ 3.8]{Nikolaus}, precomposition by a lax symmetric monoidal functor induces a lax symmetric monoidal functor with respect to the Day convolution on the functor categories. In our case, this implies that precomposition with the symmetric monoidal Yoneda embedding induces a lax symmetric monoidal functor $\Fun(\Pre(\calC),\calD)^\otimes \to \Fun(\calC,\calD)^\otimes$. By the usual universal property of the presheaf category (free cocompletion), the underlying functor restricts to an equivalence $\Fun^\mathsf{L}(\Pre(\calC),\calD) \to \Fun(\calC,\calD)$. In particular, $\Fun^\mathsf{L}(\Pre(\calC),\calD)$ is a symmetric monoidal subcategory of $\Fun(\Pre(\calC),\calD)$, equivalent as a symmetric monoidal category to $\Fun(\calC,\calD)$.

    Therefore, precomposition by the Yoneda embedding induces an equivalence
    $$ \CAlg(\Fun^\mathsf{L}(\Pre(\calC),\calD)^\otimes) \xrightarrow{\ \ \simeq \ \ } \CAlg(\Fun(\calC,\calD)^\otimes).$$

    By \Cref{Prop:maindayconvolution}, the left and right-hand sides are respectively equivalent to $\Fun^{\mathsf{L,lax}}(\Pre(\calC)^\otimes,\calD^\otimes)$ and $\Fun^{\mathsf{lax}}(\calC^\otimes,\calD^\otimes)$. Indeed, for the first one, note that if $\calE^\otimes$ is a symmetric monoidal category with a symmetric monoidal full subcategory $\calF^\otimes$, then the categories of commutative algebras in $\calF$ is that of commutative algebras in $\calE$ whose underlying object lies in $\calF$. This follows from the description of commutative algebras in terms of the representing coCartesian fibrations.
\end{proof}

\begin{Prop}\label{Prop:dayconvolution} The equivalence of \Cref{Prop:laxdayconvolution} restricts to an equivalence:
$$- \circ y^\otimes: \Fun^{\mathsf{L,\otimes}}(\Pre(\calC)^\otimes,\calD^\otimes) \xrightarrow{\simeq} \Fun^{\otimes}(\calC^\otimes,\calD^\otimes)$$
between subcategories of (strongly) symmetric monoidal functors.
\end{Prop}
\begin{proof}
    Since $y^\otimes$ is (strongly) symmetric monoidal, the functor $-\circ y^\otimes$ preserves strong monoidality. On the other hand, if $F\in \Fun^{\mathsf{lax}}(\calC^\otimes,\calD^\otimes)$ is actually strongly symmetric monoidal (which, once we have specified a lax monoidal structure, is a property and does not require additional data), let $H=\mathsf{LKE}^\otimes(F)$. Then, we have to show that $H$ is strongly symmetric monoidal, i.e., that for any $\calF,\calG \in\Pre(\calC)$, the map $H(\calF)\otimes_\calD H(\calG) \longrightarrow H(\calF \otimes_{\Pre(\calC)} \calG)$ (coming from the lax symmetric monoidal structure on $H$) is an equivalence. By assumption, this holds for $\calF$ and $\calG$ representable presheaves. Moreover, the class of presheaves $\calF$ and $\calG$ for which this holds is closed under colimits in both $\calF$ and $\calG$. Indeed, $H$ preserves colimits by construction, and tensor products in both $\calD$ and $\Pre(\calC)$ preserve colimits in each variable (by \Cref{Prop:dayconvolutionpsh}). Since every presheaf is a colimit of representable ones, this concludes the proof.
\end{proof}
\vspace{0.2cm}

\subsection{Slice \texorpdfstring{$\infty$}{infinity}-categories}\label{Subsect:slices}

For $\calC$ a category and $X\in\calC$, one defines a \emph{slice category} $\calC_{/X} := \Fun(\Delta^1,\calC) \times_{\Fun(\{1\},\calC)} \{X\}$ \cite[Prop.\ 1.2.9.2]{Lurie-HTT}. This can be thought of as the category of edges in $\calC$ with target $X$, and morphisms over $X$. It comes with a forgetful functor $\calC_{/X} \to \calC$, remembering the source of an edge. Slice categories are an essential ingredient of our discussion in \Cref{Sect:Thom,Sect:comp}, as the respective sources of the motivic or topological Thom spectrum functors are slice categories.

\begin{Def}\label{Def:subslice}
    Let $\calC$ be a category, $X\in\calC$ and $\calD \subseteq \calC$ a subcategory. Then the category $\calD_{/X}$ is defined as the fiber of the forgetful functor $\calC_{/X} \to \calC$ over $\calD$. 
\end{Def}

\begin{Prop}\label{Prop:sliceofpsh} Let $\calC$ be a category, and $\calF\in\Pre(\calC)$. We view $\calC$ as a subcategory of $\Pre(\calC)$ via the Yoneda embedding. Then, the left Kan extension of the embedding $\calC_{/\calF} \to \Pre(\calC)_{/\calF}$ (from \Cref{Def:subslice}) defines an equivalence of categories $\theta: \Pre(\calC_{/\calF}) \lsimeq{} \Pre(\calC)_{/\calF}$. The functor $\theta$ restricts to an equivalence $\Pre_\Sigma(\calC_{/\calF}) \lsimeq{} \Pre_\Sigma(\calC)_{/\calF}$.
\end{Prop}
\begin{proof} The first claim is proven in \protect{\cite[\S 5.3]{ABG}}. For the second part of the statement, consider $\calG\in\Pre_\Sigma(\calC)_{/\calF}$. Then $\calG$ is a sifted colimit of arrows of the form $y(c) \to \calF$ for $c\in\calC$, so that $\theta^{-1}(\calG)$ is a sifted colimit of representables $y(c\to \calF)$. It is thus a spherical presheaf in $\Pre(\calC_{/\calD})$. The argument for the converse is the same. 
\end{proof}

Slices of symmetric monoidal categories over commutative algebras inherit a symmetric monoidal structure:
\begin{Prop}[\protect{\cite[Thm.\ 2.2.2.4, Rmk.\ 2.2.2.5]{Lurie-HA}}]\label{Prop:smstconslice} Let $\calD^\otimes$ be a symmetric monoidal category and $X\in \CAlg(\calD)$. Then the slice $\calD_{/X}$ admits a symmetric monoidal structure $\calD^\otimes_{/X}$, making in particular the projection $\calD^\otimes_{/X} \to \calD^\otimes$ symmetric monoidal. Informally, the tensor product of two objects $c\to X$ and $d\to X$ in the slice can be described by $(c\to X)\otimes (d\to X) = (c\otimes d \to X \otimes X \to X)$
where $X\otimes X \to X$ is given by the multiplicative structure of $X$.

\end{Prop}

As a particular case of \Cref{Prop:sliceofpsh}, for $\calC = \ast$, so that $\Pre(\calC) \simeq \spaces$, we obtain the following result.
\begin{Lemma}[\protect{\cite[\S 6.2]{ABG}}]\label{Prop:sliceofpshsm}
Let $X\in\mathsf{CMon}(\spaces)$. In the case $\calC = \ast$, \Cref{Prop:sliceofpsh} provides an equivalence of symmetric monoidal categories $\spaces_{/X} \simeq \Pre(X)$,
where $\spaces_{/X}$ is endowed with the symmetric monoidal structure from \Cref{Prop:smstconslice}, and $\Pre(X)$ is the category of presheaves on $X$ with the Day convolution symmetric monoidal structure (see \Cref{Subsect:dayconvolution}), viewing $X$ as a symmetric monoidal category.
\end{Lemma}

The symmetric monoidal structure from \Cref{Prop:smstconslice} has a very convenient universal property:
\begin{Prop}[\protect{\cite[Lem.\ 2.12]{ACB}}]\label{Prop:algebrasintheslice} Let $\calC^\otimes$ and $\calD^\otimes$ be symmetric monoidal categories classified by coCartesian fibrations $p:\calC^\otimes \to \Finstar$ and $q: \calD^\otimes \to \Finstar$. Let $X\in \CAlg(\calD)$, and $F:\calC^\otimes \to \calD^\otimes$ be a lax symmetric monoidal functor. 
Then lax symmetric monoidal functors $\calC^\otimes \to \calD^\otimes_{/X}$ lifting $F$ correspond to symmetric monoidal natural transformations $F \to X \circ p$, where $X$ is viewed as a section of $q$. More precisely, there is an equivalence
$$ \map_{\Fun^\mathsf{lax}(\calC^\otimes,\calD^\otimes)}(F,X\circ p) \simeq \Fun^\mathsf{lax}(\calC^\otimes, \calD^\otimes_{/X}) \times_{\Fun^\mathsf{lax}(\calC^\otimes,\calD^\otimes)} \{F\}.$$

In particular, for $\calC^\otimes = \E{n}^\otimes$, if we view $F$ as an $\E{n}$-algebra $Y\in\Alg_{\E{n}}(\calD^\otimes)$, then $\E{n}$-algebras in $\calD^\otimes_{/X}$ lifting $Y\in\Alg_{\E{n}}(\calD^\otimes)$ correspond to morphisms of $\E{n}$-algebras $Y\to X$ in $\calD^\otimes$. More precisely, there is an equivalence
$$\map_{\Alg_{\E{n}}(\calD^\otimes)}(Y,X) \simeq \Alg_{\E{n}}(\calD^\otimes_{/X}) \times_{\Alg_{\E{n}}(\calD^\otimes)} \{Y\}.$$  
\end{Prop}

\begin{Rmk} Both \Cref{Prop:smstconslice,Prop:algebrasintheslice} are proven in a more general setting in the references we gave; namely for $\mathcal{O}^\otimes$-monoidal categories, for some $\infty$-operad $\mathcal{O}^\otimes$. We will only need the case $\mathcal{O}^\otimes = \Finstar$. 
\end{Rmk}


\vspace{0.4em}
\section{Motivic and topological multiplicative Thom spectra}\label{Sect:Thom}

In this section and the next, our goal is to show the compatibility of real realization with the multiplicative motivic and topological Thom spectrum functors. This will in particular allow us to show in \Cref{Sect:comp} that the real realization of the motivic $\Einfty$-rings $\MGL, \MSO, \MSp \in \SH(\R)$ (representing variants of algebraic cobordism) are equivalent to the $\Einfty$-rings $\MO, \MSO, \MSp^\mathsf{top} \in \Sp$ respectively (representing variants of topological cobordism). Again, this result was already proved at the level of spectra in \cite[Cor.\ 4.7]{BH-eta}, but we want to take into account the $\Einfty$-ring structures.

These $\Einfty$-structures are obtained by realizing the underlying objects as Thom spectra of certain commutative algebra objects, using the fact that the corresponding Thom spectrum functors are symmetric monoidal.

Therefore, we first recall in this section the precise construction of these functors and their symmetric monoidal structures, in the topological case (\Cref{Subsect:topThom}) and in the motivic case (\Cref{Subsect:motThom}). In order to solve some size issues arising in the latter context, we will need some prerequisites about subcategories of $\tau$-compact objects, where $\tau$ is a regular cardinal (\Cref{Subsect:compact}).
\vspace{0.2cm}

\subsection{The symmetric monoidal topological Thom spectrum functor}\label{Subsect:topThom}\hfill\vspace{0.2cm}

\begin{Def}\label{Def:Thomspectrum} Let $\mathsf{Pic}(\Sp)\in\Spc$ be the subgroupoid of invertible objects in $\Sp$ (with respect to the tensor product). The (topological) \emph{Thom spectrum functor}
$\Mtop: \spaces_{/\mathsf{Pic}(\Sp)} \longrightarrow \Sp$
sends an arrow $X \to \mathsf{Pic}(\Sp)$ to the colimit of the diagram  $X \longrightarrow \mathsf{Pic}(\Sp) \xhookrightarrow{\ } \Sp$.
More precisely, under the equivalence $\spaces_{/\mathsf{Pic}(\Sp)} \simeq \Pre(\mathsf{Pic}(\Sp))$ from \Cref{Prop:sliceofpshsm}, the functor $\Mtop$ is the left Kan extension of the embedding $\mathsf{Pic}(\Sp) \xhookrightarrow{} \Sp$.
\end{Def}

Viewing $\mathsf{Pic}(\Sp) \subseteq \Sp$ as a symmetric monoidal subcategory allows us to endow the functor $\Mtop$ with a symmetric monoidal structure. More precisely, for any symmetric monoidal small subcategory $\mathsf{Pic}(\Sp) \subseteq \calC \subseteq \Sp$, the maximal groupoid in $\calC$, denoted by $\calC^\simeq$, is then an $\Einfty$-algebra in $\Spc$ (the functor $(-)^\simeq : \Cat \to \mathsf{Gpd} \simeq \Spc$ is right adjoint to the forgetful functor, it preserves products, and thus is symmetric monoidal with respect to the Cartesian structures). Thus, \Cref{Prop:smstconslice} gives us a symmetric monoidal structure on the slice category $\Spc_{/\calC^\simeq}$, with the universal property described in \Cref{Prop:dayconvolution}. 

\begin{Rmk}\label{Rmk:smallness} Note that $\mathsf{Pic}(\Sp)$ is itself essentially small (invertible objects are compact, and $\Sp^\omega$ is small, see for example \Cref{Prop:toolscompact}). In \Cref{Prop:Mtopissm} below, we restrict ourselves to intermediary subcategories $\mathsf{Pic}(\Sp) \subseteq \calC \subseteq \Sp$ that are small, in order to avoid some set-theoretic issues. Indeed, only small groupoids can be viewed as objects in $\Spc$. Although the Thom spectrum of any arrow $X\to \Sp^\simeq$ could in principle be considered, $\Sp^\simeq$ is not small. Actually, all the examples of arrows $X\to \calC^\simeq$ we will be interested in factor through $\mathsf{Pic}(\Sp)$ (but using only $\mathsf{Pic}(\Sp)$ would cause us difficulties in \Cref{Subsubsect:colimfunctforrealization}, see \Cref{Rmk:whykappa}). 
\end{Rmk}

\begin{thm}[\protect{\cite[Thm.\ 1.6]{ABG}}]\label{Prop:Mtopissm}
	For any symmetric monoidal small subcategory $\mathsf{Pic}(\Sp) \subseteq \calC \subseteq \Sp$, there is a symmetric monoidal colimit preserving topological Thom spectrum functor 
	$\Mtop: \spaces_{/\calC^\simeq} \longrightarrow \Sp$,
	left Kan extended as a symmetric monoidal functor (see \Cref{Def:smLKE}) from the embedding $\calC^\simeq \xhookrightarrow{} \Sp$.
\end{thm}

The symmetric monoidal left Kan extension described in the theorem agrees with the construction in \cite[Thm.\ 1.6]{ABG} because, by \Cref{Prop:dayconvolution}, it is determined as a symmetric monoidal functor by its restriction to $\ast_{/\calC^\simeq} \simeq \calC^\simeq$.

\begin{Notation}\label{Notation:Mtopkappa} Let $\kappa$ be a fixed uncountable regular cardinal such that $\Spc^\kappa$, the full subcategory of $\kappa$-compact objects in $\Spc$, is closed under finite limits in $\Spc$. Such a cardinal $\kappa$ exists by \cite[Lem.\ 5.4.7.6]{Lurie-HA}, and one may actually choose $\kappa = \aleph_1$ \cite[Rmk.\ 5.4.2.9]{Lurie-HTT}. Note that $\Spc^\omega$ is not closed under limits, for example $\Omega \calS^1$ is a pullback of compact objects, but it is not compact. From now on, $\Mtop$ will denote the Thom spectrum functor of \Cref{Prop:Mtopissm} corresponding to $\calC = \Sp^\kappa$ (by \Cref{Prop:Spkappasm} below, this is indeed a small symmetric monoidal subcategory in $\Sp$). It contains $\mathsf{Pic}(\Sp)$, as an invertible spectrum $E$ is compact ($E$ is the image of the compact object $\Sph$ under the equivalence $-\Smash E$), in particular $\kappa$-compact. The requirement that $\Spc^\kappa$ is closed under finite limits will become crucial in \Cref{Subsubsect:colimfunctforrealization}, see \Cref{Rmk:whykappa}.
\end{Notation}

\begin{Ex}\label{Ex:MSO} Let us see how one can define $\Einfty$-structures on $\mathsf{MO}$, $\MSO$, and $\mathsf{MU}$. The spectrum $\MO$ representing topological cobordism can be recovered as $\Mtop(BO \to \Sp^\kgpd)$ where the map $BO \to \Sp^\kgpd$ is the \emph{j-homomorphism} 
	(by \cite[Cor.\ 3.24]{ABG}, the construction $\Mtop$ agrees with that of \cite[\S IX]{LMSM}, which, by \cite[\S IX, Rmk.\ 3.3]{LMSM}, agrees in turn with the most classical one \cite{May-Einfty}). 
	To define the latter, following \cite[p.\,2]{HY}, one first considers the symmetric monoidal functor $\coprod_{n\in\N} BO_n \to \Sp$, where the left-hand side is viewed as the maximal $\infty$-groupoid in the 1-category of real vector spaces viewed as an $\infty$-category (or in other terms real vector bundles on the point), and a vector bundle is sent to the infinite suspension of its Thom space (in this case, the one-point compactification of our vector space). This map takes values in $\mathsf{Pic}(\Sp)$, and thus factors through the group completion $\Z\times BO$ of $\coprod_{n\in\N} BO_n$. Restricting to $\{0\} \times BO$ (corresponding to rank 0 virtual vector bundles on the point), we obtain the $j$-homomorphism $BO \to \mathsf{Pic}(\Sp) \xhookrightarrow{} \Sp^\kgpd$. Such constructions date back to \cite{BV} for the $\Einfty$-structure on $BO$, and \cite[\S 3]{Whitehead-rotation} for a $j$-homomorphism.
	
	Similarly, we have $\MSO\simeq \Mtop(BSO \to BO \to \Sp^\kgpd)$ and $\mathsf{MU}\simeq\Mtop(BU \to BO \to \Sp^\kgpd)$. By \Cref{Prop:algebrasintheslice}, any map of $\Einfty$-spaces $X\to \Sp^\kgpd$ defines a commutative algebra in $\spaces_{/\Sp^\kgpd}$. We thus have $\Einfty$-algebra structures on $BO \to \Sp^\kgpd$, $BSO \to \Sp^\kgpd$, and $BU \to \Sp^\kgpd$ in the slice category (since the j-homomorphism is an $\Einfty$-map by its construction as a symmetric monoidal functor). The Thom spectrum functor being symmetric monoidal, we obtain $\Einfty$-structures on $\mathsf{MO}$, $\MSO$, and $\mathsf{MU}$. 
\end{Ex}

\subsection{Digression: subcategories of \texorpdfstring{$\tau$}{tau}-compact objects}\label{Subsect:compact}\hfill\vspace{0.2cm}

Because of the choices in \Cref{Notation:Mtopkappa}, we will need some technical results about subcategories of ($\omega$-)compact and $\kappa$-compact objects in general, and also more specific results about $\kappa$-compact objects in the slices of the category $\Spc$. This will be useful to us in \Cref{Subsect:motThom,Subsubsect:colimfunctforrealization}, where we consider, respectively, the categories $\SH(X)^\omega$ and $\Sp(\Spc_{/\rR X})^\kappa$ for $X\in\Sm_\R$.

\begin{Lemma}\label{Prop:productcompact} Let $\tau$ be a regular cardinal, and let $\{\calC_i\}_{i\in I}$ be a family of categories indexed by a $\tau$-small set $I$. Then, the subcategory of $\tau$-compact objects in their product is the product of their respective subcategories of $\tau$-compact objects, i.e., $\big(\prod_{i\in I} \calC_i\big)^\tau \simeq \prod_{i\in I} (\calC_i)^\tau$.
\end{Lemma}
\begin{proof} Let $\calD := \prod_{i\in I} \calC_i$, and for all $i\in I$, let $\pi_i : \calD \to \calC_i$ be the projection. Let $d\in\calD$. We want to show that $d$ is $\tau$-compact in $\calD$ if and only if $\pi_i(d)$ is $\tau$-compact in $\calC_i$ for all $i\in I$. Assume first $d\in\calD^\tau$. Let $i\in I$ and $\colim_{j\in\mathcal{J}} c_j$ be a $\tau$-filtered colimit diagram in $\calC_i$. Note that the $\tau$-filtered colimit of a constant diagram on a certain object is always this object again (it holds by \cite[Prop.\ 4.3.1.12]{Lurie-HTT} for colimits indexed by weakly contractible categories, in particular for filtered categories by \cite[Lem.\ 5.3.1.18]{Lurie-HTT} (and a $\tau$-filtered category is filtered)). By assumption, we have equivalences
\begin{align*}
	\map_{\calC_i}(\pi_i(d),\colim_{j\in\mathcal{J}} c_j) \times &\prod_{\ell \in I\setminus\{i\}} \map_{\calC_\ell}(\pi_\ell(d),\ \colim_{j\in\mathcal{J}} \pi_\ell(d)) \\
	&\simeq \map_{\calD}(d, \colim_{j\in\mathcal{J}} (c_j, (\pi_\ell(d))_{\ell \neq i})) \\
	&\simeq \colim_{j\in\mathcal{J}} \map_\calD(d, (c_j, (\pi_\ell(d))_{\ell \neq i})) \\
	&\simeq \colim_{j\in\mathcal{J}} \left(\map_{\calC_i}(\pi_i(d), c_j) \times \prod_{\ell \in I\setminus\{i\}} \map_{\calC_\ell}(\pi_\ell(d), \pi_\ell(d))\right).
\end{align*}
Now, the natural map $\colim_{j\in\mathcal{J}} \map_{\calC_i}(\pi_i(d), c_j) \to \map_{\calC_i}(\pi_i(d),\colim_{j\in\mathcal{J}} c_j)$ is a retract of this equivalence (pick the identities in all $\map_{\calC_\ell}(\pi_\ell(d), \pi_\ell(d))$ on both sides), so it is an equivalence as well. 

The converse direction is \cite[Lem.\ 5.3.4.10(2)]{Lurie-HTT}.
\end{proof}

\begin{Lemma}\label{Prop:toolscompact} Let $\calC$ be a stable category generated under colimits by an essentially small family of compact objects $E \subseteq \calC$. Then:
	\begin{enumerate}[label = (\roman*)]
		\item The full subcategory $\calC^\omega$ of compact objects in $\calC$ is the thick subcategory of $\calC$ generated by $E$.
		
		\item Let $\calD$ be a stable category with small colimits, and $L: \calC \leftrightarrows \calD : R$ be an adjoint pair. Then $L$ preserves compact objects if and only if $L(E) \subseteq \calD^\omega$.
		
		\item For an uncountable regular cardinal $\tau$, if $E$ generates $\calC$ under filtered colimits, then the full subcategory of $\calC^\tau$ of $\tau$-compact objects in $\calC$ is the full subcategory spanned by retracts of colimits of $\tau$-small $(\omega-)$filtered diagrams with values in $\calC^\omega$. In the situation of item $(ii)$, $L$ preserves $\tau$-compact objects if and only if $L(E) \subseteq \calD^\tau$.
	\end{enumerate}
\end{Lemma}
\begin{proof} Item $(i)$ is \cite[Lem.\ 2.2]{Neeman-locthm}. In part $(ii)$, the ``only if'' direction is tautological. For the converse, note that since $L$ preserves colimits and $\calD^\omega$ is thick (by stability of $\calD$ and \cite[Tag \href{https://kerodon.net/tag/064W}{064W}]{Kerodon}), the full subcategory $\calC'' \subseteq \calC$ of objects $c''$ such that $L(c'')\in \calD^\omega$ is thick. If $L(E) \subseteq \calD^\omega$, then $\calC''$ contains $E$ by assumption. Therefore, using part $(i)$, $\calC''$ contains $\calC^\omega$, so $L$ preserves compact objects. 
	
	By the proof of \cite[Prop.\ 5.4.2.11]{Lurie-HTT} (which applies since $\tau \gg \omega$ \cite[Rmk.\ 5.4.2.9]{Lurie-HTT}), we have $\calC \simeq \mathsf{Ind}_\tau(\calC')$, where $\calC' \subseteq \calC$ is the full subcategory spanned by colimits of $\tau$-small $(\omega-)$filtered diagrams with values in $\calC^\omega$. Furthermore, the proof of \cite[Lem.\ 5.4.2.4]{Lurie-HTT} shows that the subcategory of $\tau$-compact objects of $\mathsf{Ind}_\tau(\calC')$ is spanned by retracts of objects in $\calC'$. This proves the first part of item $(iii)$. For the second part, the ``only if'' direction is again tautological. Conversely, assuming that $L(E) \subseteq \calD^\tau$, let $c\in\calC^\tau$. Then, by what we just proved, $L(c)$ is a retract of a colimit of some $\tau$-small $(\omega-)$filtered diagram with values in $L(\calC^\omega)$. By the same argument as in item $(i)$, since $L(E)\subseteq \calD^\tau$ and $\calD^\tau$ is thick, we have $L(\calC^\omega)\subseteq \calD^\tau$. The latter category being closed under $\tau$-small colimits and retracts \cite[Tags \href{https://kerodon.net/tag/064W}{064W} and \href{https://kerodon.net/tag/064Z}{064Z}]{Kerodon}, $L(c)$ is $\tau$-compact, as desired. 
\end{proof}

\begin{Lemma}\label{Prop:kappacompinslice} 
	\begin{enumerate}[label = (\roman*)]
	\item Let $\tau$ be a regular cardinal. For any $B\in \Spc^\tau$, the source functor $\Spc_{/B} \to \Spc$ preserves and detects $\tau$-compact objects.
	\item Let $k=\R$ or $\C$. The functor $r_k: \Sm_k \to \Spc$ takes values in $\Spc^\kappa$. In particular, for any $X\in\Sm_k$, an arrow $(E\to r_k X) \in \Spc_{/r_k X}$ defines a $\kappa$-compact object if and only if $E$ is $\kappa$-compact in $\Spc$.
	\end{enumerate}
\end{Lemma}
\begin{proof}
	The source functor $\Spc_{/B} \to \Spc$ (forgetting the structure map) admits a right adjoint, sending $E\in\Spc$ to the projection map $(E\times B \to B)\in\Spc_{/B}$. The latter commutes with all colimits in $\Spc$ (by universality of colimits, see e.g.\ \cite[6.1.0.1]{Lurie-HTT}), in particular with $\tau$-filtered ones, and thus the source functor preserves $\tau$-compact objects. Conversely, if $f: E \to B$ is a map with $E\in\Spc^\tau$, let $\colim_{j\in\calJ} (g_j: A_j\to B)$ be a $\tau$-filtered colimit diagram in $\Spc_{/B}$. Then, we have:
	\begin{align*}
	\map_{\Spc_{/B}}(E,\colim_{j\in\calJ} A_j) &\simeq \map_{\Spc}(E,\colim_{j\in\calJ} A_j) \times_{\map_{\Spc}(E,B)} \{\colim_{j\in\calJ} g_j\} \\
	&\simeq (\colim_{j\in\calJ} \map_{\Spc}(E, A_j)) \times_{\map_{\Spc}(E,B)} \{\colim_{j\in\calJ} g_j\} \tag{\text{since $E\in\Spc^\tau$}}\\
	&\simeq \colim_{j\in\calJ} (\map_{\Spc}(E, A_j) \times_{\map_{\Spc}(E,B)} \{ g_j\}) \tag{\text{colimits are universal in $\Spc$}}\\
	&\simeq \colim_{j\in\calJ} \map_{\Spc_{/B}}(E, A_j)
	\end{align*} 
	which proves $\tau$-compactness of $(E\to B) \in \Spc_{/B}$.

	For item $(ii)$, note that the real realization of a smooth scheme of finite type is a smooth manifold, and such a manifold admits a countable triangulation (by \cite[\S 2]{Cairns-triangulation} or \cite[Thm.\ 7 and \S 4]{Whitehead-triangulation}, and the fact that a manifold is by definition a second-countable topological space). Thus, any smooth manifold defines a $\kappa$-small space ($\kappa$ was chosen uncountable). The complex case is similar. The second part of the claim is then a particular case of $(i)$.
\end{proof}

\begin{Lemma}\label{Prop:Spkappasm} Let $B\in\Spc^\kappa$. Then $\Sp(\Spc_{/B})^\kappa \subseteq \Sp(\Spc_{/B})$ is a symmetric monoidal subcategory, and it is essentially small. Here, the symmetric monoidal structure on $\Sp(\Spc_{/B}) := \Sp \otimes_{\PrL} \Spc_{/B}$ (see \cite[Ex.\ 4.8.1.23]{Lurie-HA}) is induced by the Cartesian structure on $\Spc_{/B}$.
\end{Lemma}
\begin{proof} First note that $\Spc_{/B}$ admits finite products, since by \cref{Prop:sliceofpshsm} it is equivalent to $\Pre(\ast_{/B})$. 
By \cite[Lem.\ 5.4.2.4]{Lurie-HTT}, it suffices to show that this subcategory is closed under the tensor product and contains the unit. First note that the functor $\Sigma^\infty_+$ preserves $\tau$-compact objects for any cardinal $\tau$, because its right adjoint commutes with filtered colimits, in particular with $\tau$-filtered colimits (in view of the equivalence $\Sp(\Spc_{/B}) \simeq \Fun(B^\op,\Sp)$, this follows from the corresponding result for $\Omega^\infty:\Spc \to \Sp$ (or use the same argument as in \cite[Prop.\ C.12]{Hoyois-traceformula})). Thus, since $B\to B$ is $\kappa$-compact in $\Spc_{/B}$ by \Cref{Prop:kappacompinslice}, the unit $\Sigma^\infty_+ (B\to B)$ is $\kappa$-compact in $\Sp(\Spc_{/B})$. 

The category $\Sp(\Spc_{/B})$ is generated under filtered colimits by $E:=\{ \Sigma^n \Sigma^\infty_+ V \mid n\in\Z, V\in (\Spc_{/B})^\omega\}$, which is a (small) set of compact objects (the category $(\Spc_{/B})^\omega \simeq \Pre(B)^\omega$ is essentially small by \cite[Prop.\ 5.3.4.17]{Lurie-HTT}). In particular, it is itself essentially small.

To conclude the proof, we have to show that the subcategory of $\kappa$-compact objects is stable under the tensor product. Let $X\in \Sp(\Spc_{/B})^\kappa$, we will show that the functor $-\otimes X$ preserves $\kappa$-compact objects. By \Cref{Prop:toolscompact} (and the fact that the tensor product preserves colimits in each variable separately), it is sufficient to prove that it sends any object of the form $\Sigma^\infty_+ Y$ for $Y\in(\Spc_{/B})^\omega$ to a $\kappa$-compact object. Therefore, it suffices to show $-\otimes \Sigma^\infty_+ Y$ preserves $\kappa$-compact objects, and by the same argument, it then suffices to prove that $\Sigma^\infty_+ Y \Smash \Sigma^\infty_+ X \simeq \Sigma^\infty_+(X\times Y)$ is $\kappa$-compact for any $X,Y\in (\Spc_{/B})^\kappa$. By \Cref{Prop:kappacompinslice} below, it suffices to show that the fiber product $X'\times_{B} Y'$ in $\Spc$ is $\kappa$-compact, where $X'$ and $Y'$ are the sources of the arrows $X$ and $Y$, respectively. But by the same Lemma, $X',Y'\in\Spc^\kappa$, and $B\in\Spc^\kappa$ by assumption. Thus, by our choice of $\kappa$, the finite limit $X'\times_{B} Y'$ is again $\kappa$-compact, as needed.
\end{proof}

\subsection{The symmetric monoidal motivic Thom spectrum functor}\label{Subsect:motThom}\hfill\vspace{0.2cm}

In this subsection, we define the motivic analog to the topological Thom spectrum functor $\Mtop$. More precisely, this will be a symmetric monoidal functor $M_S: \Pre_\Sigma(\Sm_S)_{/\SH^\wgpd} \to \SH(S)$, where $\Pre_\Sigma(\Sm_S)$ is the category of spherical presheaves (\Cref{Def:sphericalpresheaves}), and $\SH^\wgpd$ is the presheaf sending $X\in\Sm_S$ to the maximal groupoid in the subcategory $\SH(X)^\omega$ of compact objects of $\SH(X)$. Again, the compactness requirement has only the purpose of avoiding set-theoretic issues.

\subsubsection{Functoriality of \texorpdfstring{$\SH (-)$}{SH(-)} and \texorpdfstring{$\SH(-)^\omega$}{the subpresheaf of compact objects}}\label{Subsubsect:functSH}\hfill\vspace{0.2cm}

The definition of the motivic Thom spectrum functor and its symmetric monoidal structure relies on the study of the functoriality in the base scheme $S$ of the whole construction, in particular of the assignments $S \mapsto \SH(S)$ and $S \mapsto \SH(S)^\omega$. This is what this subsection is dedicated to. The functoriality properties of $\SH$ we will study fit into a complete six functors formalism, constructed by Ayoub \cite{Ayoub}. We want to consider $\SH$ as a spherical presheaf (\Cref{Def:sphericalpresheaves}) of (not necessarily small) symmetric monoidal categories on $\Sm_S$. 

\begin{Def} Let $\fold$ be the collection of morphisms in $\Sm_S$ containing all finite coproducts of fold maps, i.e., maps of the form $Y:= \coprod_{i\leq n} X_i^{\amalg m_i} \to \coprod_{i\leq n} X_i := X$ for $n\geq 0$ and $m_1,\dots,m_n\geq 0$.
\end{Def}

\begin{thm}[\protect{\cite[\S 1.4.1]{Ayoub}}, \protect{\cite{Hoyois-6op}}]\label{Prop:6FFforSH} Let $S$ be a qcqs scheme and let ${f:X\to Y}$ be a morphism in $\Sm_S$. 
\begin{enumerate}[label = (\roman*)]
    \item For all $A,B\in\Sm_S$, we have $\SH(A\amalg B) \simeq \SH(A) \times \SH(B)$. Given a morphism $\nabla \in\fold$ of the form $V:= \coprod_{i\leq n} U_i^{\amalg m_i} \to \coprod_{i\leq n} U_i := U$, let $\nabla_\otimes: \SH(V) \to \SH(U)$ be the map induced under these identifications by $\prod_{i\leq n} \SH(U_i)^{m_i} \to \prod_{i\leq n} \SH(U_i)$ given by the tensor products $\otimes: \SH(U_i)^{m_i} \to \SH(U_i)$ each $i\leq n$. 
     \item There are induced pullback-pushforward adjunctions
\[\begin{tikzcd}[column sep = 5em]
	{\Pre(\Sm_X)} & {\Spc(X)} & {\SH(X)} \\
	{\Pre(\Sm_Y)} & {\Spc(Y)} & {\SH(Y)}
	\arrow["\Lmot"{description}, from=1-1, to=1-2]
	\arrow[""{name=0, anchor=center, inner sep=0}, "{f_\ast}", shift left=2, from=1-1, to=2-1]
	\arrow["{\Sigma_+^\infty}"{description}, from=1-2, to=1-3]
	\arrow[""{name=1, anchor=center, inner sep=0}, "{f_\ast}", shift left=2, from=1-2, to=2-2]
	\arrow[""{name=2, anchor=center, inner sep=0}, "{f_\ast}", shift left=2, from=1-3, to=2-3]
	\arrow[""{name=3, anchor=center, inner sep=0}, "{f^\ast}", shift left=2, from=2-1, to=1-1]
	\arrow["\Lmot"{description}, from=2-1, to=2-2]
	\arrow[""{name=4, anchor=center, inner sep=0}, "{f^\ast}", shift left=2, from=2-2, to=1-2]
	\arrow["{\Sigma_+^\infty}"{description}, from=2-2, to=2-3]
	\arrow[""{name=5, anchor=center, inner sep=0}, "{f^\ast}", shift left=2, from=2-3, to=1-3]
	\arrow["\dashv"{anchor=center}, draw=none, from=3, to=0]
	\arrow["\dashv"{anchor=center}, draw=none, from=4, to=1]
	\arrow["\dashv"{anchor=center}, draw=none, from=5, to=2]
\end{tikzcd}\]
where the squares involving $f^\ast$ commute. The squares formed by $f_\ast$ and the right adjoints $\iota \vdash \Lmot$, $\Omega^\infty \vdash \Sigma^\infty_+$ also commute. 

\item All three functors $f^\ast$ in $(ii)$ are symmetric monoidal, with respect to the Cartesian structure on the categories of presheaves and motivic spaces, and the smash product on the stable motivic categories.

\item If $f$ is moreover smooth, there are induced adjunctions
\[\begin{tikzcd}[ampersand replacement=\&, column sep = 5em]
	{\Pre(\Sm_X)} \& {\Spc(X)} \& {\SH(X)} \\
	{\Pre(\Sm_Y)} \& {\Spc(Y)} \& {\SH(Y)}
	\arrow["\Lmot"{description}, from=1-1, to=1-2]
	\arrow[""{name=0, anchor=center, inner sep=0}, "{f_\sharp}"', shift right=2, from=1-1, to=2-1]
	\arrow["{\Sigma_+^\infty}"{description}, from=1-2, to=1-3]
	\arrow[""{name=1, anchor=center, inner sep=0}, "{f_\sharp}"', shift right=2, from=1-2, to=2-2]
	\arrow[""{name=2, anchor=center, inner sep=0}, "{f_\sharp}"', shift right=2, from=1-3, to=2-3]
	\arrow[""{name=3, anchor=center, inner sep=0}, "{f^\ast}"', shift right=2, from=2-1, to=1-1]
	\arrow["\Lmot"{description}, from=2-1, to=2-2]
	\arrow[""{name=4, anchor=center, inner sep=0}, "{f^\ast}"', shift right=2, from=2-2, to=1-2]
	\arrow["{\Sigma_+^\infty}"{description}, from=2-2, to=2-3]
	\arrow[""{name=5, anchor=center, inner sep=0}, "{f^\ast}"', shift right=2, from=2-3, to=1-3]
	\arrow["\dashv"{anchor=center}, draw=none, from=0, to=3]
	\arrow["\dashv"{anchor=center}, draw=none, from=1, to=4]
	\arrow["\dashv"{anchor=center}, draw=none, from=2, to=5]
\end{tikzcd}\]
where the squares involving either $f^\ast$ or $f_\sharp$ commute.
\item In the situation of $(iv)$, the projection formula holds for all three functors $f_\sharp$, i.e., for every $A$ in the source of $f_\sharp$ and $B$ in its target, we have
$f_\sharp(A\otimes f^\ast B) \simeq f_\sharp(A) \otimes B$.
    \item The constructions of $(i)$ to $(iv)$ are functorial in $f$.
    \item For every commutative square in $\Sm_S$
\[\begin{tikzcd}
	{X'} & X \\
	{Y'} & Y
	\arrow["{g'}", from=1-1, to=1-2]
	\arrow["{f'}"', from=1-1, to=2-1]
	\arrow["f", from=1-2, to=2-2]
	\arrow["g", from=2-1, to=2-2]
\end{tikzcd}\]
    with $f$ and $f'$ smooth, there is an exchange transformation of functors $\SH(X) \to \SH(Y')$
	$$\mathsf{Ex}_{\sharp}^\ast: f'_\sharp(g')^\ast \Longrightarrow g^\ast f_\sharp$$
	(defined in the proof below). It is an equivalence if the square is Cartesian in $\Sm_S$.
    \item For every diagram in $\Sm_S$ of the form
\[\begin{tikzcd}[ampersand replacement=\&]
	W \& {Y'} \& {Z'} \\
	X \& Y \& Z
	\arrow["g"', from=1-1, to=2-1]
	\arrow[from=1-2, to=1-1]
	\arrow["{\nabla'}", from=1-2, to=1-3]
	\arrow["{u'}", from=1-2, to=2-2]
	\arrow["u", from=1-3, to=2-3]
	\arrow["f"', from=2-2, to=2-1]
	\arrow["\nabla", from=2-2, to=2-3]
\end{tikzcd}\]
    with $u$ and $u'$ smooth and $\nabla, \nabla' \in \fold$, there is a distributivity transformation of functors $\SH(W) \to \SH(Z)$
    $$\mathsf{Dis}_{\sharp\otimes} : u_\sharp\nabla'_\otimes(f')^\ast \Longrightarrow \nabla_\otimes(\pi_Y)_\sharp\pi_W^\ast$$ 
	(defined in the proof below). It is an equivalence when the square on the right-hand side is a pullback and $Z' = R_{Y/Z}(W\times_X Y)$ is the Weil restriction (see \cite[\S 2.3]{BH-norms}, and the proof below for an easier description in our case).
\end{enumerate}
\end{thm}
\begin{proof} We give references for the proof; parts for which no reference is mentioned are proven in \cite{Hoyois-6op}. Part $(i)$ is proven in \cite[proof of Lem.\ 9.6]{BH-norms}. In $(vii)$, the exchange transformation is the composition
    $$\mathsf{Ex}_{\sharp}^\ast: f_\sharp'(g')^\ast \xRightarrow{\ \  \eta_f\ \   } f_\sharp'(g')^\ast f^\ast f_\sharp \simeq f_\sharp'f'^\ast g^\ast f_\sharp \xRightarrow{\ \   \varepsilon_{f'}\ \ } g^\ast f_\sharp$$
    where $\eta_f$ and $\varepsilon_{f'}$ are the (co)units for the adjunctions $f_\sharp \dashv f^\ast$ and $f'_\sharp \dashv f'^\ast$, respectively, and the equivalence in the middle witnesses the commutativity of the square in $(vii)$. It is an equivalence for pullback squares by \cite[Cor.\,6.12]{Hoyois-6op}. In $(viii)$, the distributivity transformation is the composition
    $$\mathsf{Dis}_{\sharp\otimes} : u_\sharp\nabla'_\otimes(f')^\ast \xRightarrow{\ \mathsf{Ex}_{\sharp,\otimes}\ } \nabla_\otimes u'_\sharp (f')^\ast \simeq \nabla_\otimes (\pi_Y)_\sharp(f'\times u')_\sharp (f'\times u')^\ast \pi_W^\ast \xRightarrow{\ \varepsilon_{f'\times u'}\ } \nabla_\otimes(\pi_Y)_\sharp\pi_W^\ast$$
    where the equivalence in the middle uses $u' \simeq \pi_Y\circ (f'\times u')$ and $f' \simeq \pi_W \circ (f'\times u')$, and $\mathsf{Ex}_{\sharp,\otimes}$ is the composition
    $$\mathsf{Ex}_{\sharp,\otimes}: u_\sharp\nabla'_\otimes \xRightarrow{\ \eta_{u'}\ } u_\sharp\nabla'_\otimes (u')^\ast u'_\sharp \simeq u_\sharp u^\ast\nabla_\otimes u'_\sharp \xRightarrow{\ \varepsilon_u \ } \nabla_\otimes u'_\sharp$$
    with the equivalence in the middle given by the symmetric monoidality of $u^\ast$ (since the $\nabla_\otimes$ and $\nabla'_\otimes$ are given by the smash product). It is an equivalence when the conditions of statement $(viii)$ are satisfied by \cite[Prop.\ 5.10]{BH-norms}. The latter result assumes that $h$ is quasi-projective, but we do not need it here: this assumption is only used to ensure the existence of the Weil restriction, however the latter always exists for maps in $\fold$. Indeed, the Weil restriction $R_{Y/X}$ (when it exists) is defined by the property that $\mathsf{Sch}_X(U,R_{Y/X}V) \simeq \mathsf{Sch}_Y(U\times_X Y, V)$ for all $X$-scheme $U$ and $Y$-scheme $V$. In the case of a map $(Y = \coprod_{i\leq n} X_i^{\amalg m_i} \to \coprod_{i\leq n} X_i= X)\in \fold$, it is easy to check that $R_{Y/X}V$ exists and is given by
    $$R_{Y/X}V = \coprod_{i\leq n} \left(V_{X_{i,1}} \times_{X_i} \dots \times_{X_i} V_{X_{i,m_i}}\right) \longrightarrow \coprod_{i\leq n} X_i,$$
    where $X_{i,j}$ is the $j$-th component in $X_i^{\amalg m_i}$ and $V_{X_{i,j}}$ is the component of $V$ over $X_{i,j}$, for all $i\leq n$ and $1\leq j\leq m_i$. 
\end{proof}

We now study the functoriality of the subpresheaf of categories of compact objects $\SH^\omega$, associating to each $X\in\Sm_\R$ the full subcategory of $\SH(X)$ spanned by its compact objects.

\begin{thm}[\protect{\cite[Prop.\ C.12]{Hoyois-traceformula}}]\label{Prop:compactgenerators} Let $S$ be a qcqs scheme. Then $\SH(S)$ is generated under colimits by the essentially small family of compact objects 
	$$E_S = \{\Sigma^{-2n,-n}\Sigma^\infty_+ X \mid X\in\Sm_S, n\in\N\}.$$
\end{thm}

\begin{Prop}\label{Prop:6FFforSHomega} Let $S$ be a qcqs scheme, and let $f:X\to Y$ be a morphism in $\Sm_S$. 
	\begin{enumerate}[label = (\roman*)]
		\item The full subcategory of compact objects $\SH(X)^\omega \subseteq \SH(X)$ is the thick subcategory generated by the family $E_X$ from \Cref{Prop:compactgenerators}. It is a symmetric monoidal subcategory, and is essentially small.
		\item The pullback $f^\ast: \SH(Y) \to \SH(X)$ from \Cref{Prop:6FFforSH}$(ii)$ restricts to $f^\ast: \SH(Y)^\omega \to \SH(X)^\omega$.
		\item If furthermore $f$ is smooth, the map $f_\sharp : \SH(X)\to \SH(Y)$ from \Cref{Prop:6FFforSH}$(iv)$ restricts to $f_\sharp : \SH(X)^\omega\to \SH(Y)^\omega$.
		\item Let $A,B\in\Sm_S$. The equivalence $\SH(A\amalg B) \simeq \SH(A) \times \SH(B)$ from \Cref{Prop:6FFforSH}$(i)$ restricts to an equivalence $\SH(A\amalg B)^\omega \simeq \SH(A)^\omega \times \SH(B)^\omega$.
	\end{enumerate}
\end{Prop}
\begin{proof}
The first part of $(i)$ follows from the fact that, by \Cref{Prop:toolscompact}$(i)$ and \Cref{Prop:compactgenerators}, $\SH(X)^\omega$ is the thick subcategory generated by the essentially small family $E_X$. To prove that $\SH(X)^\omega \subseteq \SH(X)$ is a symmetric monoidal subcategory, by \cite[Lem.\ 5.4.2.4]{Lurie-HTT} it suffices to show that it is closed under the tensor product and contains the unit. For the latter, note that the unit even belongs to $E_X$. The former follows from the description $\SH(X)^\omega = \mathsf{thick}(E_X)$, the fact that $E_X$ is closed under the tensor product, and the fact that the tensor product preserves colimits in both variables. This proves part $(i)$.

To prove $(ii)$, by \Cref{Prop:toolscompact}, it suffices to show that $f^\ast$ maps the compact generators of \Cref{Prop:compactgenerators} for $\SH(Y)$ to compact objects in $\SH(X)$. For any $U\in \Sm_Y$, by \Cref{Prop:6FFforSH}$(ii)$-$(iii)$ we have $f^\ast((\mathbb{P}_Y^1)^{\Smash -n} \Smash \Sigma^\infty_+ U) \simeq (\mathbb{P}_X^1)^{\Smash -n} \Smash \Sigma^\infty_+ f^\ast(U) \in E\subseteq \SH(X)^\omega$, as needed. 

Part $(iii)$ is proven similarly: for any $U\in \Sm_X$, by \Cref{Prop:6FFforSH}$(iv)$-$(v)$, we have 
$$f_\sharp((\mathbb{P}_X^1)^{\Smash -n} \Smash \Sigma^\infty_+ U) \simeq f_\sharp((f^\ast\mathbb{P}_Y^1)^{\Smash -n} \Smash \Sigma^\infty_+ U) \simeq (\mathbb{P}_Y^1)^{\Smash -n} \Smash \Sigma^\infty_+ f_\sharp U.$$

Finally, part $(iv)$ is a particular case of \Cref{Prop:productcompact}. 
\end{proof}
\vspace{0.2cm}

\subsubsection{Construction of the motivic Thom spectrum functor}\label{Subsubsect:constrmotThom}\hfill\vspace{0.2cm}

As mentioned before, we want to view $\SH^\omega$ as a spherical presheaf of essentially small symmetric monoidal categories, inducing a spherical presheaf of spaces $\SH^\wgpd: \Sm_S^\op \to \CAlg(\Spc^\times)$. One could perform the same construction with $\SH^\kappa$ instead (also see \Cref{Rmk:whykappa} below).

\begin{Def}[\protect{\cite[Def.\ 5.5.8.8]{Lurie-HTT}}]\label{Def:sphericalpresheaves} Let $\mathcal{C}$ be an essentially small category with finite coproducts. Then, the category of \emph{spherical presheaves (of spaces) on $\mathcal{C}$}, denoted by $\Pre_\Sigma(\mathcal{C})$, is the category of functors $\mathcal{C}^\op \to \spaces$ preserving finite products (i.e., sending finite coproducts in $\mathcal{C}$ to products in $\spaces$). This construction defines a functor $\Pre_\Sigma(\bullet)$ from essentially small categories to sifted-cocomplete categories, which has a partial right adjoint given by the inclusion of essentially small sifted-cocomplete categories. In other terms, $\Pre_\Sigma(\mathcal{C})$ is the free sifted-cocomplete category generated by $\mathcal{C}$ \cite[Prop.\ 5.5.8.15]{Lurie-HTT}, and every spherical presheaf is a sifted colimit of representable presheaves \cite[Lem.\ 5.5.8.14]{Lurie-HTT}, which are in particular spherical. 
\end{Def} 

The motivic Thom spectrum functor will be defined as a particular instance of the more general motivic colimit functors. Given $\calF$ a spherical presheaf of essentially small symmetric monoidal categories on $\Sm_{S}$, with good functoriality properties inspired by the ones of $\SH$ studied in the previous subsection, we will construct symmetric monoidal functors $\Pre_\Sigma(\Sm_{S'})_{/\calF^\simeq} \to \calF(S')$ for all $S'\in\Sm_{S}$. The construction of these motivic colimit functors is itself functorial in $S'$, in the sense that we will view both the source and target categories as presheaves of (not necessarily small) symmetric monoidal categories on $\Sm_{S}$, and construct a natural symmetric monoidal transformation between these presheaves. 

Applying the construction to $\calF = \SH^\omega$ produces the motivic Thom spectrum functor, but we will also apply the construction to another presheaf $\calR$ related to real realization (\Cref{Subsubsect:colimfunctforrealization}) in order to compare the motivic Thom spectrum functor with the topological one (\Cref{Subsubsect:comp}). 

Before proceeding to the construction itself, we need another point of view on spherical presheaves of symmetric monoidal categories (on $\Sm_S$). As we now explain, the latter may be viewed as functors $\Span(\Sm_S, \all, \fold) \to \Cat$ on a certain category of spans, that preserve finite products. Recall that the objects in $\Cat$ are \emph{essentially small} categories.

\begin{Def}[\protect{\cite[Section 5]{Barwick-mackey}}]\label{Def:Span} Let $\calC$ be a category and let $\mathsf{left}$ and $\mathsf{right}$ be two classes of edges in $\calC$, both containing equivalences and closed under pullback along one another. Then, there is a category $\Span(\mathcal{C},\mathsf{left}, \mathsf{right})$ with vertices the objects of $\mathcal{C}$, and for any $X,Z\in\calC$, the edges from $X$ to $Z$ are given by all spans $X\leftarrow Y \rightarrow Z$ where $Y\in\calC$, and $(Y\to X)\in\mathsf{left}$, $(Y\to Z)\in\mathsf{right}$. The composition of two spans $V \leftarrow W \to X$ and $X \leftarrow Y \to Z$ is the span $V \leftarrow W\times_X Y \to Z$.
\end{Def}

\begin{Def}\label{Def:finstar} Let $\Fin$ be the category of finite sets, and $\Finstar$ be the category of finite pointed sets and pointed maps. For all $n\in\N$, let $\finstar{n} := \{\ast, 1,\dots,n\} \in\Finstar$, and $\fin{n} := \{1,\dots,n\} \in\Fin$.
\end{Def}

\begin{Prop}[\protect{\cite[Prop.\ C.1]{BH-norms}}]\label{Prop:presheavesandspans} Let $\calD$ be a category with finite products, and let $\all$ be the class of all edges in $\calD$. Consider the functor $\theta: \Fun(\Span(\Fin,\all,\all),\calD) \longrightarrow \Fun(\Finstar,\calD)$ induced by restriction along the functor $\Finstar \to \Span(\Fin, \all, \all)$ sending $\finstar{n}$ to  $\fin{n}$, and a map $(f: \finstar{n} \to \finstar{m})$ to $(\fin{n} \xhookrightarrow{} f^{-1}(\fin{m}) \to \fin{m})$. Then, $\theta$ restricts to an equivalence of categories $ \Fun^\times(\Span(\Fin,\all,\all),\calD) \lsimeq{} \CAlg(\calD^\times)$,
where $\Fun^\times$ denotes the full category of functors preserving finite products.
\end{Prop}

This statement generalizes to presheaves with values in $\CAlg(\calD^\times)$. We first need some notation.

\begin{Notation}\label{notation:for-prop-spanvspsh} Let $\calC$ be an extensive category $\mathcal{C}$ (e.g.\ $\Sm_S$, see \cite[Def.\ 2.3]{BH-norms}). Consider the functor
\begin{align*}
\chi: \mathcal{C}^\op \times \Span(\mathsf{Fin},\all,\all) &\longrightarrow \Span(\mathcal{C},\all, \fold) \\
(c,\fin{n}) &\longmapsto c^{\amalg n}\\
(c'\xrightarrow{f^\op} c, \fin{n}\xleftarrow{a} \fin{m} \xrightarrow{b} \fin{k}) &\longmapsto (c'^{\amalg n} \leftarrow c^{\amalg m} \to c^{\amalg k})
\end{align*}
where the restriction of $c^{\amalg m} \to c^{\amalg k}$ to the $i$-th component is the inclusion as the $b(i)$-th component, and the restriction of $c^{\amalg m} \to (c')^{\amalg }$ to the $i$-th component is given by $f$ followed by the inclusion as the $a(i)$-th component.
\end{Notation}

\begin{Prop}[\protect{\cite[Prop.\ C.5]{BH-norms}}]\label{Prop:spanvspsh} In the situation of \Cref{notation:for-prop-spanvspsh}, for $\mathcal{D}$ a category with finite products, restriction along $\chi$ induces a functor $\Theta: \mathsf{Fun}(\Span(\mathcal{C},\all, \fold), \mathcal{D}) \to \mathsf{Fun}(\mathcal{C}^\op \times \Span(\mathsf{Fin},\all,\all), \mathcal{D})$, which restricts to an equivalence $\mathsf{Fun}^\times(\Span(\mathcal{C},\all, \fold), \mathcal{D}) \lsimeq{} \Fun^\times(\mathcal{C}^\op,\CAlg(\mathcal{D}^\times))$.
\end{Prop}

Using this point of view, we are finally ready for the construction of our motivic colimit functors.

\begin{Notation}\label{Notation:Span} When working over a fixed base scheme $S$, denote $\Span := \Span(\Sm_S,\all,\fold)$.
\end{Notation}

\begin{thm}[\protect{\cite[\S 16.3]{BH-norms}}]\label{Prop:motiviccolim} Let $\calF: \Span \longrightarrow \Cat$ be a functor preserving finite products, such that:
\begin{enumerate}[label = (\roman*)]
    \item for any $f:Y\to X$ in $\Sm_S$ \emph{smooth}, $f^\ast := \calF(X\leftarrow Y \xlongequal{\ } Y) : \calF(X) \to \calF(Y)$ admits a left adjoint $f_\sharp$.
    \item for any Cartesian square in $\Sm_S$
\[\begin{tikzcd}[ampersand replacement=\&]
	{X'} \& X \\
	{Y'} \& Y
	\arrow["{g'}", from=1-1, to=1-2]
	\arrow["{f'}"', from=1-1, to=2-1]
	\arrow["\lrcorner"{anchor=center, pos=0.125}, draw=none, from=1-1, to=2-2]
	\arrow["f", from=1-2, to=2-2]
	\arrow["g", from=2-1, to=2-2]
\end{tikzcd}\]
    with $f$, $f'$ smooth morphisms, the exchange transformation $\mathsf{Ex}_{\sharp}^\ast: f'_\sharp(g')^\ast \Longrightarrow g^\ast f_\sharp$ of functors $\calF(X) \to \calF(Y')$ (defined as in \Cref{Prop:6FFforSH}) is an equivalence.
    \item Given $(\nabla: Y \to Z)\in\fold$, let $\nabla_\otimes := \calF(Y\xlongequal{\ } Y \to Z) : \calF(Y) \to \calF(Z)$. Then $\nabla$ encodes the tensor product on the symmetric monoidal category $\calF(Z)$ (\Cref{Prop:spanvspsh}). For every diagram in $\Sm_S$ of the form
\[\begin{tikzcd}
	W & {R_{Y/X}(W\times_X Y)\times_Z Y} & {R_{Y/X}(W\times_X Y)} \\
	X & Y & Z
	\arrow["g"', from=1-1, to=2-1]
	\arrow["{f'}"', from=1-2, to=1-1]
	\arrow["{\nabla'}", from=1-2, to=1-3]
	\arrow["{u'}", from=1-2, to=2-2]
	\arrow["\lrcorner"{anchor=center, pos=0.125}, draw=none, from=1-2, to=2-3]
	\arrow["u", from=1-3, to=2-3]
	\arrow["f"', from=2-2, to=2-1]
	\arrow["\nabla", from=2-2, to=2-3]
\end{tikzcd}\]
    with $u$ and $u'$ smooth morphisms and $\nabla, \nabla' \in \fold$, the distributivity transformation 
	$${\mathsf{Dis}_{\sharp\otimes} : u_\sharp\nabla'_\otimes(f')^\ast \Longrightarrow \nabla_\otimes(\pi_Y)_\sharp\pi_W^\ast}$$
	of functors $\calF(W) \to \calF(Z)$ (defined as in \Cref{Prop:6FFforSH}) is an equivalence. 
\end{enumerate}
\vspace{0.2cm}
Then, there exists a natural transformation $M_\calF: (\Sm_\bullet)_{/{\calF}^\simeq} \to \calF$ of finite products preserving functors $\Span \to \Cat$  (i.e., of spherical presheaves of symmetric monoidal categories). 

Assume furthermore that $\calF$ is a subfunctor of $\calG : \Span \to \widehat{\mathsf{Cat}}_\infty$, a spherical presheaf of non-necessarily small symmetric monoidal categories. Then, if $\calG$ lifts to the category of (not necessarily small) sifted-cocomplete categories and functors preserving sifted colimits, then we obtain a natural transformation $M_\calF: \Pre_\Sigma(\Sm_\bullet)_{/\calF^\simeq} \to \calG$ of spherical of presheaves of (not necessarily small) sifted-cocomplete symmetric monoidal categories. 
\end{thm}

\begin{Rmk}\label{Rmk:actualslice}
    In the notation $\Pre_\Sigma(\Sm_\bullet)_{/\calF^\simeq}$, $\calF^\simeq$ is viewed as a functor associating to $(u: X\to S)\in\Sm_S$ a presheaf of spaces on $\Sm_X$, given by 
    $$\Sm_X^\op \xhookrightarrow{\, u_\sharp \,} \Sm_S^\op \xrightarrow{\, \calF\, } \Cat \xrightarrow{(-)^\simeq} \spaces.$$
    Thus, $\Pre_\Sigma(\Sm_\bullet)_{/\calF^\simeq}$ associates to $X$ the slice category $\Pre_\Sigma(\Sm_X)_{/\calF^\simeq}$, viewing $\calF^\simeq \in \Pre_\Sigma(\Sm_X)$. It is a priori unclear whether the symmetric monoidal structure on the slice category induced by this construction agrees with the one in \Cref{Prop:smstconslice}. We will come back to this question in \Cref{Subsubsect:comp}.
\end{Rmk}

\begin{Notation}\label{Notation:smstc}
    We denote the symmetric monoidal structure on $\Pre_\Sigma(\Sm_\R)_{/\SH^\wgpd}$ induced by \Cref{Prop:motiviccolim} by $\Pre_\Sigma(\Sm_\R)^{\otimes'}_{/\SH^\wgpd}$. By \Cref{Prop:smstconslice}, the underlying slice category admits another symmetric monoidal structure. We denote this second monoidal structure by $\Pre_\Sigma(\Sm_\R)^{\otimes}_{/\SH^\wgpd}$. A priori, we do not know if they agree; we will come back to this question right after the proof of \Cref{Prop:motiviccolim}.
\end{Notation}

\begin{proof}[Proof of \Cref{Prop:motiviccolim}] For completeness, we repeat the construction in \cite[\S 16.3]{BH-norms} in our more general context. By straightening-unstraightening \cite[\S 3.2]{Lurie-HTT}, the functor $\calF$ classifies a Cartesian fibration $F: \calE \to \Span^\op$. Let $\Fun_{\Sm}(\Delta^1,\Span)$ be the full subcategory of $\Fun(\Delta^1,\Span)$ consisting of functors whose image is a span of the form $X\leftarrow Y \xlongequal{\ } Y$ with $Y \to X$ a smooth morphism in $\Sm_S$. The composition
$$\Fun_{\Sm}(\Delta^1,\Span)\times \Delta^1 \xrightarrow{\text{ev}} \Span \xrightarrow{\calF} \Cat$$ 
where the first functor is evaluation, can be viewed as a natural transformation between the functors $\calF\circ s$ and $\calF\circ t$, where $s$ (resp.\ $t$) is the source (resp.\ target) functor $\Fun_{\Sm}(\Delta^1,\Span)\to \Fun_{\Sm}(\Delta^0,\Span)\simeq \Span$ induced by the inclusion $\{0\} \to \Delta^1$ (resp.\ $\{1\} \to \Delta^1$). Since precomposition of functors corresponds to pullbacks of Cartesian fibrations \cite[Def. 3.3.2.2]{Lurie-HTT}, the functors $\calF\circ s$ and $\calF\circ t$ classify the Cartesian fibrations $s^\ast F$ and $t^\ast F$, respectively, defined as the pullbacks
\[\begin{tikzcd}
	{s^\ast\calE} & \calE && {t^\ast\calE} & \calE \\
	{\Fun_{\Sm}(\Delta^1,\Span)^\op} & {\Span^\op} && {\Fun_{\Sm}(\Delta^1,\Span)^\op} & {\Span^\op,}
	\arrow[from=1-1, to=1-2]
	\arrow["{s^\ast F}"', from=1-1, to=2-1]
	\arrow["\lrcorner"{anchor=center, pos=0.125}, draw=none, from=1-1, to=2-2]
	\arrow["F", from=1-2, to=2-2]
	\arrow[from=1-4, to=1-5]
	\arrow["{t^\ast F}"', from=1-4, to=2-4]
	\arrow["\lrcorner"{anchor=center, pos=0.125}, draw=none, from=1-4, to=2-5]
	\arrow["F", from=1-5, to=2-5]
	\arrow["{s^\op}", from=2-1, to=2-2]
	\arrow["{t^\op}", from=2-4, to=2-5]
\end{tikzcd}\]
and the natural transformation $\calF\circ s \Rightarrow \calF\circ t$ gives a morphism of Cartesian fibrations $\phi: s^\ast\calE \to t^\ast\calE$ (i.e., a functor over $\Fun_{\Sm}(\Delta^1,\Span)^\op$ that preserves Cartesian edges). We have a commutative diagram
\[\begin{tikzcd}[ampersand replacement=\&]
	{t^\ast\calE} \& {s^\ast\calE} \& \calE \\
	\& {\Fun_{\Sm}(\Delta^1,\Span)^\op} \& {\Span^\op.}
	\arrow["{t^\ast F}"', from=1-1, to=2-2]
	\arrow["\phi"', from=1-2, to=1-1]
	\arrow["\chi", from=1-2, to=1-3]
	\arrow["{s^\ast F}"', from=1-2, to=2-2]
	\arrow["F", from=1-3, to=2-3]
	\arrow["{s^\op}", from=2-2, to=2-3]
\end{tikzcd}\]

We now pause to give an outline of the rest of the proof. Our goal is to show that the composition $s^\op \circ t^\ast F: t^\ast\calE \to \Fun_{\Sm}(\Delta^1,\Span)^\op \to \Span^\op$ is a Cartesian fibration, classified by a presheaf of symmetric monoidal categories $(\Sm_\bullet)_{\sslash\calF}$, of which $(\Sm_\bullet)_{/\calF}$ is a subfunctor; and to construct a morphism of Cartesian fibrations $t^\ast\calE \to \calE$. The latter corresponds to the natural transformation $M_\calF$ we are looking for. Here are the steps we will follow:
\begin{enumerate}
    \item Show that $s^\op$ is a Cartesian fibration, so that the composition $s^\op \circ t^\ast F: t^\ast\calE \to \Span^\op$ is a Cartesian fibration (note that $t^\ast F$ is a Cartesian fibration because it is a pullback of $F$).
    \item Define a \emph{relative left adjoint} $\psi$ to $\phi$, which will in particular be a map $t^\ast\calE \to s^\ast\calE$.
    \item Show that the composition $\chi\circ \psi$ is a morphism of Cartesian fibrations over $\Span^\op$, and thus corresponds to a natural transformation between the functors classifying these Cartesian fibrations.
    \item Show that $s^\op \circ t^\ast F$ is indeed classified by some presheaf $(\Sm_\bullet)_{\sslash\calF}$ admitting $(\Sm_\bullet)_{/\calF}$ as a subfunctor, and deduce the first part of the statement of the theorem.
    \item Prove the second part of the statement, i.e., the extension to sifted-cocompletions.
\end{enumerate}
\vspace{0.2cm}

\textbf{Step 1:} \emph{The functor $s^\op$ is a Cartesian fibration.} Equivalently, we show that $s$ is a coCartesian fibration. Firstly, it is an inner fibration because the source functor $\Fun(\Delta^1,\Span) \to \Span$ is an inner fibration by \cite[Cor.\ 2.4.7.11]{Lurie-HTT}, and the restriction of an inner fibration to a full subcategory is still an inner fibration by \cite[\href{https://kerodon.net/tag/01CU}{Tag 01CU}]{Kerodon}. Secondly, we claim that a coCartesian edge over $X\xleftarrow{f} Y \xrightarrow{\nabla} Z$ in $\Span$ with source $X \xleftarrow{j} W \xlongequal{\ } W$ is given by the edge in $\Fun_{\Sm}(\Delta^1,\Span)$ (i.e., natural transformation)
\[\begin{tikzcd}[ampersand replacement=\&]
	W \& {R_{Y/Z}(W\times_X Y)\times_Z Y} \& {R_{Y/Z}(W\times_X Y)} \\
	W \& {R_{Y/Z}(W\times_X Y)\times_Z Y} \& {R_{Y/Z}(W\times_X Y)} \\
	X \& Y \& Z,
	\arrow[equals, from=1-1, to=2-1]
	\arrow[from=1-2, to=1-1]
	\arrow[from=1-2, to=1-3]
	\arrow[equals, from=1-2, to=2-2]
	\arrow[equals, from=1-3, to=2-3]
	\arrow["j"', from=2-1, to=3-1]
	\arrow["\lrcorner"{anchor=center, pos=0.125, rotate=180}, draw=none, from=2-2, to=1-1]
	\arrow["{f'}"', from=2-2, to=2-1]
	\arrow["{\nabla'}", from=2-2, to=2-3]
	\arrow[from=2-2, to=3-2]
	\arrow["\lrcorner"{anchor=center, pos=0.125}, draw=none, from=2-2, to=3-3]
	\arrow[from=2-3, to=3-3]
	\arrow["f"', from=3-2, to=3-1]
	\arrow["\nabla", from=3-2, to=3-3]
\end{tikzcd}\]
which we call $e$ (for a grid like this to define a natural transformation between the vertical spans on the left and right-hand sides, we need the top left and bottom right squares to be pullbacks, which is the case here). Indeed, given another edge $e'$ with the same source
\[\begin{tikzcd}
	W & B & D \\
	W & B & D \\
	X & A & C
	\arrow[equals, from=1-1, to=2-1]
	\arrow[from=1-2, to=1-1]
	\arrow[from=1-2, to=1-3]
	\arrow[equals, from=1-2, to=2-2]
	\arrow[equals, from=1-3, to=2-3]
	\arrow["j"', from=2-1, to=3-1]
	\arrow["\lrcorner"{anchor=center, pos=0.125, rotate=180}, draw=none, from=2-2, to=1-1]
	\arrow["{g'}"', from=2-2, to=2-1]
	\arrow["{\widetilde{\nabla}'}", from=2-2, to=2-3]
	\arrow[from=2-2, to=3-2]
	\arrow["\lrcorner"{anchor=center, pos=0.125}, draw=none, from=2-2, to=3-3]
	\arrow[from=2-3, to=3-3]
	\arrow["g"', from=3-2, to=3-1]
	\arrow["{\widetilde{\nabla}}", from=3-2, to=3-3]
\end{tikzcd}\]
and a commutative diagram in $\Span$
\[\begin{tikzcd}
	X & Y & Z \\
	& A & E \\
	&& C
	\arrow["f"', from=1-2, to=1-1]
	\arrow["\nabla", from=1-2, to=1-3]
	\arrow["g", from=2-2, to=1-1]
	\arrow["{\widetilde{\nabla}}"', from=2-2, to=3-3]
	\arrow["h"', from=2-3, to=1-3]
	\arrow["{\overline{\nabla}}", from=2-3, to=3-3]
\end{tikzcd}\]
(in particular $A\simeq Y\times_Z E$), the following edge $e''$ in $\Fun_{\Sm}(\Delta^1,\Span)$ lifts $Z\leftarrow E\to C$ and provides a commutative diagram $e''\circ e \simeq e'$
\[\begin{tikzcd}
	{R_{Y/Z}(W\times_X Y)} & {D\times_C E} & D \\
	{R_{Y/Z}(W\times_X Y)} & {D\times_C E} & D \\
	Z & E & C.
	\arrow[equals, from=1-1, to=2-1]
	\arrow[from=1-2, to=1-1]
	\arrow[from=1-2, to=1-3]
	\arrow[equals, from=1-2, to=2-2]
	\arrow[equals, from=1-3, to=2-3]
	\arrow[from=2-1, to=3-1]
	\arrow["{h'}"', from=2-2, to=2-1]
	\arrow[from=2-2, to=2-3]
	\arrow[from=2-2, to=3-2]
	\arrow["\lrcorner"{anchor=center, pos=0.125}, draw=none, from=2-2, to=3-3]
	\arrow[from=2-3, to=3-3]
	\arrow["h"', from=3-2, to=3-1]
	\arrow["{\overline{\nabla}}", from=3-2, to=3-3]
\end{tikzcd}\]
To construct $h'$, note that by definition of the Weil restriction, there are equivalences
\begin{align*}
    \mathsf{Sch}_Z(D\times_C E,R_{Y/Z}(W\times_X Y)) &\simeq \mathsf{Sch}_Y((D\times_C E)\times_Z Y,W\times_X Y)\\
    &\simeq \mathsf{Sch}_Y(D\times_C A,W\times_X Y)\\
    &\simeq \mathsf{Sch}_Y(B,W\times_X Y).
\end{align*}
Then, under these equivalences, $h'$ corresponds to the map in $\mathsf{Sch}_Y(B,W\times_X Y)$ induced by $g':B\to W$. The bottom right square in $e''$ is determined by the requirement of it being a pullback, and the map $h'$ is determined, under the equivalences we just saw, by the condition $e''\circ e = e'$, so that the lift $e''$ is essentially unique.\\

\textbf{Step 2:} \emph{The functor $\phi$ has a relative left adjoint.} This means that there exists a functor $\psi$ in the other direction, and a transformation $\varepsilon: \psi\phi \to \id{}$ exhibiting $\psi$ as a left adjoint to $\phi$, and moreover $\psi$ commutes with the structure maps, in the sense that $s^\ast F \circ \psi \simeq t^\ast F \circ \phi \circ \psi \Rightarrow t^\ast F$ is an equivalence \cite[Def. 7.3.2.2]{Lurie-HA}. By \cite[Prop.\ 7.3.2.6]{Lurie-HA}, since $\phi$ is a morphism of Cartesian fibrations, it suffices to show that $\phi$ has fiberwise a left-adjoint. Over $(X\leftarrow Y \xlongequal{\ } Y)\in\Fun_{\Sm}(\Delta^1,\Span)$ (so $f:Y\to X$ is smooth), the functor $\phi$ is by definition the corresponding component of the natural transformation $\calF \circ s \to \calF\circ t$ it encodes. Therefore, it is given by $f^\ast: \calF(X) \to \calF(Y)$, which admits a left adjoint $f_\sharp$ by assumption $(i)$. We thus obtain a relative left adjoint $\psi: t^\ast \calE \to s^\ast \calE$.\\

 \textbf{Step 3:} \emph{The composition $\chi \circ \psi: t^\ast\calE \to s^\ast\calE \to \calE$ is a morphism of Cartesian fibrations over $\Span^\op$.} Compatibility with the structure maps down to $\Span^\op$ holds by construction. So we have to show that $\chi\circ \psi$ preserves Cartesian edges. By \cite[\href{https://kerodon.net/tag/01UL}{Tag 01UL}]{Kerodon}, for Cartesian fibrations $p:\calE\to \calC$ and $q:\calC\to\calD$, an edge $e$ in $\calE$ is $(q\circ p)$-Cartesian if and only if $e$ is $p$-Cartesian and $p(e)$ is $q$-Cartesian. In our situation, if $e$ is an $(s^\op\circ t^\ast F)$-Cartesian edge, $e$ is $(t^\ast F)$-Cartesian and $(t^\ast F)(e)$ is $s^\op$-Cartesian. It follows that $(t^\ast F)(e)$ is the opposite of an edge in $\Fun_{\Sm}(\Delta^1,\Span)$ of the form
\begin{equation}\label{Diagram:edgeinFunsm}\begin{tikzcd}
	W & {Y'} & {Z'} \\
	W & {Y'} & {Z'} \\
	X & Y & Z
	\arrow[Rightarrow, no head, from=1-1, to=2-1]
	\arrow[from=1-2, to=1-1]
	\arrow[from=1-2, to=1-3]
	\arrow[Rightarrow, no head, from=1-2, to=2-2]
	\arrow[Rightarrow, no head, from=1-3, to=2-3]
	\arrow["g"', from=2-1, to=3-1]
	\arrow["{f'}"', from=2-2, to=2-1]
	\arrow["{\nabla'}", from=2-2, to=2-3]
	\arrow["{u'}", from=2-2, to=3-2]
	\arrow["\lrcorner"{anchor=center, pos=0.125}, draw=none, from=2-2, to=3-3]
	\arrow["u", from=2-3, to=3-3]
	\arrow["f", from=3-2, to=3-1]
	\arrow["\nabla"', from=3-2, to=3-3]
\end{tikzcd}\end{equation}
where $Y' = R_{Y/Z}(W\times_X Y)\times_Z Y$ and $Z'= R_{Y/Z}(W\times_X Y)$. Here $g$, $u$, and $u'$ are smooth. Since $e$ is a $(t^\ast F)$-Cartesian lift of this edge, it is of the form $\alpha : (Z \leftarrow Z' \xlongequal{\ } Z', H) \to (X\leftarrow W\xlongequal{\ } W, E)$ with $H\in\calF(Z')$, $E\in \calF(W)$, and $\alpha$ consists in the data of $(t^\ast F)(e)$ (the diagram above) together with an equivalence $H \lsimeq{} \nabla'_\otimes f'^\ast E$. Then, $\chi\psi(e)$ is an edge $(Z,u_\sharp(H)) \to (X,g_\sharp(E))$ given by the morphisms $(X\leftarrow Y \to Z)^\op$ and $u_\sharp H \to \nabla_\otimes f^\ast g_\sharp E$, where the latter is the composite
$$u_\sharp H \xrightarrow{\qquad} u_\sharp \nabla'_\otimes f'^\ast E \xrightarrow{\ \mathsf{Dis}_{\sharp,\otimes}\ } \nabla_\otimes (\pi_Y)_\sharp \pi_W^\ast E \xrightarrow{\ \ \mathsf{Ex}_\sharp^\ast\ \ } \nabla_\otimes f^\ast g_\sharp E,$$
with $\pi_Y: Y\times_X W \to Y$ and $\pi_W: Y\times_X W \to W$ the projections. We saw above that the first map in the composition was an equivalence, and the two other ones are equivalences by assumptions $(iii)$ and $(ii)$ in the statement respectively. Therefore, $\chi\psi(e)$ is $F$-Cartesian, as desired. \\

\textbf{Step 4:} \emph{We deduce the first part of the statement.} The morphism of Cartesian fibrations $\chi \psi$ corresponds to a natural transformation $M_\calF$ of functors $\Span \to \Cat$ between the functors classified by $s^\op \circ t^\ast F$ and $F$, respectively. The former associates to $X\in\Sm_S$ the fiber of $s^\op \circ t^\ast F$ over $X$. The fiber of $s^\op$ over $X$ can be identified with $\Sm_X$; and over an object $Y$ in $\Sm_X$, the fiber of $t^\ast F$ is by definition $\calF(Y)$. Therefore, the functor classified by $s^\op \circ t^\ast F$ may be denoted by $(\Sm_\bullet)_{\sslash\calF}$: it associates to $X\in\Sm_S$ the category $(\Sm_X)_{\sslash\calF}$ of pairs $(U,E)$ with $U\in\Sm_X$ and $E\in \calF(U)$, and a morphism $(U,E) \to (U',E')$ is the data of $f: U\to U'$ in $\Sm_X$ and a map $E \to \calF(f^\op)(E')$ in $\calF(U)$ (\emph{not} required to be an equivalence). These categories splice into a presheaf as follows. Given a morphism $X\xleftarrow{f} Y \xrightarrow{\nabla} Z$ in $\Span$, consider the Cartesian lift built in Step 1. Then $X\leftarrow Y \rightarrow Z$ is mapped to the functor $(\Sm_X)_{\sslash\calF} \to (\Sm_Z)_{\sslash\calF}$ sending a pair $(W,E)$ with $W\in\Sm_X$ and $E\in\calF(X')$ to the pair consisting of $R_{Y/Z}(X'\times_X Y)\in\Sm_Z$ and $\nabla'_\otimes f'^\ast E \in\calF(R_{Y/Z}(X'\times_X Y))$. The slice category $(\Sm_X)_{/\calF^\simeq}$ is viewed as the wide subcategory of $(\Sm_X)_{\sslash\calF}$ where the maps $E \to \calF(f^\op)(E')$ in $\calF(U)$ in the description above are required to be equivalences. Then, $(\Sm_\bullet)_{/\calF^\simeq}$ forms a subpresheaf of $(\Sm_\bullet)_{\sslash\calF}$, to which the natural transformation $M_\calF$ restricts. This proves the first part of the statement. \\

\textbf{Step 5:} \emph{We prove the second part of the statement.} If $\calF$ embeds into a presheaf of not necessarily small symmetric monoidal categories $\calG$ that lifts (as a functor with source category $\Span$) to sifted-cocomplete categories, we have a composite morphism of presheaves (which can also be viewed as a natural transformation of functors on $\Span$) $(\Sm_\bullet)_{/\calF^\simeq} \to \calF \to \calG$. Since $\calG$ lifts to sifted-cocomplete categories, by the universal property of the sifted-cocompletion (see \Cref{Def:sphericalpresheaves}), this morphism left Kan extends to a transformation $\Pre_\Sigma((\Sm_\bullet)_{/\calF^\simeq}) \simeq \Pre_\Sigma((\Sm_\bullet))_{/\calF^\simeq} \to \calF \to \calG$ (the equivalence on the left hand-side is \Cref{Prop:sliceofpsh}), where $\Pre_\Sigma((\Sm_\bullet)_{/\calF^\simeq})$ corresponds to the composition
\[\begin{tikzcd}[column sep = 4em]
	\Span & \Cat & {\widehat{\mathsf{Cat}}_\infty^\mathsf{sift}.}
	\arrow["{(\Sm_\bullet)_{/\calF^\simeq}}", from=1-1, to=1-2]
	\arrow["{\Pre_\Sigma(\bullet)}", from=1-2, to=1-3]
\end{tikzcd}\]
This is the desired morphism of presheaves of symmetric monoidal categories $M_\calF: \Pre_\Sigma(\Sm_\bullet)_{/\calF^\simeq} \to \calG$. 
\end{proof}

\begin{Def}\label{Def:motiviccolim} Let $\calF$ be a presheaf on $\Sm_S$ as in \Cref{Prop:motiviccolim}. The \emph{motivic colimit functor associated with $\calF$} is the $S$-component of the transformation $M_\calF$. Abusing notations, we also denote it by $M_\calF$.
\end{Def}

\Cref{Notation:smstc} gives rise to the following question.
\begin{Question}\label{QuestionA}
    In the situation of \Cref{Prop:motiviccolim}, do the symmetric monoidal structures $\Pre_\Sigma(\Sm_S)^{\otimes'}_{/\calF}$ (obtained from the construction of \Cref{Prop:motiviccolim}, i.e., from \cite[\S 16.3]{BH-norms}) and $\Pre_\Sigma(\Sm_S)^{\otimes}_{/\calF}$ (described in \Cref{Prop:smstconslice}) agree? 
\end{Question}
We conjecture that this question can be answered by the affirmative, but we were unable to prove it. We prove a weaker result in \Cref{Prop:Thomagreeassmfunctors}. Since both symmetric monoidal structures are by construction Day convolution with respect to some symmetric monoidal structures on the slice $(\Sm_S)_{/\calF}$, it would suffice to prove that the structures $\otimes$ and $\otimes'$ agree at this level. Here is a hint in the direction of a positive answer to \Cref{QuestionA}: 
\begin{Prop}\label{Prop:tensoragree}
    In the situation of \Cref{Prop:motiviccolim}, the tensor product of any two objects for the symmetric monoidal structures $\Pre_\Sigma(\Sm_S)^{\otimes'}_{/\calF}$ and $\Pre_\Sigma(\Sm_S)^{\otimes}_{/\calF}$ agree.
\end{Prop}
\begin{proof}
    Let $g:X \to \calF$ and $h:Y \to \calF$ be two objects in $(\Sm_S)_{/\calF}$, represented by $x\in\calF(X)$ and $y\in\calF(Y)$, respectively. By the informal description of the tensor product in $\Pre_\Sigma(\Sm_S)^{\otimes}_{/\calF}$ in \Cref{Prop:smstconslice}, we have
    $$(X \to \calF) \otimes (Y \to \calF) = X \times Y \xrightarrow {\ g\times h\ }\calF\times \calF \xrightarrow{\ \ \mu\ \ } \calF$$
    where $\mu$ gives the multiplicative structure on the presheaf of symmetric monoidal categories $\calF$. The resulting map $X \times Y \longrightarrow \calF$ is represented by the element $\calF(\pi_X^\op)(x) \otimes_{X\times Y} \calF(\pi_Y^\op)(y) \in \calF(X\times Y)$, where for $Z\in\Sm_\R$, the functor $\otimes_{Z}$ is the tensor product on $\calF(Z)$, namely $\calF(Z^{\amalg 2}\longeq{} Z^{\amalg 2} \longrightarrow Z)$ (viewing $\calF$ as a functor $\Span \to \Cat$ preserving finite products).

    On the other hand, recall that the presheaf of symmetric monoidal categories $\Pre_\Sigma(\Sm_\bullet)^{\otimes'}_{/\calF}$ is classified by the Cartesian fibration $s^\op \circ t^\ast F : t^\ast \calE \to \Span^\op$ (in the notation of the proof of \Cref{Prop:motiviccolim}). Thus, the tensor product of $g$ and $h$ is the source of an $(s^\op \circ t^\ast F)$-Cartesian lift for the edge $(S^{\amalg {2}} \xlongequal{\ } S^{\amalg {2}} \to S)^\op$ with target $(X\amalg Y, (x,y) \in\calF(X\amalg Y) \simeq \calF(X)\times \calF(Y)) \in t^\ast\calE$. We have seen in the proof of \Cref{Prop:motiviccolim} that an $s^\op$-Cartesian lift for such an edge was given by (the opposite of)
\[\begin{tikzcd}
	{X\amalg Y} & {(X\times Y)^{\amalg 2}} & {X\times Y=R_{S^{\amalg 2}/S}(X\amalg Y)} \\
	{S^{\amalg 2}} & {S^{\amalg 2}} & S
	\arrow[from=1-1, to=2-1]
	\arrow[from=1-2, to=1-1]
	\arrow["{\nabla'}", from=1-2, to=1-3]
	\arrow[from=1-2, to=2-2]
	\arrow[from=1-3, to=2-3]
	\arrow[equals, from=2-1, to=2-2]
	\arrow["\nabla", from=2-2, to=2-3]
\end{tikzcd}\]
and thus the source of the $(s^\op \circ t^\ast F)$-Cartesian lift we are looking for is 
$$\big(X\times Y, \calF(X\amalg Y \longleftarrow (X\times Y)^{\amalg 2} \longrightarrow X\times Y)(x,y)\big),$$
which is exactly $\calF(\pi_X^\op)(x) \otimes_{X\times Y} \calF(\pi_Y^\op)(y) \in \calF(X\times Y)$ by construction.
\end{proof}

\begin{Prop}\label{Prop:Masacolimit} Let $\calF$ be as in \Cref{Prop:motiviccolim}. Then the associated motivic colimit functor $M_\calF$ is a symmetric monoidal, colimit-preserving functor sending the arrow $u: y(X)\to \calF^\simeq$ (where $X\in \Sm_S$ has structure map $f:X\to S$) to $f_\sharp(E)$, where $E = u(X)(\id{X}) \in \calF^\simeq(X)$ is classified by $u$. More generally, the image of an arrow $\gamma :\calG\to \calF^\simeq$ in $\Pre(\Sm_S)_{/\calF^\simeq}$ is given by
$$M_\calF \simeq \colim_{(x,X) \in (\Sm_S)_{\sslash\mathcal{G}}} (p_X)_\sharp \gamma(x).$$     
\end{Prop}
\begin{proof}
    The last assertion is \cite[Rmk.\ 16.5]{BH-norms} or \cite[Rmk.\ 2.6]{BEH}, and is deduced from the expression of an element of the slice as a colimit of representable objects, and the fact that $M_\calF$ preserves colimits (or viewing $M_\calF$ as a left Kan extension).
\end{proof}

To obtain the motivic Thom spectrum functor, we now apply \Cref{Def:motiviccolim} to the functor $\SH^\omega$.
\begin{Prop}\label{Prop:motivicThom} The functor $\SH^\omega: \Sm_\R^\op \to \CAlg(\Cat^\times)$ satisfies all assumptions of \Cref{Prop:motiviccolim}; in the statement of the latter, we may choose $\calF= \SH^\omega$ and $\calG = \SH$. In particular, there is a motivic colimit functor $(\Sm_\R)_{/\SH^\wgpd} \to \SH^\omega(\R)$, extending to a symmetric monoidal functor $M := M_{\SH^\wgpd} : \Pre_\Sigma(\Sm_\R)_{/\SH^\wgpd} \longrightarrow \SHR$,
which we call the \emph{motivic Thom spectrum functor}.
\end{Prop}

\begin{proof}[Proof of \Cref{Prop:motivicThom}]
By \Cref{Prop:6FFforSH}, the functor $\SH$ defines a spherical presheaf of symmetric monoidal categories that satisfies assumptions $(i)$ to $(iii)$ in \Cref{Prop:motiviccolim}. Then, by \Cref{Prop:6FFforSHomega}, $\SH^\omega$ also satisfies these assumptions. Therefore, the statement follows as soon as we show that $\SH$, viewed as a functor with source category $\Span$ with values in categories, lifts to sifted-cocomplete categories. For all $X\in\Sm_\R$, $\SH(X)$ is cocomplete, in particular sifted-cocomplete. Moreover, if $(X\xleftarrow{f} Y \xrightarrow{\nabla}Z)$ is a morphism in $\Span$, then its image $\nabla_\otimes\circ f^\ast :\SH(X) \to \SH(Z)$ by $\SH$ preserves sifted colimits. Indeed, $f^\ast$ preserves all colimits as a left adjoint, and $\nabla$ is induced by the tensor product functor, which preserves sifted colimits since it preserves colimits in both variables separately. More precisely, if $\nabla: Z\amalg Z \to Z$ is the simplest fold map, then $\nabla_\otimes : \SH(Z) \times \SH(Z) \to \SH(Z)$ is the tensor product (\Cref{Prop:motiviccolim}$(iii)$). In this situation, for any sifted diagram $\calG: \mathcal{D} \to \SH(Z) \times \SH(Z)$, we have
\begin{align*}
\nabla_\otimes(\colim_{\mathcal{D}}\ \calG) &\simeq \nabla_\otimes(\colim_{\mathcal{D}}\ \pi_1\calG, \colim_{\mathcal{D}}\ \pi_2\calG)\\
&\simeq \left(\colim_{\mathcal{D}}\ \pi_1\calG\right) \otimes\left( \colim_{\mathcal{D}}\ \pi_2\calG\right)\\
&\simeq \colim_{\mathcal{D}\times \mathcal{D}}\ (\pi_1\calG\pi_1 \otimes \pi_2\calG\pi_2)\tag{\text{$\otimes$ preserves colimits in both variables separately}}\\
&\simeq \colim_{\mathcal{D}}\  (\pi_1\calG\otimes \pi_2\calG) \tag{by definition of a sifted category}\\
&\simeq \colim_{\mathcal{D}}\  \nabla_\otimes\calG
\end{align*}
where $\pi_1$ and $\pi_2$ are the canonical projections for the products $\SH(Z)\times \SH(Z)$ and $\mathcal{D}\times\mathcal{D}$. The same holds for more general fold maps. This finishes the proof.
\end{proof}

\begin{Rmk}\label{Rmk:Masacolimit}
    By \Cref{Prop:Masacolimit}, the motivic Thom spectrum functor is informally given by the formula
$$M(\calG\to \SH^\simeq) \simeq \colim_{(x,X) \in (\Sm_\R)_{\sslash\mathcal{G}}} (p_X)_\sharp \gamma(x),$$
where $p_X$ is the structure map $X\to \R$, and the colimit is taken in $\SH(\R)$.
\end{Rmk}

\begin{Ex}[\protect{\cite[\S 16.2 and Ex.\ 16.22]{BH-norms}}, \protect{\cite[below Lemma 4.6]{BH-eta}}]\label{Ex:MGLMSLMSp} Paralleling \Cref{Ex:MSO}, we now explain how to define $\Einfty$-structures on the motivic cobordism spectra using the motivic Thom spectrum functor.
    The motivic $\Einfty$-ring spectrum $\mathsf{MGL}$ is the image by $M$ of the motivic $j$-homomorphism $j: K^\circ \longrightarrow \SH^\wgpd$, where $K^\circ$ is the rank $0$ summand of algebraic K-theory. As we will now see, $j$ is constructed as a morphism of presheaves of symmetric monoidal categories, and thus acquires by \Cref{Prop:CAlginslice} below the structure of a commutative algebra object in the source category of the functor $M$. For any $X\in\Sm_\R$, there is a symmetric monoidal category $\mathsf{Vect}(X)$ of algebraic vector bundles on $X$, where the tensor product is the direct sum of vector bundles. Pullback of vector bundles makes the construction contravariantly functorial in $X$. There is a symmetric monoidal functor $\mathsf{Vect}(X) \to \mathsf{Pic}(\SH(X))$ sending a vector bundle $\xi$ to $\Sigma^\infty(\xi/(\xi\setminus\{0\}))$, where $\mathsf{Pic}(\SH(X))$ is the subgroupoid of $\SH(X)^\simeq$ spanned by the invertible objects. It extends to a morphism of presheaves of $\Einfty$-spaces $\mathsf{Vect}^\simeq\to \mathsf{Pic}(\SH(-))$, which factors through the group completions $\mathsf{Vect}(X)^{\simeq,\mathsf{gp}}$ because $\mathsf{Pic}(\SH(X))$ is already group-like by construction. Thus, we obtain a morphism $K \to \mathsf{Pic}(\SH(-)) \to \SH^\wgpd$ of spherical presheaves of group-like $\Einfty$-spaces. Furthermore, the rank map $\mathsf{Vect}(X)^\simeq \to \N$ is a morphism of monoids, and taking group completions yields a morphism $K \to \Z$ of spherical presheaves of group-like $\Einfty$-spaces, whose fiber is the rank zero part $K^\circ$. Restricting our morphism $K \to \SH^\wgpd$ to $K^\circ$, we obtain the motivic $j$-homomorphism $j: K^\circ \longrightarrow \SH^\wgpd$.

	This construction of $\MGL$ recovers the original definition of Voevodsky: in \cite[\S 3.5]{Voevodsky-Milnor} (also see \cite[\S 6.3]{Voevodsky-ICM}), $\MGL$ is defined using a more explicit colimit of suspensions of Thom spaces of the canonical bundles on Grassmannians; the formula of \cref{Rmk:Masacolimit} for the motivic $j$-homomorphism agrees with this more geometric definition by \cite[Theorem 6.13]{BH-norms}.

    Similarly, the motivic spectra $\MSL$ and $\mathsf{MSp}$ are respectively given by $M(KSL^\circ \to K^\circ \xrightarrow{j} \SH^\wgpd)$ and $M(KSp^\circ \to K^\circ \xrightarrow{j} \SH^\wgpd)$, where $KSL^\circ$ and $KSp^\circ$ are defined similarly as $K^\circ$, but with respect to even-dimensional oriented bundles and even-dimensional symplectic bundles, respectively. They can be described by $K^\circ \simeq \Lmot BGL \simeq \Lmot \colim_{n\in\N} BGL_n$, $KSL^\circ \simeq \Lmot BSL$, and $KSp^\circ \simeq \Lmot BSp$, respectively (\cite[after Lem.\ 4.6]{BH-eta} and \cite[before Thm.\ 16.13]{BH-norms}).
\end{Ex}

\begin{Prop}\label{Prop:CAlginslice} There is a functor
$$\Fun^\times(\Span(\Sm_S,\all,\fold), \spaces)_{/\SH^\omega} \longrightarrow \CAlg\left(\Pre_\Sigma(\Sm_S)^{\otimes '}_{/\SH^\wgpd}\right),$$
where $\Fun^\times$ denotes the category of functors preserving finite products (and compatible natural transformations). In particular, a morphism of spherical presheaves $\Einfty$-spaces $A \to \SH^\wgpd$ defines a commutative algebra object in the slice $\Pre_\Sigma(\Sm_S)^{\otimes'}_{/\SH^\wgpd}$.
\end{Prop}
\begin{proof}
    This follows from the last displayed composition of functors in \cite[proof of Prop.\ 16.17 and Rmk.\ 16.18]{BH-norms}. 
\end{proof}

\begin{Rmk}
    This is a partial analog to \Cref{Prop:algebrasintheslice}. Indeed, by \Cref{Prop:algebrasintheslice}, a morphism of commutative algebras in $\Pre_\Sigma(\Sm_S)$ with target $\SH^\wgpd$ yields a commutative algebra in the slice $\Pre_\Sigma(\Sm_S)^{\otimes}_{/\SH^\wgpd}$. We don't know if the $\otimes$ and $\otimes'$ structures agree, but \Cref{Prop:CAlginslice} tells us that, at least, we may construct some commutative algebras with respect to both structures in the same way.
\end{Rmk}

\vspace{0.4em}
\section{Real realizations of motivic Thom spectra}\label{Sect:comp}

As explained in the introduction to the previous section, we now prove that real realization interacts well with multiplicative Thom spectra. More precisely, we show that there is an equivalence of symmetric monoidal functors $\rR \circ M \simeq \Mtop \circ \widetilde{\rR}$, where $\widetilde{\rR} : \Pre_\Sigma(\Sm_k)_{/\SH^\wgpd} \to \spaces_{/\Sp^\kgpd}$ is the functor induced by real realization between the respective sources of the motivic and topological Thom spectrum functor. We will make the statement more precise and provide a proof in \cref{Subsect:compThom}. Before that, we will need to recall some facts about the real realization functor (\cref{Subsect:rR}).

As an application, we compute in \cref{Subsect:MSL} our main examples, showing that the real realizations of the motivic $\Einfty$-rings $\MGL$, $\MSL$, and $\MSp$, representing, respectively, general linear, special linear and symplectic algebraic cobordism, are equivalent to the topological $\Einfty$-rings $\MO$, $\MSO$, and $\MU$, representing, respectively, non-oriented, oriented, and complex cobordism.

The case of complex realization is deferred to \Cref{Subsect:rCThom}. The arguments are very similar to the real case; however, small modifications are required since our proof in the real case uses that the functor $\rR: \Spc(\R) \to \Spc$ is a localization (\Cref{Prop:ret2}), whereas this is not known in the complex case.

\vspace{0.2cm}
\subsection{Recollections on the real realization functor}\label{Subsect:rR}\hfill\vspace{0.2cm}

We first need to recall a few important properties of the real realization functor. See \cite{DI} for the original construction in the $C_2$-equivariant case; we follow the exposition of \cite[\S 10]{Bachmann-rho} or \cite[\S 7.2]{ret2} (in the $\infty$-categorical, non-equivariant setting). In this subsection, we assume that $k \hookrightarrow \R$ is a subfield. We could also assume $k= \R$ without loss of generality, as then the real realization functor over $k$ is simply obtained by precomposing the one over $\R$ with the base change along $k \hookrightarrow \R$.

\begin{Def}\label{Def:rR} Consider the functor of real points $\Sm_k \to \Spc$. It preserves finite products. Therefore, its left Kan extension $\Pre(\Sm_k) \to \Spc$ (respectively $\Pre(\Sm_k)_\ast \to \Spc_\ast$) is a colimit-preserving, symmetric monoidal functor for the Cartesian structures (respectively for the smash products). These functors invert motivic equivalences. Thus, passing to the localization, we obtain colimit-preserving, symmetric monoidal functors $\rR: \Spc(k) \to \Spc$, $\rR: \Spc(k)_\ast \to \Spc_\ast$, called \emph{real Betti realization}. The composition $\Spc(k)_\ast \to \Spc_\ast \to \Sp$ with the infinite suspension functor sends $\mathbb{P}^1$ to a tensor-invertible object, and thus induces a colimit-preserving, symmetric monoidal functor $\rR: \SH(k) \to \Sp$.
\end{Def}

An important feature of the real realization functor on motivic spaces is that it admits a description as an explicit localization. An analogous statement holds for motivic spectra \cite[Thm.\ 35]{Bachmann-rho}. To describe this localization, we first need a definition.

\begin{Def}[\protect{\cite[\S 1.2]{Scheiderer-ret}}] The real-étale topology on $\Sm_k$ is the topology generated by real-étale covers, that is, families $\{U_i \to X\}_{i\in I}$ of étale morphisms inducing a surjection $\coprod_{i\in I} R(U_i) \to R(X)$ (here $R(X)$ denotes the real spectrum of $X$ \cite[\S 0.4.2]{Scheiderer-ret}; it is a topological space on the set
of pairs $(x,\mathfrak{o})$ with $x\in X$ a point and $\mathfrak{o}$ an ordering on the residue field of $x$).
\end{Def}

\begin{thm}[\protect{\cite[Thm.\ 1.1]{ret2}}]\label{Prop:ret2}
    The (unstable) real realization functor $\rR : \Pre(\Sm_\R) \to \spaces$ factors through the localization $L_{\Affl, \mathsf{ret}}(\Pre(\Sm_\R))$ ($\Affl$-invariant real-étale sheaves), and the induced functor $L_{\Affl, \mathsf{ret}}(\Pre(\Sm_\R)) \to \spaces$, given by taking sections on $\Spec(\R)$, is an equivalence
\[\begin{tikzcd}
	{\Pre(\Sm_\R)} & \spaces \\
	{L_{\Affl, \mathsf{ret}}(\Pre(\Sm_\R))}
	\arrow[from=1-1, to=1-2]
	\arrow["{L_{\Affl, \mathsf{ret}}}"', from=1-1, to=2-1]
	\arrow["\simeq"', dashed, from=2-1, to=1-2]
\end{tikzcd}\]
whose inverse is given by left Kan extension.
\end{thm}

\vspace{0.2cm}
\subsection{Comparison of the topological and motivic Thom spectrum functors}\label{Subsect:compThom}\hfill\vspace{0.2cm}

In this subsection, we aim to show that the motivic Thom spectrum functor corresponds under real realization to the topological Thom spectrum functor, in the appropriate sense (see \Cref{Prop:Thomagreeassmfunctorsfinal}). To do so, we will consider another motivic colimit functor that already incorporates real realization in its definition (\Cref{Subsubsect:colimfunctforrealization}). In order to compare the latter with the motivic Thom spectrum functor (\Cref{Subsubsect:comp}), we first study in \Cref{Subsubsect:naturalitymotThom} naturality in $\calF$ of the construction from \Cref{Def:motiviccolim}.

\subsubsection{Naturality of the motivic Thom spectrum functor in the presheaf \texorpdfstring{$\calF$}{F}}\label{Subsubsect:naturalitymotThom}\hfill\vspace{0.2cm}

\begin{Prop}\label{Prop:naturalityofmotiviccolim} Let $\calF,\calG: \Span \to \Cat$ be functors satisfying the assumptions of \Cref{Prop:motiviccolim}. Assume that $\tau: \calF \to \calG$ is a natural transformation such that, in the notation of assumption $(i)$ in \Cref{Prop:motiviccolim}, for any smooth morphism $f: Y\to X$ in $\Sm_\R$, the canonical transformation $f_\sharp \tau_Y \Rightarrow \tau_X f_\sharp$ is an equivalence (see \Cref{Rmk:premotivicadjunction} below). Then, postcomposition with $\tau$ induces a natural transformation $\tau_\sharp: (\Sm_\bullet)_{/\calF^\simeq} \to (\Sm_\bullet)_{/\calG^\simeq}$, fitting into a commutative diagram of transformations of functors $\Span \to \Cat$ preserving finite products:
\begin{equation}\label{Diagram:naturality}
\begin{tikzcd}
	{(\Sm_\bullet)_{/\calF}^\simeq} & {(\Sm_\bullet)_{/\calG}^\simeq} \\
	\calF & \calG.
	\arrow["{\tau_\sharp}", from=1-1, to=1-2]
	\arrow["{M_\calF}"', from=1-1, to=2-1]
	\arrow["{M_\calG}", from=1-2, to=2-2]
	\arrow["\tau"', from=2-1, to=2-2]
\end{tikzcd}
\end{equation}

Moreover, if $\calF$ and $\calG$ are subpresheaves of some spherical presheaves of symmetric monoidal categories that lift to sifted-cocomplete categories, denoted by $\calF'$ and $\calG'$ respectively, and $\tau$ extends to a transformation $\calF' \Rightarrow \calG'$ of functors $\Span \to \widehat{\mathsf{Cat}}_\infty^\mathsf{sift}$, then the statement holds for $M_\calF$ and $M_\calG$ replaced with $\Pre_\Sigma(\Sm_\bullet)_{/\calF^\simeq} \to \calF'$ and $\Pre_\Sigma(\Sm_\bullet)_{/\calG^\simeq} \to \calG'$, respectively.
\end{Prop}

\begin{Rmk}\label{Rmk:premotivicadjunction} In the situation of \Cref{Prop:naturalityofmotiviccolim}, $\calF$ and $\calG$ are in particular \emph{$\Sm$-premotivic categories over $\Sm_\R$} \cite[\S 2.2]{EK}. That is, they are functors $\Sm_\R^\op \to \Cat$ such that for each $(f:Y\to X)\in \Sm$ a smooth morphism in $\Sm_\R$, $f^\ast: \calF(X) \to \calF(Y)$ has a left adjoint $f_\sharp$ (respectively, for $\calG$). If each component of $\tau$ admits a left adjoint, our assumption on $\tau$ gives exactly the notion of a \emph{premotivic adjunction}. Notice further that a transformation $\tau$ as in \Cref{Prop:naturalityofmotiviccolim} automatically also preserves the units and counits of the adjunctions $f_\sharp \dashv f^\ast$ for any smooth morphism $f$. Indeed, applying $\tau$ to the unit and counit for an adjunction $f_\sharp : \calF(Y) \leftrightarrows \calF(X) : f^\ast$ yields transformations which exhibit $f_\sharp : \calG(Y) \leftrightarrows \calG(X) : f^\ast$ as an adjunction; they must therefore be equivalent to the unit and counit we were originally given for this adjunction \cite[\href{https://kerodon.net/tag/02F4}{Tag 02F4}]{Kerodon}. This holds very generally and may also be seen by viewing $\tau_Y$ as a morphism between the biCartesian fibrations over $\Delta^1$ representing these adjunctions \cite[\S 5.2.2]{Lurie-HTT}.
\end{Rmk}

\begin{proof}[Proof of \Cref{Prop:naturalityofmotiviccolim}]
   Firstly, to construct $\tau_\sharp$, we have to produce a morphism between the Cartesian fibrations classified by $(\Sm_\bullet)_{/\calF^\simeq}$ and $(\Sm_\bullet)_{/\calG^\simeq}$. We use the notation of the proof of \Cref{Prop:motiviccolim}. The transformation $\tau$ induces a morphism between the Cartesian fibrations $F: \calE \to \Span^\op$ and $G: \mathcal{H} \to \Span^\op$ classified by $\calF$ and $\calG$, respectively. It pulls back to morphisms of Cartesian fibrations $s^\ast \tau : s^\ast \calE \to s^\ast\mathcal{H}$ and $t^\ast\tau: t^\ast\calE \to t^\ast\mathcal{H}$ over $\Fun_{\Sm}(\Delta^1,\Span)^\op$. The latter defines a morphism of Cartesian fibrations over $\Span^\op$ between $s^\op \circ t^\ast F$ and $s^\op \circ t^\ast G$. Indeed, assume $e$ is an $(s^\op \circ t^\ast F)$-Cartesian edge in $t^\ast \calE$. Then, by \cite[\href{https://kerodon.net/tag/01UL}{Tag 01UL}]{Kerodon}, $e$ is $t^\ast F$-Cartesian and therefore $(t^\ast\tau) (e)$ is $(t^\ast G)$-Cartesian. By \emph{loc.\ cit.}, $(t^\ast G)((t^\ast\tau) (e)) = (t^\ast F)(e)$ is $s^\op$-Cartesian, and so $t^\ast\tau(e)$ is $(s^\op \circ t^\ast G)$-Cartesian.\\

   Secondly, we show the commutativity of Diagram (\ref{Diagram:naturality}). We have a diagram
\[\begin{tikzcd}[column sep = 4em, row sep = 1em]
	{t^\ast\calE} && {s^\ast\calE} && \calE \\
	& {t^\ast\mathcal{H}} && {s^\ast \mathcal{H}} && {\mathcal{H}} \\
	\\
	\\
	&&&&& {\Span^\op}
	\arrow["{\psi_\calF}"{description}, from=1-1, to=1-3]
	\arrow["{t^\ast\tau}"{description}, from=1-1, to=2-2]
	\arrow["{t^\ast F}"{description},curve={height=5pt}, dashed, from=1-1, to=5-6]
	\arrow["{\chi_\calF}"{description}, from=1-3, to=1-5]
	\arrow["{s^\ast\tau}"{description}, from=1-3, to=2-4]
	\arrow["\tau"{description}, from=1-5, to=2-6]
	\arrow["F"{description}, dashed, from=1-5, to=5-6]
	\arrow["{\psi_\calG}"{description}, from=2-2, to=2-4]
	\arrow["{t^\ast G}"{description}, curve={height=10pt}, from=2-2, to=5-6]
	\arrow["{\chi_\calG}"{description}, from=2-4, to=2-6]
	\arrow["G"{description}, from=2-6, to=5-6]
\end{tikzcd}\]
of morphisms of Cartesian fibrations over $\Span^\op$. We have to show that the top parallelogram commutes. Since the square on its right-hand side commutes by construction of $s^\ast\tau$, it suffices to show the commutativity of the square on its left-hand side. Recall that $\psi_\calF$ was obtained as a relative left adjoint to the morphism of Cartesian fibrations $\phi_\calF : s^\ast\calE \to t^\ast\calE$ corresponding to the natural transformation $\Fun_\Sm(\Delta^1,\Span) \times \Delta^1 \to \Span \to \Cat$ (postcomposition of the evaluation map with $F$) and similarly for $\calG$. The diagram 
\[\begin{tikzcd}[row sep = 3em, column sep = 3em]
	{s^\ast\calE} & {t^\ast\calE} \\
	{s^\ast \mathcal{H}} & {t^\ast\mathcal{H}}
	\arrow["{\phi_\calF}"{description}, from=1-1, to=1-2]
	\arrow["{s^\ast\tau}"{description}, from=1-1, to=2-1]
	\arrow["{t^\ast\tau}"{description}, from=1-2, to=2-2]
	\arrow["{\phi_\calG}"{description}, from=2-1, to=2-2]
\end{tikzcd}\]
involving the left adjoints $\phi_\calF \dashv \psi_\calF$ and $\phi_\calG \dashv \psi_\calG$ is commutative by construction, since both composites correspond to the natural transformation
\[\begin{tikzcd}[column sep = 7em]
	{\Fun_\Sm(\Delta^1,\Span) } & \Span & \Cat.
	\arrow[""{name=0, anchor=center, inner sep=0}, "s", curve={height=-12pt}, from=1-1, to=1-2]
	\arrow[""{name=1, anchor=center, inner sep=0}, "t"', curve={height=12pt}, from=1-1, to=1-2]
	\arrow[""{name=2, anchor=center, inner sep=0}, "F", curve={height=-12pt}, from=1-2, to=1-3]
	\arrow[""{name=3, anchor=center, inner sep=0}, "G"', curve={height=12pt}, from=1-2, to=1-3]
	\arrow["{\mathsf{ev}}", shorten <=3pt, shorten >=3pt, Rightarrow, from=0, to=1]
	\arrow["\tau", shorten <=3pt, shorten >=3pt, Rightarrow, from=2, to=3]
\end{tikzcd}\]
In particular, there is a canonical exchange transformation $\psi_\calG \circ t^\ast\tau \Rightarrow s^\ast\tau\circ\psi_\calF$ given by the composition 
$$\psi_\calG \circ t^\ast\tau \Rightarrow \psi_\calG \circ t^\ast\tau \circ \phi_\calF \circ \psi_\calF \simeq \psi_\calG \circ \phi_\calG \circ s^\ast\tau \circ \psi_\calF \Rightarrow s^\ast\tau\circ \psi_\calF.$$
Consider an edge in $t^\ast \calE$, it takes the form of a diagram similar to (\ref{Diagram:edgeinFunsm}) in Step 3 of the proof of \Cref{Prop:motiviccolim}, together with objects $H\in\calF(Z')$, $E\in\calF(W)$, and a morphism $\kappa: H \to \nabla'_\otimes f'^\ast E$ in $\calF(Z')$. The difference compared to the aforementioned Step 3 is that $Y'$ and $Z'$ are not required to be Weil restrictions, and $\kappa$ need not be an equivalence. On the one hand, the composite $s^\ast\tau \circ \psi_\calF$ maps this edge to the edge consisting in the data of the same diagram and the morphism
$$\tau_Z\bigg(u_\sharp H \xrightarrow{u_\sharp\kappa} u_\sharp \nabla'_\otimes f'^\ast E \xrightarrow{\mathsf{Dis}_{\sharp,\otimes}} \nabla_\otimes (\pi_Y)_\sharp\pi_W^\ast E \xrightarrow{\mathsf{Ex}_\sharp^\ast} \nabla_\otimes f^\ast g_\sharp E\bigg)$$
in $\calG(Z)$, where the $(-)_\sharp$, $(-)_\otimes$, and $(-)^\ast$ refer to the functoriality of $\calF$. On the other hand, $\psi_\calG \circ t^\ast \tau$ maps it to the edge consisting in the data of the same diagram and the morphism
$$u_\sharp \tau_{Z'}(H) \xrightarrow{\tau_{Z'}(\kappa)} u_\sharp \tau_W(\nabla'_\otimes f'^\ast (E)) \simeq u_\sharp \nabla'_\otimes f'^\ast \tau_W(E) \xrightarrow{\mathsf{Dis}_{\sharp,\otimes}} \nabla_\otimes (\pi_Y)_\sharp\pi_W^\ast \tau_W(E) \xrightarrow{\mathsf{Ex}_\sharp^\ast} \nabla_\otimes f^\ast g_\sharp \tau_W(E),$$
where the $(-)_\sharp$, $(-)_\otimes$, and $(-)^\ast$ now refer to the functoriality of $\calG$. By assumption and by \Cref{Rmk:premotivicadjunction}, $\tau$ is compatible with the exchange transformation $\mathsf{Ex}_\sharp^\ast$ and with $u_\sharp$. For compatibility with the distributivity transformation, note that it is by definition the composition of the exchange transformation $\mathsf{Ex}_{\sharp,\otimes}$ with the counit of the adjunction $(-)_\sharp \dashv (-)^\ast$. The transformation $\tau$ is compatible with them by assumption and \Cref{Rmk:premotivicadjunction}, since compatibility of $\tau$ with $(-)_\otimes$ is given by it being a morphism of presheaves of symmetric monoidal categories (naturality for forward morphisms in $\Span$, i.e., spans consisting of one identity and one fold map). This shows that our exchange transformation $\psi_\calG \circ t^\ast\tau \Rightarrow s^\ast\tau\circ\psi_\calF$ is an equivalence, and, as explained above, this allows us to conclude that Diagram (\ref{Diagram:naturality}) commutes.\\
  
Finally, since $t^\ast\tau$ is a pullback, the induced transformation $(\Sm_\bullet)_{/\calF^\simeq} \to (\Sm_\bullet)_{/\calG^\simeq}$ is indeed given over $X\in\Sm_\R$ by the functor $(\Sm_X)_{/\calF^\simeq} \to (\Sm_X)_{/\calG^\simeq}$ induced by post-composition by $\tau$ (in the sense that a pair $(Y\in\Sm_X, E\in\calF(Y))$ is sent to $(Y\in\Sm_X, \tau_Y(E)\in\calG(Y))$). This justifies the notation $\tau_\sharp$ in Diagram $(\ref{Diagram:naturality})$. When $\calF$ and $\calG$ are subpresheaves of $\calF'$ and $\calG'$ lifting to sifted-cocomplete categories, the argument is the same as in Step 5 of the proof of \Cref{Prop:motiviccolim}.
\end{proof}

\subsubsection{Motivic colimit functor associated with the real realization functor}\label{Subsubsect:colimfunctforrealization}\hfill\vspace{0.2cm}

We now construct a functor $\calR^\kappa: \Span \to \Cat$ satisfying the assumptions of \Cref{Prop:motiviccolim}, which provides us with a motivic colimit functor related to real realization. Using the naturality result of the previous subsection, we compare it to the motivic Thom spectrum functor. Then, in \Cref{Subsubsect:comp}, we will compare our new motivic colimit functor to the topological Thom spectrum functor.

\begin{Def}\label{Def:calR} Let $\calR$ be the presheaf of (not necessarily small) symmetric monoidal categories on $\Sm_\R$, defined by $X \in \Sm_\R \mapsto \Sp(\spaces_{/\rR(X)})$, where functoriality is given by pullback (see \Cref{Prop:6FFforR} below). Let $\calR^\kappa: \Span = \Span(\Sm_\R, \all, \fold) \to \Cat$ be the spherical presheaf of essentially small symmetric monoidal categories defined by $X \in \Sm_\R \mapsto \Sp(\spaces_{/\rR(X)})^\kappa$ (recall from \Cref{Notation:Mtopkappa} that $\kappa$ is an uncountable regular cardinal such that $\Spc^\kappa$ is closed under finite limits in $\Spc$). Here $\Sp(\spaces_{/\rR(X)})^\kappa$ is viewed as a (small) symmetric monoidal subcategory of $\Sp(\spaces_{/\rR(X)})$ (by \Cref{Prop:Spkappasm} for $B=\ast$).
\end{Def}

\begin{Rmk}\label{Rmk:whykappa} Unlike the case of $\SH$ in \Cref{Subsect:motThom}, $\calR^\omega$ is a priori not a subpresheaf of $\calR$. Indeed, its functoriality is given by the pullback functors, which a priori do not preserve compact objects in the case of $\calR$. However, as we will see in \Cref{Prop:6FFforR}, they do preserve $\kappa$-compact objects. This is the reason for our choice of the cardinal $\kappa$. One could also want to use the subcategories of invertible objects $\mathsf{Pic}(\SH(X))$ and $\mathsf{Pic}(\calR(X))$ instead, in which case we do obtain subpresheaves, but then it is not clear that the functors $f_\sharp$ (for $f$ a smooth morphism in $\Sm_\R$) restrict to these subcategories, as $f_\sharp$ is not symmetric monoidal. 
\end{Rmk}

To apply \Cref{Def:motiviccolim} to $\calF = \calR^\kappa$, we show the following result.
\begin{Prop}\label{Prop:6FFforR} The functor $\calR$ from \Cref{Def:calR} satisfies all the assumptions of \Cref{Prop:motiviccolim}, except possibly for essential smallness of the categories in its image, and lifts to sifted-cocomplete categories. It also satisfies a projection formula as in \Cref{Prop:6FFforSH}$(v)$. The functor $\calR^\kappa$ satisfies all the assumptions of \Cref{Prop:motiviccolim}; in the statement of the latter, we may choose $\calF= \calR^\kappa$ and $\calG = \calR$. In particular, we obtain a symmetric monoidal motivic colimit functor
 $$M_\calR : \Pre_\Sigma(\Sm_\R)_{/\calR^\kgpd} \longrightarrow \calR(\R) \simeq \Sp.$$  
\end{Prop}
\begin{proof}
\textbf{Step 1:} \emph{We prove that $\calR$ is a spherical presheaf of symmetric monoidal categories and satisfies assumption $(i)$.} For any morphism $f:X \to Y$ in $\Sm_\R$, we have a diagram
\[\begin{tikzcd}[column sep = 7em, row sep = 2.5 em]
	{\spaces_{/\rR(X)}} & {\Sp\left(\spaces_{/\rR(X)}\right)} \\
	{\spaces_{/\rR(Y)}} & {\Sp\left(\spaces_{/\rR(Y)}\right)}
	\arrow["{\Sigma^\infty_+}"', from=1-1, to=1-2]
	\arrow[""{name=0, anchor=center, inner sep=0}, "{f_\sharp}"', shift right=7, from=1-1, to=2-1]
	\arrow[""{name=1, anchor=center, inner sep=0}, "{f_\ast}", shift left=7, from=1-1, to=2-1]
	\arrow[""{name=2, anchor=center, inner sep=0}, "{f_\sharp}"', shift right=4, from=1-2, to=2-2]
	\arrow[""{name=3, anchor=center, inner sep=0}, "{f^\ast}"{description}, from=2-1, to=1-1]
	\arrow["{\Sigma^\infty_+}", from=2-1, to=2-2]
	\arrow[""{name=4, anchor=center, inner sep=0}, "{f^\ast}"', shift right=4, from=2-2, to=1-2]
	\arrow["\dashv"{anchor=center}, draw=none, from=0, to=3]
	\arrow["\dashv"{anchor=center}, draw=none, from=2, to=4]
	\arrow["\dashv"{anchor=center}, draw=none, from=3, to=1]
\end{tikzcd}\]
The adjunction $f_\sharp \dashv f^\ast$ is built as in the case of schemes, where $f^\ast$ is the pullback functor and $f_\sharp$ is induced by post-composition of structure maps. In particular, $f^\ast$ preserves products because it is a right adjoint, so it is symmetric monoidal with respect to the Cartesian symmetric monoidal structures on both categories. The right adjoint $f_\ast$ to $f^\ast$ exists because pullbacks commute with colimits in spaces. Since both $f_\sharp$ and $f^\ast$ are left adjoints, they preserve colimits and thus descend to an adjunction on the stabilization of both categories (see \cite[Rmk.\ 2.2]{Volpe}, using that stabilization is tensoring with $\Sp$ in $\PrL$ by \cite[Ex.\ 4.8.1.23]{Lurie-HA}). The induced functor $f^\ast$ becomes symmetric monoidal for the smash product. Moreover, we have, for all $X,Y\in\Sm_\R$
\begin{align*}
	\calR(X\amalg Y) &\simeq \Sp(\spaces_{/\rR(X \amalg Y)}) \simeq \Sp(\spaces_{/\rR(X)} \times \spaces_{/\rR(Y)}) \\
	&\simeq \Sp(\spaces_{/\rR(X)}) \times \Sp(\spaces_{/\rR(Y)}) \simeq \calR(X) \times \calR(Y).
\end{align*}
This proves assumption $(i)$ and establishes $\calR$ as a spherical presheaf of symmetric monoidal categories. \\

\textbf{Step 2:} \emph{We prove that $\calR$ satisfies assumption $(ii)$.} Consider a square as in the statement of assumption $(ii)$. We first prove a projection formula, i.e., that for all $S\in\calR(X)$ and $T\in\calR(Y)$, we have $f_\sharp(S \Smash f^\ast T) \simeq f_\sharp(S) \Smash T$ (via the canonical map). Since both sides preserve colimits in $S$ and $T$ separately, it suffices to show that the formula holds for the infinite suspension spectra of $S\in\spaces_{/\rR(X)}$ and $T\in\spaces_{/\rR(Y)}$. We have equivalences
\begin{align*}
f_\sharp(\Sigma^\infty_+ S \Smash f^\ast \Sigma^\infty_+T) &\simeq f_\sharp(\Sigma^\infty_+ S \Smash \Sigma^\infty_+ f^\ast T) \simeq f_\sharp(\Sigma^\infty_+ (S \times_{\rR(X)} f^\ast T))\simeq \Sigma^\infty_+ f_\sharp(S \times_{\rR(X)} f^\ast T)\\
    &\simeq \Sigma^\infty_+ (f_\sharp(S) \times_{\rR(Y)} T) \tag{pasting law for pullbacks, \cite[Lem.\ 4.4.2.1]{Lurie-HTT}}\\
    &\simeq f_\sharp(\Sigma^\infty_+ S) \Smash \Sigma^\infty_+ T
\end{align*}
induced by the canonical map (as can be checked at the level of spaces, because it is already the case for the equivalence $f_\sharp(S \times_{\rR(X)} f^\ast T) \simeq f_\sharp(S) \times_{\rR(Y)} T$).

To prove that the exchange transformation is an equivalence, it again suffices to show it before stabilization. For any arrow $(S\to \rR(X)) \spaces_{/\rR(X)}$ with structure map $x: S\to \rR(X)$, we have
\begin{align*}
    g^\ast f_\sharp(S\to \rR(X)) &\simeq g^\ast( S \to \rR(X) \to \rR(Y)) \simeq (S \times_{\rR(Y)} \rR(Y') \to \rR(Y'))\\
    &\simeq f'_\sharp(S \times_{\rR(X)} \rR(X') \to \rR(X')) \tag{$\star$}\\
    &\simeq f'_\sharp g'^\ast (S\to \rR(X))
\end{align*}
where $(\star)$ follows from the pasting law for pullbacks in the diagram
\[\begin{tikzcd}
	{S\times_{\rR(Y)} \rR(Y')} & {\rR(X')} && {\rR(Y')} \\
	S & {\rR(X)} && {\rR(Y)}
	\arrow[from=1-1, to=1-2]
	\arrow[from=1-1, to=2-1]
	\arrow["{f'}"', from=1-2, to=1-4]
	\arrow["{g'}"', from=1-2, to=2-2]
	\arrow["\lrcorner"{anchor=center, pos=0.125}, draw=none, from=1-2, to=2-4]
	\arrow["g"{description}, from=1-4, to=2-4]
	\arrow["x", from=2-1, to=2-2]
	\arrow["f", from=2-2, to=2-4]
\end{tikzcd}\]
Thus the exchange transformation is an equivalence.\\

\textbf{Step 3:} \emph{We prove that $\calR$ satisfies assumption $(iii)$.} Consider a diagram as in the statement of assumption $(iii)$. To lighten notation we deal only with the case $Y = Z^{\amalg n}$; the case of general maps in $\fold$ is similar. Write $U := W\times_X Y = \coprod_{i\leq n} U_i$, where $U_i$ lives over the $i$-th copy $Z_i$ of $Z$ in $Y$. Let $e: Y':= R_{Y/X}(U)\times_Z Y \to U$ be the natural map; it is given over $Z_i$ by the projection on the $i$-th component $\prod_{j\leq n} U_j \to U_i$. Then, for any $(A \to \rR(W)) \in \Spc_{/\rR(W)}$, let $B = (\pi_W)^\ast(A)$, which lives over $\rR(U)$. Write $B_i$ for the component of $B$ over $\rR(U_i)$. Now, $e^\ast(B)$, as a space over $\rR(Y') = \coprod_{i\leq n} (\rR(U_1) \times_{\rR(Z)} \dots \times_{\rR(Z)} \rR(U_n))$, is given by $$\coprod_{i\leq n} (\rR(U_1) \times_{\rR(Z)} \dots \times_{\rR(Z)} B_i \times_{\rR(Z)} \dots \times_{\rR(Z)} \rR(U_n)).$$
In the remainder of Step 3, unless specified otherwise, all products are fibered over $\rR(Z)$. Then $u_\sharp\nabla'_\otimes f'^\ast (A) = u_\sharp\nabla'_\otimes e^\ast (B)$ is given by the fiber product over $\rR(U_1) \times \dots \times \rR(U_n)$ of the factors $\rR(U_1) \times \dots \times B_i \times \dots \times \rR(U_n)$ for $1\leq i\leq n$. That is, $u_\sharp\nabla'_\otimes e^\ast (B) = B_1 \times \dots \times B_n \longrightarrow \rR(Z)$, which coincides with $\nabla_\otimes(\pi_Y)_\sharp(\pi_W)^\ast(A) = \nabla_\otimes(\pi_Y)_\sharp(B)$.
Now $\nabla_\otimes u'_\sharp f'^\ast (A) = \nabla_\otimes u'_\sharp e^\ast (B)$ is given by
$$\prod_{i\leq n} (\rR(U_1) \times \dots \times B_i \times \dots \times \rR(U_n)) \longrightarrow \rR(Z).$$
The distributivity transformation is now the map
$$B_1 \times \dots \times B_n  \longrightarrow \prod_{i\leq n} (\rR(U_1) \times \dots \times B_i \times \dots \times \rR(U_n)) \longrightarrow B_1 \times \dots \times B_n,$$
which is seen to be the identity by chasing through the construction. Therefore, in our situation, the distributivity transformation is an equivalence.\\

\textbf{Step 4:} \emph{We show that $\calR$ lifts to sifted-cocomplete categories.} For all $X\in\Sm_\R$, the slice $\spaces_{/\rR(X)}$ is presentable (as a particular case of \Cref{Prop:sliceofpshsm}, it is equivalent to $\Pre(\ast_{/\rR(X)})$). Then, using \cite[Prop.\ 1.4.4.4]{Lurie-HA}, the category $\Sp(\spaces_{/\rR(X)})$ is cocomplete. Moreover, for any map $f$ in $\Sm_\R$, the functor $f^\ast$ preserves colimits; and for any fold map $\nabla$, the functor $\nabla_\otimes$ preserves sifted colimits, by the same argument as in the proof of \Cref{Prop:motivicThom}. Since the operations $f^\ast$ and $\nabla_\otimes$ describe the functoriality of $\calR$ (viewed as a functor on the category $\Span$), the latter lifts to sifted-cocomplete categories.\\

\textbf{Step 5:} \emph{We prove that $\calR^\kappa$ satisfies the assumptions of \Cref{Prop:motiviccolim}.} By \Cref{Prop:kappacompinslice}, for any $X\in\Sm_\R$, $\rR X \in \Spc^\kappa$. Then, by \Cref{Prop:Spkappasm}, $\calR(X)^\kappa \subseteq \calR(X)$ is indeed a small symmetric monoidal subcategory. We are left to show that for any $f:X\to Y$ in $\Sm_\R$, the functors $f^\ast: \calR(Y) \to \calR(X)$ and $f_\sharp: \calR(X) \to \calR(Y)$ preserve $\kappa$-compact objects. For the latter, it simply follows from the fact that its right adjoint $f^\ast$ preserves all colimits, in particular $\kappa$-filtered ones. To show that $f^\ast$ preserves $\kappa$-compact objects, we use \Cref{Prop:toolscompact}. Recall from the proof of \Cref{Prop:Spkappasm} that $\calR(Y)$ is generated under filtered colimits by the (small) set of compact objects $\{ \Sigma^n \Sigma^\infty_+ V \mid n\in\Z, V\in (\Spc_{/\rR Y})^\omega\}$. Thus, we have to show that $f^\ast$ maps any object of the form $\Sigma^n \Sigma^\infty_+ V \in \calR(Y)$ with $n\in\Z$ and $V\in (\Spc_{/\rR Y})^\omega$ to a $\kappa$-compact object. Since $f^\ast$ commutes with $\Sigma^n \Sigma^\infty_+$ by construction, and this latter functor preserves $\kappa$-compactness in $\calR(Y)$ (as in the proof of \Cref{Prop:Spkappasm}), we only have to show that $f^\ast V \in (\Spc_{/\rR X})^\kappa$. Since $f^\ast$ is defined by pullback, by \Cref{Prop:kappacompinslice} it suffices to show that $V' \times_{\rR X} \rR Y \in \Spc^\kappa$, where $V = (V'\to \rR Y)$. This holds by choice of $\kappa$, since this is a finite limit of $\kappa$-compact spaces by \Cref{Prop:kappacompinslice} again (note that $V$ is in particular $\kappa$-compact in $\Spc_{/\rR Y}$). 
\end{proof}

We now want to apply \Cref{Prop:naturalityofmotiviccolim} to a transformation $\alpha: \SH^\omega \to \calR^\kappa$ of functors $\Span \to \Cat$. We first define it at the level of $\SH$ and $\calR$. To do so, we use the universal property of $\SH$ as a presheaf of symmetric monoidal categories. Indeed, as we saw in the construction of the real realization functor, to specify a symmetric monoidal colimit preserving functor $\SH(S) \to \calC$ where $\calC$ is a symmetric monoidal stable category, it suffices to define a symmetric monoidal functor $\Sm_S^\times \to \calC$ which is suitable $\Affl$-invariant and well-behaved with respect to the Nisnevich topology, and inverts $\mathbb{P}^1$. A more precise statement is proven in \cite{Robalo}. This universal property was made into one of the presheaf $\SH$ in \cite{DG}. Before stating it, we need a definition that axiomatizes some functoriality properties of $\SH$ we saw in \Cref{Prop:6FFforSH}.

\begin{Def}[\cite{DG}] Let $\calC$ be an essentially small 1-category with finite limits and a terminal object, and $P$ be a collection of maps in $\calC$ stable under pullback and equivalences. A \emph{$P$-pullback formalism over $\calC$} is a functor $C: \calC^\op \to \CAlg(\widehat{\mathsf{Cat}}_\infty)$ such that:
\begin{itemize}
\item for any $f\in P$, $f^\ast:=C(f)$ has a left adjoint $f_\sharp$,
\item the exchange transformation is an equivalence for pullback squares with two parallel sides in $P$ (see \Cref{Prop:6FFforSH}$(vii)$),
\item $C$ satisfies the projection formula (see \Cref{Prop:6FFforSH}$(v)$).
\end{itemize}
A \emph{morphism of $P$-pullback formalisms over $\calC$} is a natural transformation $\tau:C\to C'$ such that the canonical transformation $f^{C'}_\sharp \tau \Rightarrow \tau f^C_\sharp$ is an equivalence for any $f\in P$. This can be organized into a category $\mathsf{PB}(\calC, P)$.
\end{Def}

\begin{Def}[\cite{DG}]\label{Def:pullbackformalism} Let $\mathsf{PB}(\Sm_\R, \mathsf{Sm})^{c,pt}_{\Lmot, (\mathbb{P}^1,\infty)}$ be the subcategory of $\mathsf{PB}(\Sm_\R, \mathsf{Sm})$ consisting of \emph{cocomplete pointed $\Affl$-local Nisnevich-local $\Sm$-pullback formalisms with $\mathbb{P}^1$ invertible}, that is, of objects $C$ such that:
\begin{itemize}
\item for any $X\in\Sm_\R$, $C(X)$ is a cocomplete and pointed category,
\item $C$ satisfies Nisnevich excision,
\item for any $X\in \Sm_\R$, the projection induces a fully faithful functor $C(X) \to C(X\times \Affl)$,
\item the essentially unique morphism of pointed cocomplete pullback formalisms $\Spc(-)_\ast \to C$ sends $(\mathbb{P}^1,\infty) \in \Spc(\R)_\ast$ to an invertible object in $C(\R)$. 
\end{itemize}
The morphisms are those maps of $P$-pullback formalisms that are componentwise colimit-preserving.
\end{Def}

\begin{thm}[\protect{\cite[Cor.\ 6.35 and Rmk.\ 7.11]{DG}}]\label{Prop:universalpropSH} The functor $\SH$ is initial in the category $\mathsf{PB}(\mathsf{Sm}_\R, \mathsf{Sm})^{c,pt}_{\Lmot, (\mathbb{P}^1,\infty)}$ (see \Cref{Def:pullbackformalism}).
\end{thm}

\begin{Lemma}\label{Prop:calRPBformalism} The presheaf of symmetric monoidal categories $\calR$ is a cocomplete pointed $\Affl$-local Nisnevich-local $\Sm$-pullback formalism over $\Sm_\R$ with $\mathbb{P}^1$ invertible. 
\end{Lemma}
\begin{proof}
It follows from \Cref{Prop:6FFforR} that $\calR$ is a pullback formalism. Moreover, for any $X\in\Sm_\R$, the category $\Sp(\Spc_{/\rR(X)}) = \calR(X)$ is pointed and cocomplete. The pullback formalism $\calR$ is furthermore $\Affl$-local, because $\rR(X) \simeq \rR(X\times \Affl)$ for any $X\in\Sm_R$ via the projection, so that $\calR(X) \to \calR(X\times \Affl)$ is in particular fully faithful. To check that $\calR$ has Nisnevich excision, let $\iota: U\xhookrightarrow{} X$ and $p:Y\to X$ form an elementary distinguished Nisnevich square (i.e. $\iota$ is an open immersion, and $p$ is an étale map inducing an isomorphism between the complement of $U$ and its preimage by $p$). Then the corresponding pullback square in $\Sm_\R$ becomes a pushout square in $\Spc(\R)$; since $\rR$ preserves colimits we have 
$$\Spc_{/\rR(X)} \simeq \Spc_{/\rR(U) \amalg_{\rR(U\times_X Y)} \rR(Y)} \simeq \Spc_{/\rR(U)} \times_{\Spc_{/\rR(U\times_X Y)}} \Spc_{/\rR(Y)}$$
by the descent property in $\Spc$. Finally, we have 
$$\calR(X) \simeq \Sp(\Spc_{/\rR(X)}) \simeq \Sp(\Spc_{/\rR(U)}) \times_{\Sp(\Spc_{/\rR(U\times_X Y)})} \Sp(\Spc_{/\rR(Y)}) \simeq \calR(U) \times_{\calR(U\times_X Y)} \calR(Y)$$
because stabilization preserves limits. Indeed, it is described by tensoring with $\Sp$ in $\PrL$ \cite[Ex.\ 4.8.1.23]{Lurie-HA}, but $\Sp$ is dualizable \cite[Prop.\ D.7.2.3]{Lurie-SAG}, thus $-\otimes \Sp$ preserves all limits and colimits; moreover limits in $\PrL$ are computed on the underlying categories \cite[Prop.\ 5.5.3.13]{Lurie-HTT}. Finally, the essentially unique morphism of pointed cocomplete pullback formalisms $\chi: \Spc(-)_\ast \to \calR$ sends $(\mathbb{P}^1,\infty) \in \Spc(\R)_\ast$ to $\Sigma^\infty ( (\rR(\mathbb{P}^1) \to \rR(\R),\infty)) \simeq \Sigma^\infty(\Sph^1,\ast) \in \Sp$, which is invertible (see \Cref{Constr:alpha} below for the description of $\chi$). This verifies all the axioms. 
\end{proof}

\begin{Construction}\label{Constr:alpha} Let $\alpha: \SH \to \calR$ be the unique morphism in $\mathsf{PB}(\Sm_\R, \mathsf{Sm})^{c,pt}_{\Lmot, (\mathbb{P}^1,\infty)}$ from $\SH$ to $\calR$ (using \Cref{Prop:universalpropSH} and \Cref{Prop:calRPBformalism}). Unraveling the constructions in \cite[Thm.\ 6.3 and Rmk.\ 7.11]{DG}, we can give an explicit description of the components of $\alpha$. Given $S\in\Sm_\R$, real realization induces a finite-product-preserving functor $\Sm_S \to \Spc_{/\rR(S)}$ sending $X\in\Sm_S$ to the real realization of its structure map $\rR(X) \to \rR(S)$ (in the terminology of \cite{DG}, this is a morphism of pre-pullback formalisms, for which $\Sm_\bullet$ is initial). The left Kan extension of this functor inverts motivic equivalences, and factors through $\Spc(S)$. The corresponding functor on pointed objects can be post-composed with the infinite suspension functor $(\Spc_{/\rR(S)})_\ast \to \calR(S)$. There is therefore an induced symmetric monoidal functor $\Spc(S)_\ast \to \calR(S)$ that preserves colimits, and inverts $\mathbb{P}^1$. Finally, this functor factors through $\SH(S)$, and we obtain $\alpha_S: \SH(S) \to \calR(S)$. When $S = \R$, this recovers by construction the symmetric monoidal functor $\rR$ from \Cref{Def:rR}.
\end{Construction}

As promised, we now apply \Cref{Prop:naturalityofmotiviccolim} to the transformation $\alpha$ from \Cref{Constr:alpha}.
\begin{Prop}\label{Prop:compMSHMR} The transformation $\alpha$ from \Cref{Constr:alpha} restricts to a natural transformation $\alpha : \SH^\omega\to \calR^\kappa$, inducing a commutative diagram of symmetric monoidal categories and symmetric monoidal functors
\[\begin{tikzcd}
	{\Pre_\Sigma(\Sm_\R)_{/\SH^\wgpd}} & {\Pre_\Sigma(\Sm_\R)_{/\calR^\kgpd}} \\
	{\SH(\R)} & {\Sp\left(\spaces_{/\rR(\R)}\right) \simeq \Sp.}
	\arrow["{\alpha_\sharp}", from=1-1, to=1-2]
	\arrow["M"', from=1-1, to=2-1]
	\arrow["{M_{\calR^\kappa}}", from=1-2, to=2-2]
	\arrow["{\alpha_\R = \rR}", from=2-1, to=2-2]
\end{tikzcd}\]
\end{Prop}
\begin{proof}
 We begin by proving the first claim. By \Cref{Prop:toolscompact}, for all $X\in\Sm_R$, $\SH(X)^\omega$ is the thick subcategory generated by the family $E_X$ of \Cref{Prop:compactgenerators}. By stability of $\calR(X)$, and the fact that $\kappa$-compact objects are stable under finite colimits \cite[Tag \href{https://kerodon.net/tag/064W}{064W}]{Kerodon}, $\calR(X)^\kappa \subseteq \calR(X)$ is itself a thick subcategory. Since $\alpha_X : \SH(X) \to \calR(X)$ preserves colimits, it suffices to show that for any $n\in \Z$ and $Y\in\Sm_X$, we have $\alpha_X(\Sigma^{-2n,-n}\Sigma^\infty_+ Y) \in \calR(X)^\kappa$. This object is equivalent to $\Sigma^\infty_+(\rR Y \to \rR X)$, which is $\kappa$-compact because $\Sigma^\infty_+$ preserves $\kappa$-compact objects as we saw before, and $\rR Y \to \rR X$ is $\kappa$-compact in $\calR(X)$ by \Cref{Prop:kappacompinslice}. 

To prove the remainder of the statement, it suffices to check that $\alpha : \SH \to \calR$ satisfies the assumption of \Cref{Prop:naturalityofmotiviccolim} (compatibility with the left adjoints $(-)_\sharp$), because then so does its restriction to the subcategories of compact and $\kappa$-compact objects respectively. But $\alpha$ is by construction a morphism of pullback formalisms, and so the assumption holds by definition. We have $\alpha_\R = \rR$ by \Cref{Constr:alpha}.
\end{proof}

\subsubsection{Comparing the motivic colimit functor $M_{\calR^\kappa}$ and the topological Thom spectrum functor}\label{Subsubsect:comp}\hfill\vspace{0.2cm}

As mentioned at the beginning of the previous subsection, we now compare $M_{\calR^\kappa}$ with the topological Thom spectrum functor. We saw in \Cref{Prop:ret2} that the unstable real realization functor $\rR: \Pre(\Sm_\R) \to \Spc$ was a localization functor, inducing an equivalence between the category of $\Affl$-invariant real-étale sheaves and that of spaces. We now apply this localization to the category $\Pre(\Sm_\R)_{/\calR^\kgpd}$.

\begin{Lemma}\label{Prop:Mtop'}
The functor $M_{\calR^\kappa} : \Pre_\Sigma(\Sm_\R)_{/\calR^\kgpd} \to \Sp$ from \Cref{Prop:6FFforR} factors through the localization $L_{\Affl, \mathsf{ret}}(\Pre(\Sm_\R))_{/\calR^\kgpd}$ of $\Affl$-invariant real-étale sheaves over $\calR^\kgpd$. In particular, there is a commutative diagram of symmetric monoidal functors

\begin{equation}\label{Diag:Mtop'}\begin{tikzcd}[column sep = 4em]
	{\Pre_\Sigma(\Sm_\R)_{/\SH^\wgpd}} & {\Pre_\Sigma(\Sm_\R)_{/\mathcal{R}^\kgpd}} & {L_{\Affl, \mathsf{ret}}(\Pre(\Sm_\R))_{/\mathcal{R}^\kgpd}} & {\spaces_{/\Sp^\kgpd}} \\
	{\SH(\R)} & {\mathcal{R}(\R) \simeq \Sp}
	\arrow["{\alpha_\sharp}", from=1-1, to=1-2]
	\arrow["M"', from=1-1, to=2-1]
	\arrow["{L_{\Affl, \mathsf{ret}}}"{description}, from=1-2, to=1-3]
	\arrow["\rR", curve={height=-18pt}, from=1-2, to=1-4]
	\arrow["{M_{\calR^\kappa}}", from=1-2, to=2-2]
	\arrow["\simeq"{description}, from=1-3, to=1-4]
	\arrow[dashed, from=1-3, to=2-2]
	\arrow["{\Mtop'}", dashed, from=1-4, to=2-2]
	\arrow["\rR", from=2-1, to=2-2]
\end{tikzcd}
\end{equation}
\end{Lemma}
\begin{proof} By \cite[Prop.\ 2.11]{BEH}, to show that the functor $M_{\calR^\kappa} : \Pre_\Sigma(\Sm_\R)_{/\calR^\kgpd} \to \Sp$ factors through $L_{\Affl, \mathsf{ret}}(\Pre_\Sigma(\Sm_\R))_{/\calR^\kgpd}$, it suffices to show that $\calR^\kgpd \in L_{\Affl, \mathsf{ret}}(\Pre_\Sigma(\Sm_\R))$ is an $\Affl$-invariant real-étale sheaf. We will prove this at the end. Let us first see how this implies the statement of the lemma. Note first that we may ignore the sphericity condition, i.e., the inclusion $L_{\Affl, \mathsf{ret}}(\Pre_\Sigma(\Sm_\R)) \subseteq L_{\Affl, \mathsf{ret}}(\Pre(\Sm_\R))$ is an equivalence, since any real-étale sheaf is in particular spherical (for any $U,V\in\Sm_\R$, the inclusions of $U$ and $V$ in $U\amalg V$ form a real-étale cover of the latter). Moreover, \Cref{Prop:ret2} gives an equivalence $L_{\Affl, \mathsf{ret}}(\Pre(\Sm_\R))_{/\calR^\kgpd} \simeq \Spc_{/\rR(\calR^\kgpd)}$ induced by real realization, which for $\Affl$-invariant real-étale sheaves corresponds to taking the real sections. Then, $\rR(\calR^\kgpd) \simeq \Sp(\Spc_{/\rR(\R)})^\kgpd \simeq \Sp^\kgpd$, and we obtain Diagram (\ref{Diag:Mtop'}), where the symmetric monoidal structure on $\spaces_{/\Sp^\kgpd}$ is induced by that of $L_{\Affl, \mathsf{ret}}(\Pre(\Sm_\R))_{/\calR^\kgpd}$ (described in \Cref{Prop:otimes'smstc} below).
	
To conclude the proof, we are left with showing that $\calR^\kgpd$ is an $\Affl$-invariant real-étale sheaf. We have already proven that $\calR$ was an $\Affl$-invariant Nisnevich sheaf (\Cref{Prop:calRPBformalism}). The argument for the real-étale topology is the same, using the fact that the unstable real realization functor inverts the real-étale equivalence by \Cref{Prop:ret2}. Then, the $\Affl$-invariance of $\calR^\kgpd$ follows. We defer to \Cref{Prop:calRomegasheaf} just below the proof that $\calR^\kappa$ is a real-étale sheaf. Then, $\calR^\kgpd$ is also a real étale sheaf: it is obtained from $\calR^\kappa$ by postcomposition with the functor $(-)^\simeq : \Cat \to \Spc$, which preserves limits as it is right adjoint to the inclusion, and being a sheaf is a limit condition.
\end{proof}

\begin{Lemma}\label{Prop:calRomegasheaf} The subpresheaf of $\kappa$-compact objects $\calR^\kappa \subseteq \calR$ is a real-étale sheaf on $\Sm_\R$.
\end{Lemma}
\begin{proof}
	Recall from the proof of \Cref{Prop:Mtop'} that $\calR$ is a real-étale sheaf on $\Sm_\R$. Thus, it suffices to show that for any real-étale cover $\{f_i:U_i \to X\}$ in $\Sm_\R$, if $a\in\calR(X)$ satisfies $f_i^\ast(a) \in \calR(U_i)^\kappa$, then also $a\in\calR(X)^\kappa$. Let $\colim_{j\in\calJ} a_j$ be a $\kappa$-filtered colimit diagram in $\calR(X)$. We want to show that 
	\begin{align*} \colim_{j\in\calJ}\map_{\calR(X)}(a, a_j)\tag{$\star$} \lsimeq{} \map_{\calR(X)}(a,\colim_{j\in\calJ} a_j). \end{align*}
	We will first construct ``mapping sheaves'' and consider the analog of $(\star)$ for these. Since by definition all $f_i$'s are étale, we work over $\mathsf{Et}_X$, the small étale site of $X$. Let $a,b\in\calR(X)$. We can construct a ``mapping presheaf'' $\underline{\map}_\calR(a,b)$ sending $(f:U\to X) \in \mathsf{Et}_X$ to the space $\map_{\calR(U)}(f^\ast a, f^\ast b)$ \cite[Cor.\ 3.8]{Bachmann-stccomm}. Since $\calR$ is in particular a real-étale sheaf on $\mathsf{Et}_X$, $\underline{\map}_\calR(a,b)$ is also a sheaf. Indeed, if $\mathcal{V} \to Y$ is a real-étale cover in $\mathsf{Et}_X$, then $\calR(Y)$ is equivalent to the limit over the \v Cech complex $\lim \calR(\check{C}_\bullet(\mathcal{V}))$, so that 
	$$\underline{\map}_\calR(a,b)(Y) \simeq \map_{\lim \calR(\check{C}_\bullet(\mathcal{V}))}(a,b) \simeq \lim \map_{\calR(\check{C}_\bullet(\mathcal{V}))}( a, b) = \lim \underline{\map}_{\calR}(a,b)(\check{C}_\bullet(\mathcal{V})),$$
	where in the second-to-last term, the pullback of $a$ and $b$ to the terms of the \v Cech complex is implicit.
	There is a natural map of presheaves $\xi: \colim^{\Pre(\mathsf{Et}_X)}_{j\in\calJ} \underline{\map}_\calR(a,a_j) \to \underline{\map}_\calR(a,\colim_{j\in\calJ} a_j)$ inducing equivalences after restriction to each $U_i$. This holds because colimits are computed sectionwise in presheaves, and for any morphism $f:V \to U_i$, the restriction $f^\ast a \in \calR(V)^\kappa$ is $\kappa$-compact (by assumption $f_i^\ast a \in \calR(U_i)^\kappa$ is $\kappa$-compact, and $f^\ast$ preserves $\kappa$-compact objects by \Cref{Prop:6FFforR}). The right-hand side is already a sheaf, and thus $\xi$ induces an equivalence between the sections on $X$ of the sheafification of the left-hand side, which is equivalent to $\colim^{\mathsf{Shv}(\mathsf{Et}_X)}_{j\in\calJ} \underline{\map}_\calR(a,a_j)$, and that of the right-hand side. Therefore, we have equivalences
	\begin{align*} 
		\colim_{j\in\calJ} \map_{\calR(X)}(a,a_j) &=: \colim_{j\in\calJ} \left(\underline{\map}_\calR(a,a_j)(X)\right) \overset{(\star\star)}\simeq \left(\colim^{\mathsf{Shv}(\mathsf{Et}_X)}_{j\in\calJ} \underline{\map}_\calR(a,a_j)\right)(X) \\
		&\simeq \underline{\map}_{\calR}(a,\colim_{j\in\calJ} a_j)(X) := \map_{\calR(X)}(a,\colim_{j\in\calJ} a_j),
	\end{align*}
which gives $(\star)$. Here, $(\star\star)$ is obtained by viewing global sections as corepresented by $y(X)$, and using that the (sheafification of the) latter is a compact object in $\mathsf{Shv}(\mathsf{Et}_X)$. To see this, note that it is the terminal object in $\mathsf{Shv}(\mathsf{Et}_X) \simeq \mathsf{Shv}(RX)$ \cite[Thm.\ B.10]{ES}, so it suffices to check that the terminal object in $\mathsf{Shv}(RX)$ is compact. This follows from the fact that $RX$ is a coherent topological space \cite[Prop.\ 4.1]{ret2} and the end of the proof of \cite[Prop.\ 6.5.4.4]{Lurie-HTT}.
\end{proof} 

\begin{Lemma}\label{Prop:otimes'smstc}
	The localization $L' : \Pre(\Sm_\R)_{/\calR^\kgpd} \to L_{\Affl, \mathsf{ret}}(\Pre(\Sm_\R))_{/\calR^\kgpd}$ from \Cref{Prop:Mtop'} is symmetric monoidal, making the right-hand side into a presentably symmetric monoidal subcategory of the left-hand side via the fully faithful right adjoint of the localization $L'$.
\end{Lemma}
\begin{Notation}\label{Notation:smstc2}
	As in \Cref{Notation:smstc}, we denote the corresponding symmetric monoidal structure on $\spaces_{/\Sp^\kgpd}$ obtained via the equivalence with $L_{\Affl, \mathsf{ret}}(\Pre(\Sm_\R))_{/\calR^\kgpd}$ by $\spaces_{/\Sp^\kgpd}^{\otimes'}$, to distinguish it from the one in \Cref{Prop:smstconslice}, which we denote by $\spaces_{/\Sp^\kgpd}^{\otimes}$.
\end{Notation}
\begin{proof}
	To lighten notation, let $\calC := \Pre(\Sm_\R)_{/\calR^\kgpd}$ and $\calD :=L_{\Affl, \mathsf{ret}}(\Pre(\Sm_\R))_{/\calR^\kgpd}$. Recall that there is a source functor (or forgetful functor) $\calC\to \Pre(\Sm_\R)$ remembering only the domain of an arrow. We first use the criterion of \cite[Lem.\ 3.4]{GGN}: if we show that local equivalences are preserved by the functors $ - \otimes X$ for all $X\in \calC$, then this shows that there is an induced symmetric monoidal structure on $\calD$, such that $L'$ is symmetric monoidal and its right-adjoint $\iota$ is lax symmetric monoidal. This holds because local equivalences in $\calC$ are those maps that are equivalences of presheaves after applying the source functor and the $\Affl$-real-étale localization $L:\Pre(\Sm_\R) \to L_{\Affl, \mathsf{ret}}(\Pre(\Sm_\R))$. On the sources, the tensor product of $\calC$ is the categorical product (\Cref{Prop:tensoragree}). Since $L$ commutes with products (since real realization does), we are done.

	As an accessible localization of $\calC$ in the sense of \cite[\S 5.5.4]{Lurie-HTT}, $\calD$ is presentable. By symmetric monoidality of $L'$, for any $X,Y\in\calD$, we have $X\otimes_\calD Y = L'(\iota X \otimes_\calC \iota Y)$. Thus, to show that $\calD$ is actually a symmetric monoidal subcategory of $\calC$, it suffices to show that such tensor products $\iota X \otimes_\calC \iota Y$ are already local, and that so is the unit. The latter is clear, and for the former, we use again that being local is checked after applying the source functor; and that $\otimes_\calC$ is the categorical product on the sources. Since $L$ commutes with products, the product of local objects is local, and this is what we needed. 
	
	Finally, we show that the symmetric monoidal structure is presentable. Let $X\in\calD$, we have to show that $-\otimes_\calD X: \calD\to\calD$ preserves colimits. Let $\colim_i Y_i$ be a colimit diagram in $\calD$. Using that both $L'$ and $\iota$ preserve the tensor products, and $L'$ preserves colimits, we have:
    \begin{align*}
        \left(\colim_i Y_i\right) \otimes_\calD X &\simeq \left(\colim_i L'\iota Y_i\right) \otimes_\calD X \simeq L'\left(\colim_i \iota Y_i\right) \otimes_\calD L'\iota X \simeq L'\left( \left(\colim_i \iota Y_i\right) \otimes_\calC \iota X\right)\\
        &\simeq L'\left( \colim_i \left(\iota Y_i\otimes_\calC \iota X\right) \right) \tag{$\calC$ is presentably symmetric monoidal}\\
        &\simeq \colim_i L'\left( \iota(Y_i\otimes_\calD X)\right) \simeq \colim_i (Y_i \otimes_\calD \calG).
    \end{align*}
This finishes the proof.
\end{proof}

The functor $\Mtop'$ can be viewed as the real realization of the motivic Thom spectrum functor. We can now show that it is equivalent to the topological Thom spectrum functor. 

\begin{Prop}\label{Prop:Thomagreeassmfunctors} The symmetric monoidal structures $\spaces^{\otimes'}_{/\Sp^\kgpd}$ and $\spaces^\otimes_{/\Sp^\kgpd}$ agree (see \Cref{Notation:smstc2}). The functors $\Mtop'$ (from Diagram (\ref{Diag:Mtop'})) and $\Mtop$ (from \Cref{Prop:Mtopissm}) agree as symmetric monoidal functors.
\end{Prop}
\begin{proof} By \Cref{Prop:sliceofpshsm}, the symmetric monoidal structure $\spaces^\otimes_{/\Sp^\kgpd}$ from \Cref{Prop:smstconslice} is equivalently the Day convolution structure on $\Pre(\ast_{/\Sp^\kgpd}) \simeq \Pre(\Sp^\kgpd)$, where $\Sp^\kgpd$ is viewed as a symmetric monoidal subcategory of $\Sp$ with the usual smash product. On the other hand, the underlying category of $\spaces^{\otimes'}_{/\Sp^\kgpd}$ is the same as that of $\spaces^{\otimes}_{/\Sp^\kgpd}$, namely $\Pre(\Sp^\kgpd)$. Then, by \Cref{Prop:dayconvolutionpsh}, to show that the symmetric monoidal structure is equivalent to the Day convolution one, it suffices to show that:
\begin{enumerate}
    \item the tensor product $\otimes'$ preserves colimits in both variables,
    \item the Yoneda embedding $\Sp^\kgpd \to \Pre(\Sp^\kgpd)$ admits a symmetric monoidal structure with respect to the smash product on the left-hand side and the structure $\spaces^{\otimes'}_{/\Sp^\kgpd}$ on the right-hand side.
\end{enumerate}

Part $(1)$ was proved in \Cref{Prop:otimes'smstc}. To show $(2)$, consider the commutative diagram of fully faithful embeddings
\[\begin{tikzcd}
	{\Pre(\ast_{/\Sp^\kgpd}) \simeq \spaces_{/\Sp^\kgpd}} & {\Pre(\Sm_\R)_{/\mathcal{R}^\kgpd} \simeq \Pre\big((\Sm_\R)_{/\mathcal{R}^\kgpd}\big)} \\
	{\Sp^\kgpd \simeq \ast_{/\Sp^\kgpd}} & {(\Sm_\R)_{/\mathcal{R}^\kgpd}.}
	\arrow["\iota", hook, from=1-1, to=1-2]
	\arrow["{y_1}", hook, from=2-1, to=1-1]
	\arrow[dashed, hook, from=2-1, to=2-2]
	\arrow["{y_2}", hook, from=2-2, to=1-2]
\end{tikzcd}\]

The dashed arrow exists because the composition $\iota\circ y_1$ factors through the subcategory of representables (its image only contains arrows with source $y(\R) \in \Pre(\Sm_\R)$). We claim that $y_2$ is symmetric monoidal with respect to the monoidal structures introduced in \Cref{Prop:motiviccolim}. Indeed, the coCartesian fibration $\Pre(\Sm_\R)_{/\mathcal{R}^\kgpd}^{\otimes'} \to \Finstar$ representing the symmetric monoidal structure is classified by a composition 
$$\Finstar \xrightarrow{\ \ H\ \ } \Cat \xrightarrow{\, \Pre(-)\, }\PrL \xrightarrow{\ \ U\ \ } \widehat{\mathsf{Cat}}_\infty$$
where $H$ classifies the coCartesian fibration $(\Sm_\R)_{/\mathcal{R}^\kgpd}^{\otimes'} \longrightarrow \Finstar$ encoding the symmetric monoidal structure constructed in the proof of \Cref{Prop:motiviccolim}, and $U$ is the forgetful functor. Let $\eta$ be the unit of the adjunction $\Pre(-)\dashv U$. It induces a natural transformation
\[\begin{tikzcd}[column sep = 4em, row sep = 4em]
	\Finstar && \Cat && \widehat{\mathsf{Cat}}_\infty \\
	&&& \PrL
	\arrow["H", from=1-1, to=1-3]
	\arrow[""{name=0, anchor=center, inner sep=0}, equals, from=1-3, to=1-5]
	\arrow[""{name=1, anchor=center, inner sep=0}, "{U\circ\Pre(-)}"', curve={height=23pt}, from=1-3, to=1-5]
	\arrow["{\Pre(-)}"', from=1-3, to=2-4]
	\arrow["U"', from=2-4, to=1-5]
	\arrow["\eta", shorten <=2pt, shorten >=2pt, Rightarrow, from=0, to=1]
\end{tikzcd}\]
between the functors $\Finstar \longrightarrow \widehat{\mathsf{Cat}}_\infty$ classifying $(\Sm_\R)_{/\mathcal{R}^\kgpd}^{\otimes'}$ and $\Pre(\Sm_\R)_{/\mathcal{R}^\kgpd}^{\otimes'}$, respectively. This induces a symmetric monoidal functor between these categories, whose underlying functor is by construction the Yoneda embedding $y_2$, as claimed. By \Cref{Prop:otimes'smstc}, the embedding $\iota$ is symmetric monoidal too. Then also $y_1$ is symmetric monoidal, with respect to the structure $\Sp^{\kgpd,\otimes'}$ (induced by that of $(\Sm_\R)_{/\calR^\kgpd}$). Then, the restriction of $\Mtop'$ to $\Sp^\kgpd$ is a symmetric monoidal functor $\Sp^{\kappa,\simeq,\otimes'} \longrightarrow \Sp^\otimes$, where the symmetric monoidal structure on $\Sp$ is the usual one. But we also know that the underlying functor is the inclusion. Therefore, $\Sp^{\kappa,\simeq,\otimes'}$ is endowed with the symmetric monoidal structure induced from that of $\Sp$. This proves $(2)$. 

Thus, the categories $\spaces^{\otimes'}_{/\Sp^\kgpd}$ and $\spaces^{\otimes'}_{/\Sp^\kgpd}$ are both equivalent as symmetric monoidal categories to $\Pre(\Sp^\kgpd)$ with the Day convolution, where $\Sp^\kgpd$ is viewed as a symmetric monoidal subcategory of $\Sp$ with the usual smash product. By \Cref{Prop:dayconvolution}, to show that $\Mtop$ and $\Mtop'$ agree as symmetric monoidal functors, it then suffices to show that their restrictions to $\Sp^\kgpd$ agree as symmetric monoidal functors. In both cases, this restriction is the embedding of the symmetric monoidal subcategory $\Sp^\kgpd \xhookrightarrow{} \Sp$ (see \Cref{Def:Thomspectrum} and \Cref{Prop:Masacolimit}). This concludes the proof.
\end{proof}

Finally, combining the previous results, we can prove our main theorem. 

\begin{thm}\label{Prop:Thomagreeassmfunctorsfinal} The real realization of the motivic multiplicative Thom spectrum functor $M : \Pre_\Sigma(\Sm_\R)_{/\SH^\wgpd} \to \SH$ (as constructed in \cite[\S 16.3]{BH-norms}) is equivalent as a symmetric monoidal functor to the topological multiplicative Thom spectrum functor $\Mtop : \spaces_{/\Sp^\kgpd} \to \Sp$ (as constructed in \cite[Thm.\ 1.6]{ABG}). More precisely, there is a commutative diagram of symmetric monoidal functors
\[\begin{tikzcd}[row sep = 2.5em, column sep = 3em]
   {\Pre_\Sigma(\Sm_\R)_{/\SH^\wgpd}} & {\spaces_{/\Sp^\kgpd}} \\
	{\SH(\R)} & \Sp
	\arrow[from=1-1, to=1-2]
	\arrow["M"', from=1-1, to=2-1]
	\arrow["\Mtop", from=1-2, to=2-2]
	\arrow["{\rR}"', from=2-1, to=2-2]
    \end{tikzcd}\]
where the top horizontal arrow is also induced by $\rR$ (see Diagram (\ref{Diag:Mtop'}) in \cref{Prop:Mtop'}).
\end{thm}
\begin{proof} Combine \cref{Prop:Mtop',Prop:Thomagreeassmfunctors}.
\end{proof}

\vspace{0.2cm}

\subsection{The realizations of the algebraic cobordism ring spectra \texorpdfstring{$\MGL$}{MGL}, \texorpdfstring{$\MSL$}{MSL}, and \texorpdfstring{$\MSp$}{MSp}}\label{Subsect:MSL}\hfill\vspace{0.2cm}

As advertised, we now apply \Cref{Prop:Thomagreeassmfunctorsfinal} to three very classical examples of $\Einfty$-Thom spectra.

\begin{thm}\label{Prop:MGLMSLMSp} The equivalence $\rR \circ M \simeq \Mtop \circ \rR \circ \alpha_\sharp$ from Diagram (\ref{Diag:Mtop'}) and \Cref{Prop:Thomagreeassmfunctors} induces equivalences of $\Einfty$-rings
    $$\rR\mathsf{MGL} \simeq \mathsf{MO}, \qquad \rR\mathsf{MSL} \simeq \mathsf{MSO}, \qquad \rR\mathsf{MSp}\simeq \mathsf{MU}.$$
\end{thm}
\begin{proof}
    We saw in \Cref{Ex:MGLMSLMSp} that $\rR\mathsf{MGL} = \rR(M(K^\circ \xrightarrow{j} \SH^\wgpd))$. By Diagram (\ref{Diag:Mtop'}) and \Cref{Prop:Thomagreeassmfunctors}, this is equivalent as an $\Einfty$-ring to $\Mtop(\rR(K^\circ \to \SH^\wgpd \to \calR^\kgpd))$. Recall from \Cref{Ex:MSO} that $\mathsf{MO} = \Mtop(BO \xrightarrow{j} \Sp^\kgpd)$. So what we have to show is that $\rR(K^\circ \to \SH^\wgpd \to \calR^\kgpd) \simeq (BO \xrightarrow{j} \Sp^\kgpd)$ as commutative algebras in $\spaces^\otimes_{/\Sp^\kgpd}$, or equivalently, as $\Einfty$-maps with target $\Sp^\kgpd$ (by \Cref{Prop:algebrasintheslice}). Recall that $K^\circ \to \SH^\wgpd \to \calR^\kgpd$ is a morphism of spherical presheaves of $\Einfty$-spaces, viewed as a commutative algebra in $\Pre_\Sigma(\Sm_\R)^{\otimes'}_{/\calR^\simeq}$ using \Cref{Prop:CAlginslice}. This is how its real realization is viewed as a map of $\Einfty$-spaces. In other terms, we have to show that the real realization of the motivic $j$-homomorphism, followed by $\alpha: \SH^\simeq \to \calR^\simeq$, is the topological $j$-homomorphism as a map of $\Einfty$-spaces.
    
    This follows from the constructions of the topological and motivic $j$-homomorphisms as symmetric monoidal functors in \Cref{Ex:MSO,Ex:MGLMSLMSp}. Indeed, the various steps of the construction, starting from ($\infty$-)groupoids of vector bundles and then extending to the group completion, correspond to each other under real realization. To see this, recall that $\rR(BGL_n) \simeq BO_n$, and that $K^\circ \simeq \Lmot BGL$ by \Cref{Ex:MGLMSLMSp}, so that $\rR(K^\circ) \simeq BO$. 
	Indeed, $\rR(BGL_n) \simeq B(\rR GL_n) \simeq B(GL_n(\R)) \simeq BO_n$, where the first equivalence follows from the fact that $\rR$ commutes with colimits and finite products, in particular with geometric realizations and the bar construction; and the second one from the fact that $GL_n(\R)$ deformation retracts onto the orthogonal group $O_n(\R)$ by Gram-Schmidt orthogonalization. We then have: 
	$$\rR(BGL) \simeq \rR(\colim_{n} BGL_n) \simeq \colim_{n} \rR(BGL_n) \simeq \colim_{n} O_n \simeq BO.$$

    The proof follows in the exact same way for $\mathsf{MSL}$ and $\mathsf{MSp}$, using \Cref{Ex:MGLMSLMSp} and the fact that $\rR(KSL^\circ) = BSO$ and $\rR(KSp^\circ) = BU$ (with the same proof as in the case of $K^\circ$).
    \end{proof}

\vspace{0.4em}
\section{Complex realizations of motivic Thom spectra}\label{Subsect:rCThom}

This appendix extends the results of \Cref{Sect:comp} to the complex realization functor. The structure of the proof is essentially the same, but we need a workaround for \Cref{Prop:ret2}, which stated in particular that $\rR = \Pre(\Sm_\R) \to \Spc$ was a localization; indeed, it is unknown whether this also holds for the complex realization functor. Instead, we reduce to a purely topological question, which appears as the first part of \Cref{Prop:MandMtopagreeassmII} below. The proof provided in this appendix also works for the real realization functor $\rR$.\\

We first prove the analog of \Cref{Prop:motiviccolim} (construction of the motivic colimit functor $M_\calF$) for presheaves of symmetric monoidal categories on $\Spc^\kappa$ instead of $\Sm_S$. Recall that $\kappa$ is an uncountable regular cardinal such that $\Spc^\kappa$ is closed under finite limits in $\Spc$ (\Cref{Notation:Mtopkappa}), and that $\Spc^\kappa$ is essentially small.

\Cref{Prop:motiviccolim} involves Weil restrictions along fold maps. Here is a topological analog of this notion.

\begin{Notation}\label{Notation:topWeil} In $\Spc^\kappa$, for a fold map $\nabla: Y = \coprod_{i\leq n} Z_i^{\amalg n_i} \to \coprod_{i\leq n} Z_i = Z$, and a map $V \to Y$, let
    $$R_\nabla(V) := \coprod_{i\leq n} V_{Z_{i,1}} \times_{Z_i} \dots \times_{Z_i} V_{Z_{i,n_i}} \in (\Spc^\kappa)_{/Z},$$
where $V_{Z_{i,j}}$ is the component of $V$ above the $j$-th copy of $Z_i$ in $Y$, for $i\leq n$ and $j\leq n_i$.

As in the case of schemes, for any $W\in\Spc_{/Z}$, we then have an equivalence 
$$ \map_{(\Spc^\kappa)_{/Z}}(W,R_\nabla(V)) \simeq \map_{(\Spc^\kappa)_{/Y}}(W\times_Z Y,V).$$
\end{Notation}

\begin{thm}\label{Prop:topcolim} Let $\calF: \Span' := \Span(\Spc^\kappa,\all,\fold) \longrightarrow \Cat$ be a functor preserving finite products, such that:
\begin{enumerate}[label = (\roman*)]
    \item for any $f:Y\to X$ in $\Spc^\kappa$, $f^\ast := \calF(X\leftarrow Y \xlongequal{\ } Y) : \calF(X) \to \calF(Y)$ admits a left adjoint, denoted by $f_\sharp$.
    \item for any Cartesian square in $\Spc^\kappa$
\[\begin{tikzcd}[ampersand replacement=\&]
	{X'} \& X \\
	{Y'} \& Y,
	\arrow["{g'}", from=1-1, to=1-2]
	\arrow["{f'}"', from=1-1, to=2-1]
	\arrow["\lrcorner"{anchor=center, pos=0.125}, draw=none, from=1-1, to=2-2]
	\arrow["f", from=1-2, to=2-2]
	\arrow["g", from=2-1, to=2-2]
\end{tikzcd}\]
    the exchange transformation $\mathsf{Ex}_{\sharp}^\ast: f'_\sharp(g')^\ast \Longrightarrow g^\ast f_\sharp$ of functors $\calF(X) \to \calF(Y')$ (defined as in \Cref{Prop:6FFforSH}) is an equivalence.
    \item Given $(\nabla: Y \to Z)\in\fold$, let $\nabla_\otimes := \calF(Y\xlongequal{\ } Y \to Z) : \calF(Y) \to \calF(Z)$. Then $\nabla$ encodes the tensor product on the symmetric monoidal category $\calF(Z)$ (\Cref{Prop:spanvspsh}). For every diagram in $\Spc^\kappa$ of the form
\[\begin{tikzcd}
	W & {R_{Y/X}(W\times_X Y)\times_Z Y} & {R_{Y/X}(W\times_X Y)} \\
	X & Y & Z
	\arrow["g"', from=1-1, to=2-1]
	\arrow["{f'}"', from=1-2, to=1-1]
	\arrow["{\nabla'}", from=1-2, to=1-3]
	\arrow["{u'}", from=1-2, to=2-2]
	\arrow["\lrcorner"{anchor=center, pos=0.125}, draw=none, from=1-2, to=2-3]
	\arrow["u", from=1-3, to=2-3]
	\arrow["f"', from=2-2, to=2-1]
	\arrow["\nabla", from=2-2, to=2-3]
\end{tikzcd}\]
    with $\nabla, \nabla' \in \fold$, the distributivity transformation 
	$${\mathsf{Dis}_{\sharp\otimes} : u_\sharp\nabla'_\otimes(f')^\ast \Longrightarrow \nabla_\otimes(\pi_Y)_\sharp\pi_W^\ast}$$
	of functors $\calF(W) \to \calF(Z)$ (defined as in \Cref{Prop:6FFforSH}) is an equivalence. 
\end{enumerate}
\vspace{0.2cm}
Then, there exists a natural transformation $M_\calF: ((\Spc^\kappa)_{/\bullet})_{/{\calF}^\simeq} \to \calF$ of finite products preserving functors $\Span' \to \Cat$  (i.e., of spherical presheaves of symmetric monoidal categories). 
\end{thm}

\begin{proof}
The proof of \Cref{Prop:motiviccolim} applies in this case, up to the following comments:
\begin{enumerate}
    \item[(0)] In the construction, the category $\Fun_\Sm(\Delta^1,\Span)$ has to be replaced with $\Fun_\all(\Delta^1,\Span')$, the full subcategory of $\Fun(\Delta^1,\Span')$ spanned by edges of the form $X \leftarrow Y \longeq Y$ with $X,Y\in\Spc^\kappa$. We also have to prove that $\Spc^\kappa$ is an extensive category, so that finite products preserving functors $\Span' \to \Cat$ indeed correspond to spherical presheaves of symmetric monoidal categories on $\Spc^\kappa$ (\Cref{Prop:spanvspsh}). Note that $\kappa$-small colimits in $\Spc^\kappa$ are computed in $\Spc$, because $\Spc^\kappa$ is a full subcategory closed under such colimits in $\Spc$. In particular, $\Spc^\kappa$ admits finite coproducts, and they are disjoint (by choice of $\kappa$, pullbacks in $\Spc^\kappa$ exist and are computed in $\Spc$). The stability of finite coproduct decompositions under pullbacks follows from the same property in $\Spc$. 
    \item The proof of Step 1 goes through verbatim, if we replace $\Span$ with $\Span'$, $\Fun_\Sm(\Delta^1,\Span)$ with $\Fun_\all(\Delta^1,\Span')$, and Weil restrictions with those of \Cref{Notation:topWeil}.
    \item Step 2 follows formally from the assumptions in the same way as in \Cref{Prop:motiviccolim}.
	\item The argument of Step 3 is still valid, given the description of Cartesian edges in Step 1: they are the diagrams from assumption $(iii)$ in the statement.
    \item In Step 4, the fiber of $s^\op$ over some $X\in\Spc^\kappa$ is our situation $(\Spc^\kappa)_{/X}$, and that of $t^\ast F$ over some $(Y\to X)\in(\Spc^\kappa)_{/X}$ is $\calF(Y)$. We thus obtain a transformation $M_\calF : ((\Spc^\kappa)_{/\bullet})_{\sslash\calF} \to \calF$ which restricts to the desired transformation $M_\calF : ((\Spc^\kappa)_{/\bullet})_{/\calF^\simeq} \to \calF$.
\end{enumerate}
\end{proof}

We will now apply the construction from the previous theorem to a presheaf $\calC$ on $\Spc^\kappa$ whose precomposition with $\rR$ is the presheaf $\calR$ of \Cref{Def:calR}. Working with $\rC$ instead will allow us to prove analogous results for the complex realization of Thom spectra. 

\begin{Def}\label{Def:calC} Let $\calC$ be the presheaf of (not necessarily small) symmetric monoidal categories on $\Spc^\kappa$, defined by $X \in \Spc^\kappa \mapsto \Sp(\spaces_{/X})$, where functoriality is given by pullback (see \Cref{Prop:6FFforC}), and let $\calC^\kappa: \Span' = \Span(\Spc^\kappa, \all, \fold) \to \Cat$ be the spherical presheaf of the subcategories of $\kappa$-compact objects (again, see \Cref{Prop:6FFforC}).
\end{Def}

\begin{Prop}\label{Prop:6FFforC} The presheaves $\calC$ and $\calC^\kappa$ from \Cref{Def:calC} satisfy all the assumptions of \Cref{Prop:topcolim} (except that $\calC$ is valued in $\widehat{\mathsf{Cat}}_\infty$ instead of $\Cat$). In particular, we obtain symmetric monoidal, colimit-preserving functors
 $$M_\calC : (\Spc^\kappa)_{/\calC^\kgpd} \longrightarrow \calC(\ast)^\kappa \simeq \Sp^\kappa \quad \text{and} \quad M_\calC : \Pre(\Spc^\kappa)_{/\calC^\kgpd} \longrightarrow \Sp.$$  
\end{Prop}
\begin{proof} The assumptions of \Cref{Prop:topcolim} are satisfied because the proof of \Cref{Prop:6FFforR} goes through (we do not need Step 4), as the latter did not rely on the categories being sliced over real realizations of schemes, but only some $\kappa$-compact spaces. \Cref{Prop:topcolim} then provides a morphism of spherical presheaves $((\Spc^\kappa)_{/\bullet})_{/\calC^\kgpd} \to \calC^\kappa$. Taking the component on $\ast$, we obtain the first functor of the statement $(\Spc^\kappa)_{/\calC^\kgpd} \longrightarrow \calC(\ast)^\kappa \simeq \Sp^\kappa$. The symmetric monoidal left Kan extension (see \Cref{Def:smLKE}) of its postcomposition with the inclusion $\Sp^\kappa \to \Sp$ then yields the second functor of the statement.
\end{proof}

We may now prove the analog of \Cref{Prop:compMSHMR}, comparing the motivic Thom spectrum functor with a more topological colimit functor. 

\begin{Prop}\label{Prop:compMSHMC} Let $k$ be $\R$ or $\C$. There is a commutative diagram of symmetric monoidal functors
\[\begin{tikzcd}
	{(\Sm_k)_{/\SH^\wgpd}} & {(\Spc^\kappa)_{/\mathcal{C}^\kgpd}} \\
	{\SH(k)} & { \Sp,}
	\arrow["{\beta_k}", from=1-1, to=1-2]
	\arrow["M"', from=1-1, to=2-1]
	\arrow["{M_\mathcal{C}}", from=1-2, to=2-2]
	\arrow["{r_k}", from=2-1, to=2-2]
\end{tikzcd}\]
where $M$ is the motivic Thom spectrum functor over $k$ (as in \Cref{Prop:motivicThom}). 
\end{Prop}
\begin{proof} By \Cref{Prop:kappacompinslice}, the functor $r_k : \Sm_k \to \Spc$ factors through $\Spc^\kappa$. It preserves coproducts (and thus fold maps). Therefore, it induces a functor $r_k: \Span \to \Span'$ (see \cite[\S 5.11]{Barwick-mackey}).
	
Proceeding as in the proof of \Cref{Prop:naturalityofmotiviccolim,Prop:compMSHMR}, since $r_k^\ast\calC := \calC \circ r_k : \Span \to \Span' \to \Cat$ is simply the presheaf $\calR$ from \Cref{Def:calR} in the case $k=\R$, and a complex analog of it otherwise, we have a commutative diagram of morphisms of spherical presheaves of symmetric monoidal categories
\begin{equation}\label{Diag:SHvrkastcalC}
\begin{tikzcd}
	{(\Sm_\bullet)_{/\SH^\wgpd}} & {(\Sm_\bullet)_{/r_k^\ast\calC^\kgpd}} \\
	{\SH^\omega} & {r_k^\ast\calC^\kappa.}
	\arrow[from=1-1, to=1-2]
	\arrow["M"', from=1-1, to=2-1]
	\arrow["{M_{r_k^\ast\calC^\kappa}}", from=1-2, to=2-2]
	\arrow[from=2-1, to=2-2]
\end{tikzcd}
\end{equation}

We now compare $M_{r_k^\ast\calC^\kappa}$ with $M_{\calC^\kappa} : (\Spc^\kappa)_{/\calC^\kgpd} \to \calC^\kappa$, or more precisely, their components on $k\in\Sm_k$ and $\ast\in\Spc^\kappa$ respectively. To do so, we will construct a commutative diagram 
\begin{equation}\label{Diag:SpanvsSpan'}
\adjustbox{scale = 0.85, center}{\begin{tikzcd}
	& {t^\ast r_k^\ast\calC^\kappa} &&& {t'^\ast\calC^\kappa} \\
	&&& {r_k^\ast\calC^\kappa} &&& {\calC^\kappa} \\
	\\
	{\Span(\Fin)^\op} &&& {\Span^\op} &&& {\Span'^\op}
	\arrow["\theta"{description}, from=1-2, to=1-5]
	\arrow["{\chi_{r_k^\ast\calC^\kappa}\,\circ\,\psi_{r_k^\ast\calC^\kappa}}"{description}, from=1-2, to=2-4]
	\arrow["{s^\op \,\circ\, t^\ast r_k^\ast G}"{description, pos=0.6}, dashed, from=1-2, to=4-4]
	\arrow["{\chi'_{\calC^\kappa}\,\circ\,\psi'_{\calC^\kappa}}"{description}, from=1-5, to=2-7]
	\arrow["{s'^\op \,\circ\, t'^\ast G}"{description, pos=0.6}, dashed, from=1-5, to=4-7]
	\arrow[from=2-4, to=2-7]
	\arrow["{r_k^\ast G}", from=2-4, to=4-4]
	\arrow["\lrcorner"{anchor=center, pos=0.125}, draw=none, from=2-4, to=4-7]
	\arrow["G", from=2-7, to=4-7]
	\arrow["k"{description}, from=4-1, to=4-4]
	\arrow["{\text{pt}}"{description}, curve={height=18pt}, from=4-1, to=4-7]
	\arrow["{r_k}"{description}, from=4-4, to=4-7]
\end{tikzcd}}
\end{equation}
where, abusing notation, $G: \calC^\kappa \to \Span'^\op$ is the Cartesian fibration classifying $\calC^\kappa$, and $k : \Span(\Fin)^\op \to \Span^\op$ and $\text{pt}: \Span(\Fin)^\op \to \Span'^\op$ are induced by the functors $\Fin \to \Sm_k$ and $\Fin \to \Spc^\kappa$ sending $\{1,\dots n\}$ to $k^{\amalg n}$ and $\ast^{\amalg n}$, respectively. Moreover, the maps $s,t, \chi, \psi$ are as in the proof of \Cref{Prop:motiviccolim}, and $s',t',\chi',\psi'$ are their topological analogs considered in the proof of \Cref{Prop:topcolim} (with the indices indicating the presheaf under consideration).\\

We will additionally prove that, pulling everything back to $\Span(\Fin)^\op$, we obtain a commutative diagram of morphisms of Cartesian fibrations
\begin{equation}\label{Diag:pullbacktoFin}
\adjustbox{scale = 0.75, center}{
	\begin{tikzcd}[column sep = 1em, row sep = 1em]
	{k^\ast t^\ast r_k^\ast\calC^\kappa \simeq ((\Sm_k)_{/(r_k^\ast\calC)^\kgpd})^{\otimes'}} && {\text{pt}^\ast t'^\ast \calC^\kappa \simeq ((\Spc^\kappa)_{/\calC^\kgpd})^{\otimes'}} \\
	\\
	& {k^\ast  r_k^\ast\calC^\kappa} && {\text{pt}^\ast\calC^\kappa \simeq \Sp^{\kappa,\otimes}} \\
	&&& \begin{array}{c} \\ \end{array} \\
	&& {\Span(\Fin)^\op.}
	\arrow["{\beta_k}", from=1-1, to=1-3]
	\arrow["M"{description}, from=1-1, to=3-2]
	\arrow[draw={rgb,255:red,117;green,117;blue,117}, dashed, from=1-1, to=5-3]
	\arrow["{M_\calC}"{description}, from=1-3, to=3-4]
	\arrow[draw={rgb,255:red,117;green,117;blue,117}, dashed, from=1-3, to=5-3]
	\arrow["\simeq"{pos=0.3}, from=3-2, to=3-4]
	\arrow[draw={rgb,255:red,117;green,117;blue,117}, dashed, from=3-2, to=5-3]
	\arrow[draw={rgb,255:red,117;green,117;blue,117}, dashed, from=3-4, to=5-3]
\end{tikzcd}
} \end{equation}

Composing the corresponding commutative square of symmetric monoidal functors with the evaluation at $k \in\Sm_k$ of the square (\ref{Diag:SHvrkastcalC}), we obtain the commutative diagram from the statement.\\

Recalling the constructions in \Cref{Prop:motiviccolim} and \Cref{Prop:topcolim}, we construct Diagram (\ref{Diag:SpanvsSpan'}) as follows:\\
\adjustbox{scale = 0.65, center}{
	\begin{tikzcd}[column sep = 2em, row sep=2.5em, font=\large, labels={font=\normalsize}]
    &&& {t'^\ast\calC^\kappa} && {s'^\ast\calC^\kappa} &&&& {\calC^\kappa} \\
	\\
	{t^\ast r_k^\ast\calC^\kappa} &&&& {s^\ast r_k^\ast\calC^\kappa} &&& {r_k^\ast\calC^\kappa} \\
	&&&& \begin{array}{c} \\ \end{array} & {\Fun_\all(\Delta^1,\Span')^\op} &&&& {\Span'^\op} \\
	\\
	&&&& {\Fun_\Sm(\Delta^1,\Span)^\op} &&& {\Span^\op}
	\arrow[""{name=0, anchor=center, inner sep=0}, "{\psi'_{\calC^\kappa}}", shift left=2, from=1-4, to=1-6]
	\arrow["{t'^\ast G}"{description, pos=0.2}, dashed, from=1-4, to=4-6]
	\arrow[""{name=1, anchor=center, inner sep=0}, "{\phi'_{\calC^\kappa}}", shift left=2, from=1-6, to=1-4]
	\arrow["{\chi'_{\calC^\kappa}}"{description, pos=0.8}, from=1-6, to=1-10]
	\arrow["{s'^\ast G}"{description}, from=1-6, to=4-6]
	\arrow["\lrcorner"{anchor=center, pos=0.125}, draw=none, from=1-6, to=4-10]
	\arrow["G"{description}, from=1-10, to=4-10]
	\arrow["\theta"{description}, from=3-1, to=1-4]
	\arrow[""{name=2, anchor=center, inner sep=0}, "{\psi_{r_k^\ast\calC^\kappa}}", shift left=2, from=3-1, to=3-5]
	\arrow["{t^\ast r_k^\ast G}"{description, pos=0.2}, from=3-1, to=6-5]
	\arrow["\nu"{description}, from=3-5, to=1-6]
	\arrow[""{name=3, anchor=center, inner sep=0}, "{\phi_{r_k^\ast\calC^\kappa}}", shift left=2, from=3-5, to=3-1]
	\arrow["{\chi_{r_k^\ast\calC^\kappa}}"{description, pos=0.8}, from=3-5, to=3-8]
	\arrow["{s^\ast r_k^\ast G}"{description}, from=3-5, to=6-5]
	\arrow["\lrcorner"{anchor=center, pos=0.125}, draw=none, from=3-5, to=6-8]
	\arrow[from=3-8, to=1-10]
	\arrow["\lrcorner"{anchor=center, pos=0.125, rotate=0}, draw=none, from=3-8, to=4-10]
	\arrow["{r_k^\ast G}"{description}, from=3-8, to=6-8]
	\arrow["{s'^\op}"{description, pos=0.2}, dashed, from=4-6, to=4-10]
	\arrow["{r_k \circ -}"{description}, from=6-5, to=4-6]
	\arrow["{s^\op}"{description, pos=0.2}, from=6-5, to=6-8]
	\arrow["{r_k}"{description}, from=6-8, to=4-10]
	\arrow["\dashv"{anchor=center, rotate=-90}, draw=none, from=0, to=1]
	\arrow["\dashv"{anchor=center, rotate=-90}, draw=none, from=2, to=3]
\end{tikzcd}
}\\

Here, $\nu$ and $\theta$ are obtained by the universal property of the pullbacks $s'^\ast\calC^\kappa$ and $t'^\ast\calC^\kappa$, respectively. In particular, the cube on the right-hand side commutes, and so do the squares 
$\phi'_{\calC^\kappa} \circ \nu \simeq \theta \circ \phi_{r_k^\ast\calC^\kappa}$ and $(r_k\circ -) \circ t^\ast r_k^\ast G \simeq t'^\ast G \circ \theta$ on the left-hand side. It remains to check the commutativity of the top left square involving the left adjoints, namely, that $\nu \circ \psi_{r_k^\ast\calC^\kappa} \simeq \psi'_{\calC^\kappa} \circ \theta$. To do so, we apply the same strategy as in \Cref{Prop:naturalityofmotiviccolim}: since the square with the right adjoints commutes, it suffices to show that the corresponding exchange transformation 
$$\psi'_{\calC^\kappa} \circ \theta \Longrightarrow \psi'_{\calC^\kappa} \circ \theta \circ \phi_{r_k^\ast\calC^\kappa} \circ \psi_{r_k^\ast\calC^\kappa} \simeq \psi'_{\calC^\kappa} \circ \phi'_{\calC^\kappa} \circ \nu \circ \psi_{r_k^\ast\calC^\kappa} \Longrightarrow \nu \circ \psi_{r_k^\ast\calC^\kappa}$$
is an equivalence. Again, this is easily checked using the explicit description of the edges in $t^\ast r_k^\ast \calC^\kappa$ as in the proof of \Cref{Prop:topcolim} (and thus as in the proof of \Cref{Prop:naturalityofmotiviccolim}).\\

To obtain Diagram (\ref{Diag:pullbacktoFin}), we are left to show that the map $\beta_k: k^\ast t^\ast r_k^\ast\calC^\kappa \to \text{pt}^\ast t'^\ast \calC^\kappa$ induced by the universal property of the pullback preserves Cartesian edges. Note first that an edge in $k^\ast t^\ast r_k^\ast \calC^\kappa$ is Cartesian over $\Span(\Fin)^\op$ if and only if its projection on $t^\ast r_k^\ast \calC^\kappa$ is, and similarly for Cartesian edges in $\text{pt}^\ast t'^\ast \calC^\kappa$ (see \cite[\href{https://kerodon.net/tag/01UF}{Tag 01UF}]{Kerodon}). One concludes using the explicit form of such Cartesian edges (as in the proof of \Cref{Prop:topcolim}, i.e., as in Step 3 of the proof of \Cref{Prop:motiviccolim}), and that $r_k$ preserves those.
\end{proof}

Finally, we prove the analogs of \Cref{Prop:Mtop',Prop:Thomagreeassmfunctors,Prop:Thomagreeassmfunctorsfinal}.

\begin{thm}\label{Prop:MandMtopagreeassmII}
There is a commutative diagram of symmetric monoidal functors 
\[\begin{tikzcd}
	{\Pre(\Spc^\kappa)_{/\calC^\kgpd}} & {\Spc_{/\Sp^\kgpd}} \\
	\Sp
	\arrow["{\mathsf{ev}_\ast}", from=1-1, to=1-2]
	\arrow["M_\calC"', from=1-1, to=2-1]
	\arrow["\Mtop", from=1-2, to=2-1]
\end{tikzcd}\]
where $M_\calC$ is the functor from \Cref{Prop:6FFforC}, and the top horizontal arrow is a symmetric monoidal localization (with respect to the symmetric monoidal structure of \Cref{Prop:topcolim} on the source, and the Day convolution on the target, see \Cref{Prop:sliceofpshsm}). In particular, composing with the square of \Cref{Prop:compMSHMC}, for $k=\R$ or $\C$, we obtain a commutative diagram of symmetric monoidal functors
\[\begin{tikzcd}[column sep = 3em]
	{\Pre(\Sm_k)_{/\SH^\wgpd}} & {\spaces_{/\Sp^\kgpd}} \\
	{\SH(k)} & \Sp.
	\arrow["{\mathsf{ev}_\ast\circ\beta_k}", from=1-1, to=1-2]
	\arrow["M"', from=1-1, to=2-1]
	\arrow["\Mtop", from=1-2, to=2-2]
	\arrow["{r_k}"', from=2-1, to=2-2]
\end{tikzcd}\]
\end{thm}
\begin{proof} We reproduce and summarize the proofs of \Cref{Prop:Mtop'} and \Cref{Prop:Thomagreeassmfunctors}. We first show that 
	\begin{itemize}[leftmargin = 3em]
	\item[(1)] the functor $\mathsf{ev}_\ast : \Pre(\Spc^\kappa) \to \Spc$ given by evaluation at $\ast \in \Spc^\kappa$ is a localization, i.e., that it has a fully faithful right adjoint $R$.
	\end{itemize}
	By \cite[Prop.\ 2.11]{BEH} (whose proof is still valid in our case, i.e., with $\Sm_k$ replaced with $\Spc^\kappa$), if we show that
	\begin{itemize}[leftmargin = 3em]
	\item[(2)] $\calC^\kgpd$ is a local object, namely, $(R\circ \mathsf{ev}_\ast)(\calC^\kgpd) \simeq \calC^\kgpd$ (via the unit of the adjunction $\mathsf{ev}_\ast \dashv R$),
	\end{itemize}
	then we obtain an induced localization $L : \Pre(\Spc^\kappa)_{/\calC^\kgpd} \to \Spc_{/\mathsf{ev}_\ast( \calC^\kgpd)} \simeq \Spc_{/\Sp^\kgpd}$, such that local objects are those whose image by the source functor $\Pre(\Spc^\kappa)_{/\calC^\kgpd} \to \Pre(\Spc^\kappa)$ is local, and similarly for local equivalences. As in \Cref{Prop:otimes'smstc}, we show that the right adjoint of $L$, which we denote by $R'$, exhibits $\Spc_{/\Sp^\kgpd}$ as a presentably symmetric monoidal subcategory of $\Pre(\Spc^\kappa)_{/\calC^\kgpd}$, and that $L$ is symmetric monoidal. To do so, we prove that
	\begin{itemize}[leftmargin = 3em]
	\item[(3)] for any $X\in \Pre(\Spc^\kappa)_{/\calC^\kgpd}$, the functor $- \otimes X$ preserves local equivalences in $\Pre(\Spc^\kappa)_{/\calC^\kgpd}$,
	\item[(4)] for all $X,Y\in \Spc_{/\Sp^\kgpd}$, $R'X \otimes R'Y$ is a local object, and so is the monoidal unit.
	\end{itemize}
	As in \Cref{Prop:Mtop'}, we get an induced symmetric monoidal, colimit preserving functor $\Mtop': \Spc_{/\Sp^\kgpd} \to \Sp$, and then the same argument as in \Cref{Prop:Thomagreeassmfunctors} proves that the symmetric monoidal structure on the source is given by Day convolution, and that $\Mtop'$ agrees with the classical Thom spectrum functor $\Mtop$.\\
	
	\textbf{Proof of $(1)$:} Note that a right adjoint $R$ exists and is by definition given by right Kan extension (viewing a space as a presheaf on $\{\ast\} \subseteq \Spc^\kappa$). The pointwise formula for the right Kan extension shows that $R$ associates to a space $X$ the restriction to $\Spc^\kappa$ of its Yoneda embedding $y(X) = \map_\Spc(-,X)$. Now, $R$ is fully faithful because the counit is an equivalence $\mathsf{ev}_\ast \circ R \simeq \id{\Spc}$. Indeed, for all $Y\in\Spc^\kappa$, we have $\mathsf{ev}_\ast(R(Y)) \simeq \map_\Spc(\ast,Y) \simeq Y$. \\

    \textbf{Proof of $(2)$:} We compute, for all $Y\in\Spc^\kappa$:
$$[(R\circ \mathsf{ev}_\ast)(\calC^\kgpd)](Y) \simeq \map_{\Spc}(Y,\Sp^\kgpd) \simeq \limit^\Spc_Y\,\map_{\Spc}(\ast,\Sp^\kgpd) \\
	\simeq \limit^\Spc_Y\,\Sp^\kgpd \simeq \left(\limit^{\Cat}_{Y}\,\Sp^\kappa\right)^\simeq,$$
since $(-)^\simeq: \Cat \to \Spc$ is a right adjoint. Therefore, we have to show that $\limit^{\Cat}_{Y} \Sp^\kappa \simeq \Sp(\Spc_{/Y})^\kappa$. Let $U: \PrL \to \widehat{\mathsf{Cat}}_\infty$ be the forgetful functor. We first compute
\begin{align*}
\limit^{\widehat{\mathsf{Cat}}_\infty}_{Y} \Sp &\simeq U\colim^{\PrL}_Y\, \Sp \tag{by \cite[Cor.\ 5.5.3.4 and Thm.\ 5.5.3.18]{Lurie-HTT}} \\
&\simeq U\left((\colim^{\PrL}_Y\, \Spc)\otimes \Sp\right) \tag{$\PrL$ is presentably symmetric monoidal \cite[Rmk.\ 4.8.1.18]{Lurie-HA}}\\
&\simeq \Sp(U(\colim^{\PrL}_Y\, \Spc)) \tag{tensoring with $\Sp\in\PrL$ is stabilization \cite[Ex.\ 4.8.1.23]{Lurie-HA}}\\
&\simeq \Sp(\limit^{\widehat{\mathsf{Cat}}_\infty}_Y\, \Spc_{/\ast})\\
&\simeq \Sp(\Spc_{/\colim^{\Spc}_Y\, \ast}) \tag{by descent for $\Spc$}\\
&\simeq \Sp(\Spc_{/Y}).
\end{align*}

By \Cref{Prop:6FFforC}, the functors $y^\ast: \calC(Y) \to \calC(\ast)$, for $y:\ast \to Y$, preserve $\kappa$-compact objects. Therefore, the above equivalence, which we denote by $G$, fits into a commutative diagram
\[\begin{tikzcd}[row sep = 3em, column sep = 5em]
	{\Sp(\Spc_{/Y})\simeq \calC(Y)} && {\limit_{y:\ast\to Y} \calC(\{y\}) \simeq \limit^{\widehat{\mathsf{Cat}}_\infty}_{Y} \Sp} \\
	{\calC(Y)^\kappa} && {\limit_{y:\ast\to Y} \calC(\{y\})^\kappa}
	\arrow[""{name=0, anchor=center, inner sep=0}, "{F :=\limit_{y:\ast\to Y} y^\ast}", shift left=2, from=1-1, to=1-3]
	\arrow[""{name=1, anchor=center, inner sep=0}, "G", shift left=2, from=1-3, to=1-1]
	\arrow["A", hook, from=2-1, to=1-1]
	\arrow["{F' :=\limit_{y:\ast\to Y}\ y^\ast\vert_{\calC(Y)^\kappa}}"', from=2-1, to=2-3]
	\arrow["{B = \limit_{y:\ast \to Y} B_y}"', from=2-3, to=1-3]
	\arrow["\simeq"{description}, draw=none, from=0, to=1]
\end{tikzcd}\]
where $B_y: \calC(\{y\})^\kappa \xhookrightarrow{} \calC(\{y\})$ is the inclusion. Therefore, it suffices to show that $GB$ factors through $A$ as a functor $G'$. Indeed, we then have $G'F' \simeq \id{}$ (as $AG'F'\simeq GBF' \simeq  GFA \simeq A$), and $F'G'\simeq \id{}$, as can be seen by computing the projections $\pi_y$ on each component of the limit: we have $\pi_y F'G' \simeq \pi_y$ because
$$B_y \pi_y F'G' \simeq B_y \circ y^\ast\vert_{\calC(X)^\kappa} \circ G' \simeq y^\ast A  G' \simeq y^\ast  G  B \simeq \pi_y  FG  B \simeq \pi_y  B \simeq B_y\pi_y.$$

Let $c\in\calC(Y)$ be in the image of $GB$, then each $y^\ast (c)$ is in the image of $\pi_yFGB \simeq B_y\pi_y$, so it is $\kappa$-compact in $\calC(\{y\})$. Then, for any $\colim_{j\in\calJ}\, z_j$ a $\kappa$-filtered colimit diagram in $\calC(Y)$, we have, using that $\calC(Y) \simeq \limit_{y:\ast \to Y}\, \calC(\{y\})$ by the above:
\begin{align*}
 \map_{\calC(Y)}(c, \colim_{j\in\calJ}\, z_j) &\simeq \limit_{y:\ast \to Y}\, \map_{\calC(\{y\})} (y^\ast c, y^\ast \colim_{j\in\calJ}\, z_j) \\
&\simeq \limit_{y:\ast \to Y}\, \map_{\calC(\{y\})} (y^\ast c, \colim_{j\in\calJ}\,y^\ast z_j) \tag{by \Cref{Prop:6FFforC}}\\
&\simeq \limit_{y:\ast \to Y}\, \colim_{j\in\calJ}\, \map_{\calC(\{y\})} (y^\ast c, y^\ast z_j) \tag{by $\kappa$-compactness of $y^\ast c$}\\
&\simeq \colim_{j\in\calJ}\, \limit_{y:\ast \to Y}\, \map_{\calC(\{y\})} (y^\ast c, y^\ast z_j) \tag{by \cite[Prop.\ 5.3.3.3]{Lurie-HTT}}\\
&\simeq \colim_{j\in\calJ}\, \map_{\calC(Y)}(c, z_j), 
\end{align*} 
so $c\in\calC(Y)$ is $\kappa$-compact, i.e., it is in the image of $A$. This proves $(2)$.\\

\textbf{Proof of $(3)$:} Recall that local equivalences are preserved and detected by the source functor $\Pre(\Spc^\kappa)_{/\calC^\kgpd} \to \Pre(\Spc^\kappa)$. By the same argument as in \Cref{Prop:tensoragree}, the tensor product in $\Pre(\Spc^\kappa)_{/\calC^\kgpd}$ with the symmetric monoidal structure from \Cref{Prop:topcolim} is given on the sources by Day convolution, which is simply the product of the presheaves, since the monoidal structure on $\Spc^\kappa$ we are considering is the Cartesian one. For $(X\to \calC^\kgpd)\in \Pre(\Spc^\kappa)_{/\calC^\kgpd}$ and $f: (A \to \calC^\kgpd) \to (B \to \calC^\kgpd)$ a local equivalence, $\mathsf{ev}_\ast(A) \to \mathsf{ev}_\ast(B)$ is an equivalence, and so $\mathsf{ev}_\ast(A \times X) \simeq \mathsf{ev}_\ast(A) \times \mathsf{ev}_\ast(X)  \to \mathsf{ev}_\ast(B) \times \mathsf{ev}_\ast(X) \simeq \mathsf{ev}_\ast(B \times X)$ is an equivalence as well. Therefore, $(X\to \calC^\kgpd) \otimes f$ is a local equivalence, as needed.\\

\textbf{Proof of $(4)$:} It again suffices to show the statement on the sources. The monoidal unit is local because $R(\mathsf{ev}_\ast(\ast)) = R(\ast) = y(\ast)\vert_{\Spc^\kappa} = \ast$. Moreover, given $X, Y$ as in $(4)$, with sources $x,y \in \Pre(\Spc^\kappa)$ respectively, the source of $R'X \otimes R'Y$ is $Rx\times Ry \simeq R(x\times y)$, which is local by definition.
\end{proof}

We now apply our comparison result to the standard examples of motivic Thom spectra.
\begin{thm}\label{Prop:MGLMSLMSpcomplex}
	The commutative diagram of \Cref{Prop:MandMtopagreeassmII} induces equivalences of $\Einfty$-rings 
	 $$\rC\mathsf{MGL} \simeq \mathsf{MU}, \qquad \rC\mathsf{MSL} \simeq \mathsf{MSU}, \qquad \rC\mathsf{MSp}\simeq \mathsf{MSp^{top}}.$$
\end{thm}
\begin{proof}
The proof is similar to that of \Cref{Prop:MGLMSLMSp}, using the construction of $\beta_k$ given in the proof of \Cref{Prop:compMSHMC}.
\end{proof}

\newpage
\section*{Index of notation}\label{Sect:indexofnotation}

\vspace{-1em}
\begin{table}[H]
  \centerline{%
    \begin{tabular}[h]{lll}
 $\Affl$ & Affine line over some base scheme & \\
 $\Alg_{\E{n}}(\calC^\otimes)$, $\CAlg(\calC^\otimes)$ & Categories of $\E{1}$-, resp.\ $\Einfty$-algebras in a symmetric monoidal category $\calC^\otimes$ & \S \ref{Subsect:dayconvolution} \\
 $\calC^\simeq$ & Maximal ($\infty$-)groupoid in a category $\calC$ (right adjoint to the forgetful functor) & \ref{Def:Thomspectrum} \\
 $\calC_{/X}$ & Slice category for $\calC$ a category and $X \in \calC$  & \S \ref{Subsect:slices} \\
 $\Cat$ & Category of small ($\infty$-)categories \\
 $\E{n}$ & $\infty$-operad of little $n$-cubes & \S \ref{Subsect:slices}\\
 $\Fin$ & 1-category of finite sets (including $\emptyset$) & \ref{Def:finstar}\\
 $\mathsf{Fun^\times}(\mathcal{C},\mathcal{D})$ & Category of functors $\mathcal{C}\to \mathcal{D}$ preserving finite products & \ref{Prop:presheavesandspans}\\
 $\mathsf{Fun^L}(\mathcal{C},\mathcal{D})$ & Category of colimit-preserving functors $\mathcal{C}\to \mathcal{D}$ & \ref{Prop:laxdayconvolution}\\
 $\mathsf{Fun^{\otimes}}(\mathcal{C}^\otimes,\mathcal{D}^\otimes)$ & Category of symmetric monoidal functors $\mathcal{C}^\otimes\to \mathcal{D}^\otimes$ & \ref{Prop:dayconvolution}\\
 $\mathsf{Fun^{lax}}(\mathcal{C}^\otimes,\mathcal{D}^\otimes)$ & Category of lax symmetric monoidal functors $\mathcal{C}^\otimes \to \mathcal{D}^\otimes$ & \ref{Prop:maindayconvolution}\\
 $\Gm$ & Group scheme $S\times_{\Spec(\Z)} \Spec(\Z[t,t^{-1}])$ over a base scheme $S$ & \\
 $\kappa$ & Some chosen uncountable cardinal, see \Cref{Notation:Mtopkappa} & \ref{Notation:Mtopkappa}\\
 $M$ & Motivic Thom spectrum functor & \ref{Prop:motivicThom}\\
 $\Mtop$ & Topological Thom spectrum functor & \ref{Def:Thomspectrum}\\
 $\mathsf{Map}_\calC$ & Mapping space in a category $\calC$ & \\
 $\mathsf{MGL}, \MSL, \mathsf{MSp}$ & General linear, special linear, and symplectic cobordism motivic Thom spectra & \ref{Ex:MGLMSLMSp}\\
 $\mathsf{MO}, \MSO, \mathsf{MU}$ & Unoriented, oriented, and complex cobordism Thom spectra & \ref{Ex:MSO}\\
 $\finstar{n}$, $\fin{n}$ & Finite pointed set $\{\ast, 1,\dots,n\}$ in $\Finstar$, resp.\ finite set $\{1,\dots,n\}$ in $\Fin$ & \ref{Def:finstar}\\
 $\calC^\omega$ & Full subcategory spanned by compact objects in a category $\calC$ & \ref{Def:Thomspectrum}\\
 $\Pre(\calC)$ & Category of presheaves of spaces on a category $\calC$ & \\ 
 $\Pre_\Sigma(\calC)$ & Category of spherical presheaves of spaces on a category $\calC$ & \ref{Def:sphericalpresheaves}\\ 
 $\PrL$ & Category of presentable categories and left-adjoint functors & \\
 $\rR$, $\rC$ & Real and complex Betti realization functors & \S \ref{Subsect:rR}\\
 $\spaces$ & Category of spaces & \\
 $\calS$ & Motivic sphere spectrum & \\
 $\calS^n$ & (Simplicial) $n$-sphere in the category of motivic spaces or motivic spectra & \\
 $\Sph$ & Topological sphere spectrum & \\
 $\Sph^n$ & $n$-sphere in the category of spaces or spectra & \\
 $\SH(S)$ & Category of motivic ($\mathbb{P}^1$-)spectra over $S$ & \\
 $\Sigma^\infty$ & Infinite suspension functors for topological or motivic spectra & \\
 $\Sm_S$ & 1-category of smooth schemes of finite type over the base scheme $S$ & \\
 $\Sp$ & Category of (topological) spectra & \\
 $\Span$ & When working over a base scheme $S$, $\Span(\Sm_S, \mathsf{all}, \mathsf{fold})$ & \ref{Notation:Span} \\
 $\Span(\calC,\mathsf{left},\mathsf{right})$ & Category of objects in $\calC$, with morphisms certain span diagrams & \ref{Def:Span} \\
 $\Spc(S)$ & Category of motivic spaces over the base scheme $S$ &
 \end{tabular}%
  }
\end{table}

\hfill\\

\paragraph{\textbf{Acknowledgments: }} This article was born as a part of my Master's thesis, about the real realization of another motivic spectrum, namely $\ko$, the very effective cover of the Hermitian $K$-theory spectrum, because the proof required identifying the real realization of $\MSL$ as an $\mathcal{E}_2$-ring. I want to express my deepest gratitude to my supervisor Prof.\ Tom Bachmann for suggesting this question, which turned out so rich in learnings for me, and his continued support throughout. My gratitude also goes to my second supervisor Prof.\ Jérôme Scherer at EPFL, and to Anton Engelmann, Klaus Mattis, Luca Passolunghi, and professors Marc Hoyois and Markus Land for very helpful discussions. I also want to thank two anonymous referees and Qi Zhu, for useful comments about an earlier version of this article. Finally, I acknowledge financial support by the EPFL Student Support Program, and by the Deutsche Forschungsgemeinschaft (DFG, German Research Foundation) through the Collaborative Research Center TRR 326 Geometry and Arithmetic of Uniformized Structures, project number 444845124.

\newpage
\bibliographystyle{alpha}  
\bibliography{biblio.bib} 
\end{document}